\documentclass[onecolumn,11pt]{IEEEtran}
\usepackage{amsmath,amssymb,graphicx,latexsym}
\usepackage{mathrsfs}
\usepackage{enumerate}
\usepackage{bbm}

\newtheorem{theorem}{Theorem}
\newtheorem{remark}{Remark}
\newtheorem{corollary}{Corollary}
\newtheorem{lemma}{Lemma}

\def\R{{\mathbb{R}}}
\def\P{{\mathbb{P}}}
\def\E{{\mathbb{E}}}

\def\F{{\mathcal{F}}}
\def\C{{\mathcal{C}}}
\def\B{{\mathcal{B}}}
\def\N{{\mathfrak{N}}}

\def\X{{\mathcal{X}}}
\def\I{{\mathcal{I}}}
\def\eqd{\,{\buildrel d \over =}\,}

\def\nll{{a}}
\def\nuu{{b}}
\def\xll{{c}}
\def\xuu{{d}}
\def\arrowd{\,{\buildrel d \over \rightarrow}\,}
\def\arrowwh{\,{\buildrel w_{H^1} \over \rightarrow}\,}
\def\arrowwl2{\,{\buildrel w_{L^2} \over \rightarrow}\,}
\def\arrowl2{\,{\buildrel L^2 \over \rightarrow}\,}

\DeclareMathOperator{\am}{argmax}

\DeclareMathOperator{\nus}{nats}

\ifCLASSINFOpdf
\else
\fi
\begin{document}
%
\title{On the Maximum Entropy of a Sum with Constraints and Channel Capacity Applications}
%
%
%

\author{Francisco~J.~Piera$^\dag$
\thanks{$^\dag$Ph.D., Purdue University.}
}

\maketitle

\begin{abstract}
We study the maximum achievable differential entropy at the output of a system assigning to each input $X$ the sum $X+N$, with $N$ a given noise with probability law absolutely continuous with respect to the Lebesgue measure and where the input and the noise are allowed to be dependent. We consider fairly general average cost constraints in the input, as well as amplitude constraints. Under appropriate mild regularity requirements on the noise density, it is shown that the corresponding search for the optimum may be performed over joint distributions for the input and the noise concentrated in lower dimensional geometrical objects represented by graphs of sufficiently regular functions in the associated noise-input plane, and a full characterization for the optimal curve is provided in terms of the equation and boundary conditions it satisfies. The results are then applied to correspondingly characterize the independent input and noise case, so providing bounds for channel capacity. Analysis of achievable bounds and associated capacity-achieving input distributions is also provided. Several examples including commonly considered noise distributions, cost functions and amplitude ranges are worked out in detail. General conditions for the existence of a maximum achievable entropy at the output are provided for such a system, as well as for the independent input and noise case, so establishing conditions for capacity to be achievable. The laws of such capacity-achieving input distributions are also uniquely specified. A general approximation to capacity scheme is also provided, showing that in the not achievable case the corresponding noise can be written as the limit of a sequence of noise random variables with associated channels being capacity achievable, and whose capacities converge to the capacity of the original channel. The relationship between our results for this input and noise dependent setting and the notion of feedback capacity are also explored.

\end{abstract}


%
\IEEEpeerreviewmaketitle

\section{Introduction}
\label{intro}

\IEEEPARstart{P}{ractical} communication systems cannot usually
transmit arbitrary channel inputs, commonly due to various physical restrictions imposed. An illustrating example in this context is
given by considering physical amplifiers, which present a
limitation on the maximum power or signal amplitudes they can tolerate.
This and other phenomena naturally results in studying the
capacity of a communication channel under various type of input
constraints, most commonly amplitude and power constraints. For example, the problem of finding capacity and characterizing capacity-achieving
distributions in additive noise channels with amplitude and/or average cost
constraints has been considered in the literature
in several contexts, both in the Gaussian and non-Gaussian noise case; see for example
\cite{AD2000,TCSHFK2003,AT2004,JB1979} and references therein. The
problem of computing and characterizing the capacity-achieving
distribution in additive noise channels with inputs constrained to lie in a given fixed discrete signaling set
is considered for example in
\cite{AFRM2007}. Other types of channels with input constraints have also been considered in the literature, such as conditionally Gaussian channels and Rayleigh-fading channels; see for example \cite{TCSHFK2005,IAMTSS2001,WZJL2006} and references therein. The case of channels with partially unknown noise has been considered in \cite{CBIF1996}.

The problem of capacity and capacity-achieving distribution computation is in general very hard to solve in additive noise channels with arbitrary noise distributions, being in fact rarely solvable explicitly. This has motivated the consideration of other related concepts such as for example the notion of \textit{capacity per unit cost}, allowing for explicit formulae in some cases with cost constrained input channels, as for instance additive and multiplicative Gaussian noise channels and additive non-Gaussian noise channels, where no explicit computation of channel capacity is possible; see \cite{SV1990}.

As widely known, the problem of computing capacity in an additive noise channel with input constraints accounts for finding the maximum possible entropy at its output, with the maximization being carried over the set of all input probability distributions representing inputs satisfying the imposed constraints and being stochastically independent of the noise. This maximization is a particular case of a maximization taken over all joint probability distributions for the input and the noise, in the sense that the last one provides an upper bound for the former by allowing for the input and the noise to be dependent, with of course the input marginal distribution representing an input satisfying the imposed constraints, and the noise marginal distribution being kept fixed to the given one. The problem of maximizing the output entropy under no independence conditions between the input and the noise has not received too much attention in the literature though. Along this line, one interesting related work can be found in \cite{TCZZ1994}, where the authors consider the problem of finding the maximum entropy of a sum of two dependent random variables subject to the constraint of both marginal densities being equal to a fixed given one.

Needless to say, the maximization of the corresponding entropy at the output without independence restrictions between the input and the noise, and satisfying of course the imposed constraints, confront us with a maximization over a much larger searching set for the optimum than in the former independent case, the search being now performed over all joint probability distributions for the input and the noise and not only over those factorizing as the product of their one-dimensional marginals. However, we will show in the paper that under appropriate mild regularity requirements on the noise density, a density that it is assumed to exist, the maximization of the entropy at the output in the jointly distributed input and noise case reduces to performing the corresponding optimal search over joint distributions concentrated in a lower dimensional region, specifically a curve expressible as the graph of a sufficiently regular function relating the input and the noise. We will characterize such a relationship, giving the corresponding equations that the associated function satisfies and providing conditions for its existence. This will lead in particular to corresponding bounds for the maximum entropy at the output in the independent input and noise case, so providing bounds for channel capacity. Conditions for attainable bounds will also be given, as well as the corresponding analysis for the associated capacity-achieving input distributions.

We consider general average cost constraints and amplitude constraints on the input. Our proof techniques for the provided maximum entropy results will mainly rely on tools from the calculus of variations and functional analysis. As widely known, a direct approach to channel capacity computation via calculus of variations is in general difficult, among other reasons because of the superposition of the input and noise distributions that results in a convolution in the independent case. In addition to provide results for the important case where dependency is allowed, our approach will serve to characterize the independent case by precisely avoiding those sort of expressions.

In addition to providing a full characterization for the optimum entropy at the output, both in the dependent and independent input and noise cases, we will also provide general conditions for the optimums to exist, further identifying in the independent case, conditions for the optimum to be achievable, that is, conditions for the existence of a corresponding capacity-achieving input distribution. Such distributions will be fully identified when they exist. In passing, a bound on how much entropy at the output may be gained when going from the independent input and noise case to the fully dependent one will also be provided, so providing a bound for the gain of entropy at the output by the use of feedback. Along this line, the relationship between our results for the input and noise dependent setting and the notion of feedback capacity will also be explored.

The rigorous mathematical setting and a precise discussion of the previous statements are presented in the next sections. Though similar techniques can in principle be used as a general approach to a multi-dimensional setting, for sake of simplicity and clarity in the exposition we will consider a one-dimensional one, where the main ideas can be presented without being obscured by pure technical issues.

The organization of the paper is as follows. In Section \ref{setting} we introduce the general problem to be considered, namely the problem of characterizing the maximum achievable entropy at the output of a system assigning to each input random variable the sum of such input plus a given noise random variable, the input and the noise random variables being allowed to be dependent and the input satisfying, in general, both amplitude and average cost constraints. In Section \ref{main} we provide a full characterization of such an optimum, identifying all the equations and boundary conditions it satisfies when it exists. In Section \ref{main2} we then consider the application of the results to the independent input and noise case, so characterizing the associated channel capacity and the existence of a capacity-achieving input distribution. In Section \ref{examples} we work out several examples of interest illustrating the usefulness of our results characterizing the optimum. In Section \ref{existence} we provide general conditions for the previous optimums to exist, so identifying conditions for the maximum entropy at the output to be attainable in both the dependent and independent signal-noise cases, and also fully characterizing the corresponding capacity-achieving input distributions in the independent case. Finally, in Section \ref{generalcapacity} we consider the relationship between our results for the input and noise dependent setting and the notion of feedback capacity.

\section{General Setting and Some Notation}
\label{setting}

Let $(\Omega,\F,\P)$ be a given probability space. For $\nll,\nuu\in\R^*\doteq\R\cup\{-\infty,+\infty\}$ with $\nll<\nuu$, we write
$$
\N(\nll,\nuu)
$$
for the class of all $\overline{(\nll,\nuu)}$-valued random variables (RVs) $N$ in $(\Omega,\F,\P)$, the overline denoting closure (in $\R$), with
$$
\E|N|<\infty
$$
and probability law $\Lambda_N$ being absolutely continuous (a.c.) with respect to (w.r.t.) Lebesgue measure $\ell$ with density
$$
p_N\doteq\frac{d\Lambda_N}{d\ell},
$$
the Radon-Nikod\'ym derivative of $\Lambda_N$ w.r.t. $\ell$. It is assumed that $p_N\in\C^1(\nll,\nuu)$, the class of continuously differentiable functions\footnote{Troughout, $\C^k(A)$ will denote the class of $k$-times continuously differentiable functions over the domain $A$.} in $(\nll,\nuu)$, that $p_N>0$ everywhere in $(\nll,\nuu)$ and that the differential entropy of $N$,
$$
h(N)\doteq-\int p_N(n)\ln p_N(n)\ell(dn),
$$
is well defined (as a number in $\R$), i.e.,
\begin{equation}\label{fe}
\int p_N(n)|\ln p_N(n)|\ell(dn)<\infty.
\end{equation}
Note we take for convenience in computations and without loss of generality (w.l.o.g.) the logarithm in the definition of differential entropy as a natural logarithm, i.e., w.r.t. to the base $e$ (and therefore entropy is measured in nats), and assumed as usual the convention $0(\pm\infty)=0$ and $\ln 0=-\infty$. Moreover, for sake of definiteness we also set $\ln x=-\infty$ for $x<0$, whenever necessary in a Lebesgue null set.

Let $N\in\N(\nll,\nuu)$, $G:\R\rightarrow\R$ a given function in $\C^1(-\infty,\infty)$, $G$ convex with derivative $G'\neq0$ $\ell$-almost everywhere, and $\beta\in G(\R)$, the image set of $G$. Throughout we will consider real-valued RVs $X$ in $(\Omega,\F,\P)$ jointly distributed with $N$, such that
$$
\E|X|<\infty\;\;\;\text{and}\;\;\;\E|G(X)|<\infty,
$$
and satisfying the constraint
\begin{equation}\label{acc}
\E G(X)\leq\beta
\end{equation}
with $\E$ denoting expectation, i.e.,
$$
\E |X|\doteq\int |X(\omega)|\P(d\omega)=\int |x|\Lambda_X(dx)
$$
and similarly for the others, and where $\Lambda_X$ denotes the corresponding marginal probability law for $X$ extracted from the joint law $\Lambda_{X,N}$
$$
\Lambda_X(B)=\Lambda_{X,N}(B,\R)=\Lambda_{X,N}(B,\overline{(\nll,\nuu)}),\;\;\;B\in\B(\R),
$$
with $\B(\R)$ the Borel sigma-algebra in $\R$. (Note $\Lambda_X$ is not assumed to be a.c. w.r.t. $\ell$, i.e., it is not assumed to have a density w.r.t. Lebesgue measure.)

Since $G$ is convex with derivative $G'\neq0$ $\ell$-almost everywhere, the image set $G(\R)$ is an interval with non-empty interior\footnote{The condition $G'\neq0$ $\ell$-almost everywhere in particular precludes the trivial case of G being constant, where the constraint plays no role whatsoever. It also avoids the degenerate case where $G$ could take upon a minimum in an interval of positive Lebesgue measure, so being constant in that interval.}. In fact, by setting
$$
\alpha_0\doteq\inf_{x\in\R}G(x)\in[-\infty,\infty),
$$
the interior of $G(\R)$ is the open interval $(\alpha_0,\infty)$. We set
$$
\I_G\doteq(\alpha_0,\infty)
$$
and, in order to avoid trivial or degenerate cases when it comes to the constraint,
$$
\E G(X)\leq\beta,
$$
consider from now on
$$
\beta\in\I_G.
$$

Whenever the resulting sum
$$
Y\doteq X+N
$$
has a probability law $\Lambda_Y$ a.c. w.r.t. Lebesgue measure with density
$$
p_Y\doteq\frac{d\Lambda_Y}{d\ell}
$$
and such that the differential entropy of $Y$
$$
h(Y)\doteq-\int p_Y(y)\ln p_Y(y)\ell(dy)
$$
is well defined \footnote{Throughout, a differential entropy will be said to be well defined when the corresponding analogous condition than for $p_N$ in (\ref{fe}) holds.}, we write
$$
Y=\Sigma(X,N;G,\beta).
$$
The function $G$ above is called the \textit{cost function} and the constraint (\ref{acc}) and $\beta$ are referred to as an \textit{average cost constraint} and \textit{maximum average cost} at the input, respectively. Following the notation, when $X$ and $N$ are in addition stochastically independent, i.e., with their joint law factorizing as the product of the corresponding marginals, we write
$$
Y=\Sigma_{\bot}(X,N;G,\beta).
$$

The motivation for asking for the resulting sum $X+N$ in $\Sigma(X,N;G,\beta)$ to have a probability law a.c. w.r.t. Lebesgue measure comes from the independent case. Indeed, if $N$ has a probability law $\Lambda_N$ a.c. w.r.t. Lebesgue measure and $X$ and $Y$ are RVs such that $Y=X+N$ with $X$ and $N$ independent, then for any set $B\in\B(\R)$ with $\ell(B)=0$ we have
\begin{eqnarray*}
\Lambda_Y(B)&=&\Lambda_{X+N}(B)\\
&=&\int\Lambda_N(B-x)\Lambda_X(dx)\\
&=&0
\end{eqnarray*}
with $B-x\doteq\{z\in\R:z+x\in B\}$, since $\ell(B-x)=0$ and $\Lambda_N$ is a.c. w.r.t. Lebesgue measure. Therefore, $\Lambda_Y$ is a.c. w.r.t. Lebesgue measure too.

Furthermore, we write
$$
\Sigma(N;G,\beta)
$$
for the set of all RVs $Y$ that admit the representation
$$
Y=\Sigma(X,N;G,\beta)
$$
for some RV $X$, and analogously for the set
$$
\Sigma_{\bot}(N;G,\beta).
$$
Note they are both non-empty. Indeed, the constant RV $X\doteq x$ with $G(x)=\beta$ (at least one such $x$ always exists since $\beta\in\I_G\subseteq G(\R)$) trivially satisfies the average cost constraint and, moreover, the sum $x+N$ is a.c. w.r.t. Lebesgue measure and with well defined differential entropy (note $h(x+N)=h(N)$ and $N\in\N(\nll,\nuu)$).

Given an RV $N\in\N(\nll,\nuu)$, cost function $G$ and maximum average cost $\beta$, we will consider in the following sections the problems of finding
\begin{equation}\label{GMP}
\sup\{h(Y):Y\in\Sigma(N;G,\beta)\}
\end{equation}
and
\begin{equation}\label{IMP}
\sup\{h(Y):Y\in\Sigma_{\bot}(N;G,\beta)\}
\end{equation}
where for sake of well definiteness the suprema are always considered as $\R^*$-valued. Note we may interpret
$$
Y=\Sigma_{\bot}(X,N;G,\beta)
$$
as the output of a discrete time memoryless additive noise channel with input $X$, average cost constraint at the input
$$
\E G(X)\leq\beta
$$
and noise $N$. For such a channel, we define the corresponding channel capacity $C(N;G,\beta)$ by
$$
C(N;G,\beta)\doteq\sup\{h(Y):Y\in\Sigma_{\bot}(N;G,\beta)\}-h(N)
$$
with the same well definiteness convention for the suprema as before. Note that, and as expected,
$$
C(N;G,\beta)\geq 0
$$
(possibly $\infty$), since
$$
\sup\{h(Y):Y\in\Sigma_{\bot}(N;G,\beta)\}\geq h(x+N)=h(N)
$$
upon considering the same constant RV $X\doteq x$ as before. The definition of channel capacity is inspired from the well known fact (see for example \cite{RG1968}) that in a discrete time memoryless additive noise channel with noise $N$, input $X$ and output $Y$, and with all respective probability laws being a.c. w.r.t. Lebesgue measure and with densities giving well define differential entropies, we have for the corresponding input-output mutual information $I(X;Y)$
$$
I(X;Y)=h(Y)-h(N).
$$
Note since obviously
\begin{equation*}
\sup\{h(Y):Y\in\Sigma_{\bot}(N;G,\beta)\}\leq\sup\{h(Y):Y\in\Sigma(N;G,\beta)\},
\end{equation*}
we have the gross upper bound on the channel capacity
$$
C(N;G,\beta)\leq\sup\{h(Y):Y\in\Sigma(N;G,\beta)\}-h(N).
$$

In the following section we will characterize the pairs of RVs
$$
(X^\dag,Y^\dag)
$$
in
$$
Y^\dag=\Sigma(X^\dag,N;G,\beta)
$$
satisfying
\begin{equation*}
\sup\{h(Y):Y\in\Sigma(N;G,\beta)\}=\max\{h(Y):Y\in\Sigma(N;G,\beta)\}=h(Y^\dag),
\end{equation*}
which will then be exploited to characterize when possible the case
\begin{equation*}
\sup\{h(Y):Y\in\Sigma_\bot(N;G,\beta)\}=\max\{h(Y):Y\in\Sigma_\bot(N;G,\beta)\}=h(Y^\ddag)
\end{equation*}
for some RV $Y^\ddag$ related to an RV $X^\ddag$ via
$$
Y^\ddag=\Sigma_{\bot}(X^\ddag,N;G,\beta).
$$
In that case
$$
C(N;G,\beta)=h(Y^\ddag)-h(N)
$$
and the probability law of such $X^\ddag$ will then give a corresponding capacity-achieving input distribution.

In the context of the problems posed by (\ref{GMP}) and (\ref{IMP}), we will consider amplitude constraints at the input as well, in problems of the general form
$$
\sup\{h(Y):Y\in\Sigma(N;G,\beta,\overline{(\xll,\xuu)})\}
$$
and
$$
\sup\{h(Y):Y\in\Sigma_{\bot}(N;G,\beta,\overline{(\xll,\xuu)})\}
$$
with $\xll,\xuu\in\R^*$, $\xll<\xuu$, and
$$
\Sigma(N;G,\beta,\overline{(\xll,\xuu)})
$$
and
$$
\Sigma_{\bot}(N;G,\beta,\overline{(\xll,\xuu)})
$$
defined the same as before but including the additional constraint
\begin{equation}\label{ac}
X\in\overline{(\xll,\xuu)},\;\;\P\text{-almost surely},
\end{equation}
i.e., with
$$
\Lambda_X(\overline{(\xll,\xuu)})=1,
$$
and assuming of course the constraints (\ref{acc}) and (\ref{ac}) are compatible \footnote{That is, $\inf\{G(\overline{(c,d)})\}\leq\beta$. Also, since the case with $c=-\infty$ and $d=\infty$ imposes no amplitude constraint at all, the relevant cases are either when $c>-\infty$, when $d<\infty$, or both. The case $c=d$ is trivial, of course.}. The corresponding notation for channel capacity
$$
C(N;G,\beta,\overline{(\xll,\xuu)})
$$
obeys to the same logic, i.e., it stands for
$$
\sup\{h(Y):Y\in\Sigma_{\bot}(N;G,\beta,\overline{(\xll,\xuu)})\}-h(N).
$$

Finally, and before proceeding with the development of the paper in the next section, we note that in addition to the trivial case when $G$ is constant, which has been ruled out from the conditions imposed on $G$, there is another case whose solution is straightforward. We briefly point it out in the following remark.

\begin{remark}\label{linear}
Consider $N\in\N(\nll,\nuu)$ and linear cost function $G(x)=px+q$, $x\in\R$, $p,q\in\R$ (the case of constant $G$ being anyway included here when $p=0$). Then, for any $\beta\in\R$ if $p\neq 0$, or for $\beta=q$ if $p=0$, we have
$$
\sup\{h(Y):Y\in\Sigma(N;G,\beta)\}=\infty
$$
and
$$
C(N;G,\beta)=\infty.
$$
Indeed, consider the family of RVs $\{X_k\}_{k\geq 0}$, where each $X_k$ is a Gaussian RV of mean $\mu$ and variance $\sigma^2_k$ independent of $N$, and with $\mu$ satisfying
$$
\mu=\frac{\beta-q}{p}
$$
if $p\neq 0$, or
$$
\mu=q=\beta
$$
if $p=0$, and $\sigma_k^2$ chosen such that
$$
\sigma_k^2\rightarrow\infty
$$
as $k\rightarrow\infty$. Then, the constraint is satisfied, and also we have
$$
h(X_k)\rightarrow\infty
$$
as $k\rightarrow\infty$ as well. From the well known inequality (see for example \cite{TC1991}) that for $X_k$ and $N$ independent
$$
h(X_k+N)\geq\max\{h(X_k), h(N)\},
$$
it follows that
$$
h(X_k+N)\geq h(X_k)\rightarrow\infty
$$
as $k\rightarrow\infty$, and therefore
$$
C(N;G,\beta)=\infty.
$$
Hence,
$$
\sup\{h(Y):Y\in\Sigma(N;G,\beta)\}=\infty
$$
too.
\end{remark}

\section{Optimum Characterization}
\label{main}


For a given $N\in\N(\nll,\nuu)$, cost function $G$ and maximum average cost $\beta$, we first consider the problem posed in (\ref{GMP}), i.e., the problem of finding
$$
\sup\{h(Y):Y\in\Sigma(N;G,\beta)\}
$$
and its channel capacity applications subsumed in (\ref{IMP}). At the end of the section we will discuss on how the case including amplitude constraints
$$
\sup\{h(Y):Y\in\Sigma(N;G,\beta,\overline{(\xll,\xuu)})\}
$$
can be handled, as well as its corresponding channel capacity applications.

Towards a characterization of (\ref{GMP}), we begin by noting that given a probability density function (pdf) $p$, there always exists a non-decreasing transformation $f$ such that
$$
\Lambda_{f(N)}(B)=\int_{B}p(u)\ell(du)
$$
for all $B\in\B(B)$, i.e., such that the RV $f(N)$ admits $p$ as a density w.r.t. Lebesgue measure. Indeed, it is easy to see that with $F_N$ the probability distribution function of $N$, i.e.,
$$
F_N(\cdot)\doteq\Lambda_N((-\infty,\cdot])=\int_{(-\infty,\cdot]}p_N(n)\ell(dn),
$$
it suffices to take
$$
f\doteq r_F\circ F_N
$$
with $\circ$ denoting composition and $r_F$ the right-inverse of the probability distribution function $F$ associated $p$, i.e.,
$$
r_F(\cdot)\doteq\inf\{u:F(u)>\cdot\}
$$
and
$$
F(\cdot)\doteq\int_{(-\infty,\cdot]}p(u)\ell(du)
$$
(giving then $(F\circ r_F)(\cdot)=\cdot$).
Since $f$ so defined is non-decreasing, its derivative $f'$ is well defined $\ell$-almost everywhere, and therefore, since also at those points and from $F\circ f=F_N$ we have
$$
(p\circ f)(\cdot)f'(\cdot)=p_N(\cdot)
$$
and $p_N$ is strictly positive in $(\nll,\nuu)$, we conclude $f'>0$ $\ell$-almost everywhere in $(\nll,\nuu)$ as well.

Motivated from the preceding discussion and since in our context we have an RV $Y$ obtained from $X$ and $N$ as
$$
Y=\Sigma(X,N;G,\beta),
$$
we are lead to searching for an extremum of the differential entropy at the output $h(Y)$ occurring at a joint distribution for $(X,N)$ induced by a relationship between $X$ and $N$ of the form
$$
X=g(N),
$$
i.e., giving
$$
Y\eqd f(N)\doteq g(N)+N,
$$
with $\eqd$ denoting equality in distribution, for some appropriate transformation $f$ (at least differentiable $\ell$-almost everywhere in $(\nll,\nuu)$ and with $f'>0$ at those points), and such that
$$
\E G(X)=\E G(g(N))\leq\beta
$$
(of course, and for well posedness, with $\E|G(g(N))|<\infty$, $\E|g(N)|<\infty$ and $h(g(N)+N)\in\R$). That when such an extremum exists it indeed corresponds to a true maximum
$$
\max\{h(Y):Y\in\Sigma(N;G,\beta)\}
$$
taken over all corresponding joint distributions $\Lambda_{X,N}$, will be discussed shortly.

We characterize such $g=f-id$ in $(\nll,\nuu)$ first, with $id$ the identity, postponing for the moment the discussion of the appropriate boundary conditions. To proceed further, since $p_N\in\C^1(\nll,\nuu)$ and $p_N>0$ in $(\nll,\nuu)$, we will in fact search for an extremum with $f'>0$ and $f'$ continuous (that is, sufficiently smooth), both everywhere in $(\nll,\nuu)$, i.e., in terms of $g$,
$$
g\in\C^1(\nll,\nuu)\;\;\text{and}\;\;g'>-1\;\;\text{everywhere in $(\nll,\nuu)$}.
$$
Note then, in particular, by considering the standard transformation of pdfs via strictly increasing differentiable mappings, we obtain
$$
p_Y(y)=p_N(f^{-1}(y))f^{-1'}(y).
$$
But,
\begin{equation}\label{rel1}
y=f(n)\;\;\;\text{and}\;\;\;n=f^{-1}(y),
\end{equation}
therefore
\begin{equation}\label{rel2}
y=f(n)=g(n)+n=g(f^{-1}(y))+f^{-1}(y),
\end{equation}
and thus, upon taking derivatives w.r.t. $y$,
$$
1=g'(f^{-1}(y))f^{-1'}(y)+f^{-1'}(y),
$$
or
\begin{equation}\label{rel3}
f^{-1'}(y)=\frac{1}{1+g'(f^{-1}(y))}=\frac{1}{1+g'(n)}.
\end{equation}
Also, since
$$
\frac{dy}{dn}=f'(n)=1+g'(n),
$$
we then conclude
$$
p_Y(y)\ln p_Y(y)\ell(dy)=p_N(n)\ln\frac{p_N(n)}{1+g'(n)}\ell(dn),
$$
and hence the differential entropy at the output \footnote{Since throughout integrals are w.r.t. Lebesgue measure, we may interchangeable consider for the domain of integration $\overline{(\nll,\nuu)}$ or $(\nll,\nuu)$.}
\begin{equation}\label{ODE}
h(Y)=-\int_{\nll}^{\nuu}p_N(n)\ln\frac{p_N(n)}{1+g'(n)}\ell(dn).
\end{equation}
As for the constraint, we have
$$
\E G(X)=\int_{\nll}^{\nuu}G(g(n))p_N(n)\ell(dn)\leq\beta.
$$
Before proceeding further, we note by Jensen's inequality (see for example \cite{DW1991})
$$
G(\E g(N) +c)\leq\E G(g(N)+c)
$$
for any $c\in\R$, and therefore, since $\beta\in \I_G$ and from (\ref{ODE}) the differential entropy at the output is invariant under changes $g\leadsto g+c$ for constant $c\in\R$ (reflecting the well known fact that $h(Y+c)=h(Y)$), a direct continuity argument shows we may w.l.o.g. assume
$$
\E G(g(N))=\beta
$$
at an extremum over such $g$, i.e.,
\begin{equation}\label{IACC}
\int_{\nll}^{\nuu}G(g(n))p_N(n)\ell(dn)=\beta.
\end{equation}
Setting
$$
J(n,u,v)\doteq -p_N(n)\ln\frac{p_N(n)}{1+v}
$$
and
$$
K(n,u,v)\doteq p_N(n)G(u),
$$
which belong to the class $\C^2$ over the domain $(a,b)\times\R\times(-1,\infty)$, we recognize then (\ref{ODE}) and (\ref{IACC}) in functional notation as
\begin{equation}\label{ipp1}
\mathcal{J}[g]\doteq\int_{\nll}^{\nuu}J(n,g,g')\ell(dn)
\end{equation}
and
\begin{equation}\label{ipp2}
\mathcal{K}[g]\doteq\int_{\nll}^{\nuu}K(n,g,g')\ell(dn)=\beta,
\end{equation}
respectively, and therefore the problem of characterizing an extremum for the differential entropy at the output in (\ref{ODE}) with average cost constraint at the input in (\ref{IACC}) accounts for solving the corresponding isoperimetric problem from the calculus of variations (see for example \cite{IMGSVF2000,BVB2006}) posed by (\ref{ipp1}) and (\ref{ipp2}), with appropriate boundary conditions considered shortly. Thus, as long as $g$ is not an extremum also for $\mathcal{K}$, a local (smooth) extremum for $\mathcal{J}$ in (\ref{ipp1}) satisfying the constraint in (\ref{ipp2}) and occurring at $g$ ($\in\C^1(\nll,\nuu)$) solves the Euler-Lagrange equation in $(\nll,\nuu)$
\begin{equation}\label{ELe}
\frac{d}{dn}\frac{\partial H}{\partial g'}-\frac{\partial H}{\partial g}=0
\end{equation}
with
$$
H(n,g,g')\doteq J(n,g,g')-\lambda K(n,g,g')
$$
for some $\lambda\in\R$, the corresponding Lagrange multiplier. Since $\lambda$ will correspond to the rate of change of the extremum $\mathcal{J}[g]$ (when it exists, of course) with respect to the isoperimetric parameter $\beta$ (see for example \cite{BVB2006}), and since we obviously have non-decreasingness with respect to $\beta$ inherited from being originally an upper bound constraint, we in fact have that $\lambda\geq 0$.

That indeed no common extremum for $\mathcal{J}$, along with constraint \footnote{Strictly speaking also satisfying the appropriate boundary conditions to be considered, as mentioned, shortly.} (\ref{ipp2}), and $\mathcal{K}$ exists, and so validating the Euler-Lagrange equation approach \footnote{The abnormal case of common extremals for $\mathcal{J}$ and $\mathcal{K}$ is known in the literature as the rigid extremals case.} (again, see for example \cite{IMGSVF2000,BVB2006}), will be discussed shortly. We now proceed further with the resolution of (\ref{ELe}). We have
$$
\frac{\partial H}{\partial g}(n)=-\lambda p_N(n)G'(g(n))\;\;\;\text{and}\;\;\;\frac{\partial H}{\partial g'}(n)=\frac{p_N(n)}{1+g'(n)}.
$$
Note since $\frac{\partial H}{\partial g'}=\frac{p_N}{1+g'}$ and $p_N$ continuous in $(\nll,\nuu)$, the Weierstrass-Erdmann corner conditions in particular rule out in fact the possibility of broken extremals (that is, being piecewise smooth; see for example \cite{IMGSVF2000}). Moreover, since
$$
\frac{\partial^2 H}{\partial g'^2}(n)=-\frac{p_N(n)}{(1+g'(n))^2}\neq 0
$$
in $(\nll,\nuu)$, from \cite[Thm. 3, pp. 17]{IMGSVF2000} we conclude that indeed $g\in\C^2(\nll,\nuu)$, and therefore
we may explicitly write
$$
\frac{d}{dn}\frac{\partial H}{\partial g'}(n)=\frac{p_N'(n)(1+g'(n))-p_N(n)g''(n)}{(1+g'(n))^2}.
$$
Hence, the Euler-Lagrange equation in $(\nll,\nuu)$ becomes
\begin{equation}\label{ELCFI}
\frac{p_N'(n)}{p_N(n)}-\frac{g''(n)}{1+g'(n)}+\lambda G'(g(n))(1+g'(n))=0,
\end{equation}
i.e.,
$$
\frac{d}{dn}\ln\frac{p_N(n)}{1+g'(n)}=-\lambda G'(g(n))(1+g'(n)),
$$
which admits the first integral \footnote{Note it is also possible to integrate equation (\ref{ELe}) directly, before taking the $d/dn$-derivative. However, it will prove to be more convenient to get an equation such as (\ref{firstintegral}), where $p_N$ has been isolated.}
\begin{equation}\label{firstintegral}
p_N(n)=C(1+g'(n))e^{-\lambda\left(\int G'(g(n))dn+G(g(n))\right)},
\end{equation}
$n\in(\nll,\nuu)$, for some real constant $C>0$ and with $\int G'(g(n))dn$ denoting a primitive function for $G'(g(n))$ (w.l.o.g. we may assume its arbitrary constant of integration has already been absorbed in $C$). Note then the requirement $g'>-1$ everywhere in $(\nll,\nuu)$ is automatically included.

Now we come back to the question as to what boundary conditions should be imposed on $g$ when further integrating (\ref{firstintegral}). Since fixed boundary conditions are not known a priori for the curve $g$ corresponding to a potential extremum, the so called natural boundary conditions must then be imposed (\cite{IMGSVF2000,BVB2006}), which correspond to asking, and without precluding a priori too possible asymptotes at the boundaries,
$$
\frac{\partial H}{\partial g'}(\nll+)\doteq\lim_{\substack{n\rightarrow\nll\\n>\nll}}\frac{\partial H}{\partial g'}(n)=0
$$
and
$$
\frac{\partial H}{\partial g'}(\nuu-)\doteq\lim_{\substack{n\rightarrow\nuu\\n<\nuu}}\frac{\partial H}{\partial g'}(n)=0,
$$
i.e.,
\begin{equation}\label{BCs}
\lim_{\substack{n\rightarrow\nll\\n>\nll}}\frac{p_N(n)}{1+g'(n)}=\lim_{\substack{n\rightarrow\nuu\\n<\nuu}}\frac{p_N(n)}{1+g'(n)}=0,
\end{equation}
the existence of such limits being of course implicitly asked in the previous notation \footnote{Note continuity or even differentiability in the interior does not guarantee existence of the limits at the boundaries, as shown for example by $\sin(1/x)$ in a vicinity of $0$.}.

Note from the first integral in (\ref{firstintegral}) and the boundary conditions in (\ref{BCs}), we can rule out the possibility of $\lambda=0$, concluding then $\lambda>0$. Indeed, $\lambda=0$ in (\ref{firstintegral}) would imply
$$
\frac{p_N(n)}{1+g'(n)}=C,\;\;\;n\in(\nll,\nuu),
$$
and then, from the boundary conditions in (\ref{BCs}),
$$
C=0,
$$
which cannot be true.

It remains then to check that, as claimed, no common extremum for $\mathcal{J}$, satisfying (\ref{ipp2}) and (\ref{BCs}), and $\mathcal{K}$ exists. Indeed, no solution to the Euler-Lagrange equation (\ref{ELe}) that satisfies the constraint (\ref{ipp2}) and the boundary conditions (\ref{BCs}) can be an extremum for $\mathcal{K}$. If so, since
$$
\frac{\partial H}{\partial g}=-\lambda\frac{\partial K}{\partial g}=-\lambda\left(\frac{d}{dn}\frac{\partial K}{\partial g'}+\frac{\partial K}{\partial g}\right)
$$
(being $\frac{\partial K}{\partial g'}\equiv 0$), and since also the corresponding Euler-Lagrange equation an extremum for $\mathcal{K}$ must satisfy is precisely
$$
\frac{d}{dn}\frac{\partial K}{\partial g'}+\frac{\partial K}{\partial g}=0,
$$
we conclude
$$
\frac{\partial H}{\partial g}=0.
$$
Therefore from (\ref{ELe})
$$
\frac{d}{dn}\frac{\partial H}{\partial g'}(n)=\frac{d}{dn}\frac{p_N(n)}{1+g'(n)}=0,\;\;\;n\in(\nll,\nuu),
$$
hence
$$
\frac{p_N(n)}{1+g'(n)}=\text{constant},\;\;\;n\in(\nll,\nuu),
$$
and thus, again from the boundary conditions in (\ref{BCs}),
$$
\frac{p_N}{1+g'}\equiv 0\;\;\;\text{in $(\nll,\nuu)$},
$$
a contradiction.

So far we have characterized by (\ref{ELe}) and (\ref{BCs}), equivalently (\ref{firstintegral}) and (\ref{BCs}), possible local (smooth) extrema for (\ref{ODE}) and satisfying the imposed average cost constraint at the input. However, as well known, (\ref{ELe}) and (\ref{BCs}) correspond to the infinite dimensional counterpart of the Lagrangian's null gradient condition in elementary analysis, not guaranteeing by any means the existence of an extremum, even locally. However, what we have at hand corresponds to a concave optimization problem, inherited from the concavity of $h(Y)$ in $p_Y(\cdot)$. Indeed, rewrite the entropy functional in (\ref{ODE}) as
\begin{equation}\label{ODE2}
h(Y)=h(N)+\int_{\nll}^{\nuu}p_N(n)\ln(1+g'(n))\ell(dn)
\end{equation}
and consider $g_1,g_2$ both being $\ell$-almost everywhere differentiable in $(\nll,\nuu)$ and with $g_1',g_2'>-1$ at those points, and such that
$$
\E G(g_i(N))\leq\beta,\;\;\;i=1,2
$$
and
$$
\int_{\nll}^{\nuu}p_N(n)\big|\ln(1+g_i'(n))\big|\ell(dn)<\infty,\;\;\;i=1,2
$$
with $\alpha_1,\alpha_2>0$ and $\alpha_1+\alpha_2=1$. The corresponding convex linear combination $\alpha_1g_1+\alpha_2g_2$ is $\ell$-almost everywhere differentiable in $(\nll,\nuu)$ with $(\alpha_1g_1+\alpha_2g_2)'>-1$ at those points and, from the convexity of $G$,
$$
\E G(\alpha_1g_1(N)+\alpha_2g_2(N))\leq\alpha_1\beta+\alpha_2\beta=\beta.
$$
Moreover, since $1+\alpha_1g_1'+\alpha_2g_2'=\alpha_1(1+g_1')+\alpha_2(1+g_2')$ with $\alpha_i(1+g_i')>0$ $\ell$-almost everywhere in $(\nll,\nuu)$, $i=1,2$, it is easy to see that
\begin{equation*}
\frac{1}{2}\big|\ln(1+\alpha_1g_1'+\alpha_2g_2')\big|
\leq\big|\ln(1+g_1')\big|+\big|\ln(1+g_2')\big|+|\ln\alpha_1|+|\ln\alpha_2|
\end{equation*}
$\ell$-almost everywhere in $(\nll,\nuu)$ too, and therefore
$$
\int_{\nll}^{\nuu}p_N(n)\big|\ln(1+\alpha_1g_1'(n)+\alpha_2g_2'(n))\big|\ell(dn)<\infty.
$$
Thus, from the concavity of $\ln$, it is easy to see that
\begin{equation*}
\int_{\nll}^{\nuu}p_N(n)\ln(1+\alpha_1g_1'(n)+\alpha_2g_2'(n))\ell(dn)
\geq\alpha_1\int_{\nll}^{\nuu}p_N(n)\ln(1+g_1'(n))\ell(dn)+\alpha_2\int_{\nll}^{\nuu}p_N(n)\ln(1+g_2'(n))\ell(dn),
\end{equation*}
as claimed.

Being a concave optimization problem, we are then lead to the conclusion that the existence of real constants $C>0$ and $\lambda>0$ and of a corresponding solution $g\in\C^1(\nll,\nuu)$ to (\ref{firstintegral}) and satisfying $\E|g(N)|<\infty$, (\ref{BCs}) and (\ref{IACC}) \footnote{Note then and by construction $\E G(g(N))=\beta\in\R$, and therefore in fact $\E|G(g(N))|<\infty$.} in fact corresponds to the existence of a proper and well defined maximum for (\ref{ODE}),
attained at
$$
Y^\dag=\Sigma(X^\dag,N;G,\beta)
$$
with
$$
X^\dag=g(N).
$$

It only remains to check that the maximum above does indeed solve the original problem of computing
$$
\sup\{h(Y):Y\in\Sigma(N;G,\beta)\},
$$
being then equal to $\max\{h(Y):Y\in\Sigma(N;G,\beta)\}$ with
\begin{eqnarray*}
\max\{h(Y):Y\in\Sigma(N;G,\beta)\}
&=&h(Y^\dag)\\
&=&h(X^\dag+N)\\
&=&h(g(N)+N).
\end{eqnarray*}

For that purpose, consider $\sigma(N)\subseteq\F$, the sigma-algebra generated by $N$ (the noise RV), and note each input $X$ to the channel as well as the noise belong to $L^1(\Omega,\F,\P)$, the space of integrable RVs. Therefore, each corresponding output $Y$ also belongs to $L^1(\Omega,\F,\P)$, and we can therefore consider the projection of the channel
$$
Y=X+N
$$
onto $\sigma(N)$, to get
\begin{equation}\label{projection}
\E[Y|\sigma(N)]=\E[X|\sigma(N)]+N,
\end{equation}
$\P$-almost surely. That is, with $Y_N\doteq\E[Y|\sigma(N)]$ and $X_N\doteq\E[X|\sigma(N)]$ (both defined as versions in a $\P$- almost surely sense),
$$
Y_N=X_N+N.
$$
As for the constraint
$$
\E G(X)\leq\beta,
$$
and by Jensen's inequality,
$$
\beta\geq\E G(X)=\E\E[G(X)|\sigma(N)]\geq\E G(\E[X|\sigma(N)]),
$$
i.e.,
$$
\E G(X_N)\leq\beta.
$$
Hence, we have
$$
Y_N=X_N+N
$$
with the constraint
$$
\E G(X_N)\leq\beta,
$$
and therefore the channel (along with its constraint) is invariant under projection onto $\sigma(N)$. We conclude then the maximum, when it exists, must be attained at
$$
Y^\dag=\Sigma(X^\dag,N;G,\beta)
$$
with
$$
X^\dag=g(N),
$$
i.e., with $X^\dag$ (and hence $Y^\dag$ too) being $\sigma(N)$-measurable. Indeed, since for each input-output pair
$$
Y=X+N,
$$
$$
\E G(X)\leq\beta,
$$
we have invariance
$$
Y_N=X_N+N,
$$
$$
\E G(X_N)\leq\beta,
$$
and since for each such pair there exist Borel-measurable functions $f$ and $g$ such that
$$
Y_N=f(N),\;\;\;X_N=g(N),
$$
$\P$-almost surely, that is
$$
f(N)=g(N)+N,
$$
$$
\E G(g(N))\leq\beta,
$$
then, if $X$ is such that $Y$ generates maximum entropy at the output, we have that its corresponding $g$ must in turn correspond to a maximum entropy curve in the noise-input plane, and therefore it must satisfy the Euler-Lagrange equation in (\ref{ELCFI}) for some $\lambda>0$, along with the boundary conditions in (\ref{BCs}), with the corresponding regularity properties for $g$ (equivalently $f$) following by the arguments leading to such equation. Hence, since $\lambda>0$ and with $X^\dag_N=g(N)$, in the optimum we must have
$$
\E G(g(N))=\beta,
$$
and therefore
$$
\beta\geq\E G(X^\dag)\geq\E G(X_N^\dag)=\E G(g(N))=\beta,
$$
from where we conclude
$$
\E\left[G(X^\dag)-G(g(N))\right]=0,
$$
hence
$$
\E\E\left[G(X^\dag)-G(g(N))|\sigma(N)\right]=0
$$
but, since by Jensen's inequality
$$
\E[G(X^\dag)|\sigma(N)]\geq G(g(N))
$$
$\P$-almost surely, we conclude
$$
\E[G(X^\dag)|\sigma(N)]=G(g(N))
$$
$\P$-almost surely as well, and therefore
$$
X^\dag=g(N)
$$
$\P$-almost surely too, unless $G$ is either constant or linear, both cases that cannot deliver a finite extremum (see Remark \ref{linear} in Section \ref{setting}).

Before proceeding further, we make the following remark regarding the channel projection.

\begin{remark}
Note analogously to (\ref{projection}), but considering the corresponding orthogonal projection of the channel, gives
$$
Y-\E[Y|\sigma(N)]=X-\E[X|\sigma(N)],
$$
$\P$-almost surely, and therefore we have consistency
$$
Y^\dag-\E[Y^\dag|\sigma(N)]=X^\dag-\E[X^\dag|\sigma(N)]=0,
$$
$\P$-almost surely as well, being $X^\dag$ $\sigma(N)$-measurable.
\end{remark}

Finally, before stating the result characterizing the optimum when it exists, we note from the concavity of the optimization problem at hand, we know the optimizing function $g$, when it exists, it is unique in an $\ell$-almost everywhere sense. Moreover, since the previous arguments show that $X^\dag$, and hence $Y^\dag$ too are both $\sigma(N)$-measurable, we obtain $\P$-almost surely uniqueness in the following sense. Assume there exists $(X^\dag,Y^\dag)$ such that
$$
Y^\dag=\Sigma(X^\dag,N;G,\beta)
$$
$\P$-almost surely and
\begin{equation*}
\sup\{h(Y):Y\in\Sigma(N;G,\beta)\}=\max\{h(Y):Y\in\Sigma(N;G,\beta)\}=h(Y^\dag).
\end{equation*}
If there exists other pair $(\hat{X}^\dag,\hat{Y}^\dag)$ such that
$$
\hat{Y}^\dag=\Sigma(\hat{X}^\dag,N;G,\beta)
$$
$\P$-almost surely and $h(\hat{Y}^\dag)=h(Y^\dag)$, then we have
$$
(\hat{X}^\dag,\hat{Y}^\dag)=(X^\dag,Y^\dag)\;\;\P\text{-almost surely as well}.
$$

In summary, and joining together all the elements discussed previously in this section, we have thus proved the following result.

\begin{theorem}\label{thm1}
Consider $N\in\N(\nll,\nuu)$, cost function $G$ and maximum average cost $\beta$. Assume there exist real constants $C>0$ and $\lambda>0$ and a function $g\in\C^1(\nll,\nuu)$ such that \footnote{Recall from equation (\ref{firstintegral}) that $\int G'(g(n))dn$ denotes a primitive function for $G'(g(n))$ whose arbitrary constant of integration has already been absorbed (w.l.o.g.) in $C$.}
$$
p_N(n)=C(1+g'(n))e^{-\lambda\left(\int G'(g(n))dn+G(g(n))\right)}
$$
for $n\in(\nll,\nuu)$,
$$
\lim_{\substack{n\rightarrow\nll\\n>\nll}}\frac{p_N(n)}{1+g'(n)}=\lim_{\substack{n\rightarrow\nuu\\n<\nuu}}\frac{p_N(n)}{1+g'(n)}=0,
$$
$$
\E G(g(N))=\beta
$$
and $\E |g(N)|<\infty$. Then, with
$$
Y^\dag=\Sigma(X^\dag,N;G,\beta)\;\;\;\text{and}\;\;\;X^\dag=g(N)
$$
we have
\begin{equation*}
\sup\{h(Y):Y\in\Sigma(N;G,\beta)\}=\max\{h(Y):Y\in\Sigma(N;G,\beta)\}=h(Y^\dag).
\end{equation*}
Moreover, the pair $(X^\dag,Y^\dag)$ as above is unique in a $\P$-almost surely sense.
\end{theorem}

Before discussing channel capacity applications in the next section, we make the following remarks.

\begin{remark}
Note generally speaking, and from a differential equations point of view, further integration of
$$
p_N(n)=C(1+g'(n))e^{-\lambda\left(\int G'(g(n))dn+G(g(n))\right)}
$$
in order to solve for $g$ introduces a third real constant, say $K$, leading then to the determination of $C$, $\lambda$ and $K$ to completely specify such $g$. These constants must then be solved for from the three available equations, namely the two boundary conditions in (\ref{BCs}) and the equality constraint in (\ref{IACC}). 
\end{remark}

\begin{remark}\label{ELAlt}
In some applications (see Section \ref{examples}) it will show to be useful to rewrite the first integral to the Euler-Lagrange equation, namely
$$
p_N(n)=C(1+g'(n))e^{-\lambda\left(\int G'(g(n))dn+G(g(n))\right)},
$$
and as an intermediate stage to further integrating for $g$, in terms of $f^{-1}$ with $f=g+id$ as before, which, upon using equations (\ref{rel1}), (\ref{rel2}) and (\ref{rel3}), becomes
\begin{equation*}
p_N(f^{-1}(y))f^{-1'}(y)=Ce^{-\lambda\int G'(y-f^{-1}(y))dy}.
\end{equation*}
\end{remark}

\section{Capacity Achievability}
\label{main2}

In this section we discuss the applications of Theorem \ref{thm1} in the context of channel capacity. As mentioned in Section \ref{setting}, and under the hypotheses of the previous theorem, we immediately have the gross upper bound
$$
C(N;G,\beta)\leq h(Y^\dag)-h(N)
$$
where
\begin{equation*}
\sup\{h(Y):Y\in\Sigma(N;G,\beta)\}=\max\{h(Y):Y\in\Sigma(N;G,\beta)\}=h(Y^\dag)
\end{equation*}
with
$$
Y^\dag=\Sigma(X^\dag,N;G,\beta)\;\;\;\text{and}\;\;\;X^\dag=g(N).
$$

In what follows we focus on the case when an attainable upper bound for channel capacity may in fact be given, so corresponding to a capacity-achieving input distribution.

Consider $N\in\N(\nll,\nuu)$ and $G$ a cost function as before. We assume for the remaining of this section that for each $\alpha\in \I_G$, the conditions of Theorem \ref{thm1} hold with maximum average cost $\alpha$, i.e., that for each $\alpha\in \I_G$ we may write
\begin{equation*}
\sup\{h(Y):Y\in\Sigma(N;G,\alpha)\}=\max\{h(Y):Y\in\Sigma(N;G,\alpha)\}=h(Y^\dag_\alpha)
\end{equation*}
with
$$
Y_\alpha^\dag=\Sigma(X_\alpha^\dag,N;G,\alpha)\;\;\;\text{and}\;\;\;X_\alpha^\dag=g_\alpha(N).
$$
Conditions for that to be the case will be presented in Section \ref{existence} of the paper. (Note from Section \ref{main} $h(Y_\alpha^\dag)$ is not only strictly increasing in $\alpha\in \I_G$, $\alpha$ being the isoperimetric parameter, but also differentiable\footnote{With the derivative of $h(Y_\alpha^\dag)$ w.r.t. $\alpha$ being equal to the corresponding Lagrange multiplier $\lambda(\alpha)>0$.}, and hence continuous too.)

Write $\F[\mu]$ for the Fourier transform of a finite Borel measure $\mu$ in $\R$, i.e.,
$$
\F[\mu](f)\doteq\int_{\R}e^{-j2\pi f\xi}\mu(d\xi),\;\;\;f\in\R,
$$
with $j\doteq\sqrt{-1}$. Then, for each $\alpha\in \I_G$,
$$
\F[\Lambda_{Y^\dag_{\alpha}}](f)=\int_{\R}e^{-j2\pi f\xi}p_{Y^\dag_{\alpha}}({\xi})\ell(d\xi),\;\;\;f\in\R,
$$
and
$$
\F[\Lambda_{N}](f)=\int_{\nll}^{\nuu}e^{-j2\pi f\xi}p_{N}(\xi)\ell(d\xi),\;\;\;f\in\R,
$$
and where we have expressed the Fourier transforms of the laws of $Y^\dag_{\alpha}$ and $N$ in terms of the corresponding densities w.r.t. Lebesgue measure.

Now, for each $\alpha\in \I_G$ we set
$$
\X_\alpha(f)\doteq\frac{\F[\Lambda_{Y^\dag_{\alpha}}]}{\F[\Lambda_N]}(f)
$$
in
$$
S_N\doteq\{f\in\R:\F[\Lambda_N](f)\neq 0\},
$$
and note that whenever a continuous extension of $\X_\alpha$ from $S_N$ to the whole of $\R$ exists \footnote{Note from the continuity of $\F[\Lambda_{Y^\dag_{\alpha}}](\cdot)$ and $\F[N](\cdot)$, $\X_\alpha(\cdot)$ is already continuous in $S_N$.}, say $\overline{\X}_\alpha$, that indeed corresponds to the Fourier transform of a probability measure in $\R$, then we can associate to $Y_\alpha^\dag$ an RV, $X_{\alpha}^{ind}$, whose probability law is uniquely characterize by the Fourier transform $\overline{\X}_\alpha$, and such that
$$
Y_\alpha^\dag\eqd X_{\alpha}^{ind}+N
$$
and with $X_{\alpha}^{ind}$ and $N$ independent, by construction of its corresponding Fourier transform. By Bochner's theorem \footnote{The original result and proof of Bochner can be found in \cite{SB1933}. Other proofs of Bochner's theorem can be found for example in \cite{RC1975,IMGNYV1964}.}, that such a continuously extended version $\overline{\X}_\alpha$ indeed corresponds to the Fourier transform of a probability measure in $\R$, is equivalent to asking for the extension to be positive definite, i.e., with
$$
\sum_{i=1}^n\sum_{j=1}^nz_iz^*_j\overline{\X}_\alpha(f_i-f_j)\geq 0
$$
for any finitely many real $f_1,\ldots,f_n$ and complex $z_1,\ldots,z_n$, with $z^*_i$ denoting complex conjugation, or, in matrix jargon, for the matrix with $i,j$-th elements $\X_\alpha(f_i-f_j)$ being positive semi-definite for any such finitely many reals \footnote{Other characterizations of positive definite functions can be found for example in \cite{JLBC1959} and references therein. However, the first attempt in searching for such a probability measure is of course to directly check whether an appropriate inverse Fourier transform can be explicitly found.}. Note if that is the case, the advertised unique finite non-negative Borel measure (in $\R$) in Bochner's theorem indeed corresponds to a probability measure since, and from the continuity of $\F[\Lambda_{Y^\dag_{\alpha}}]$ and $\F[N]$,
$$
\overline{\X}_\alpha(0)=\X_\alpha(0)=\frac{\F[\Lambda_{Y^\dag_{\alpha}}]}{\F[\Lambda_N]}(0)=\frac{1}{1}=1.
$$
Note also if in addition $|\overline{\X}_\alpha|\in L^1(\R,\B(\R),\ell)$, i.e.,
$$
\int_{\R}|\overline{\X}_\alpha(f)|\ell(df)<\infty,
$$
then in fact the associated probability measure turns out to be a.c. w.r.t. Lebesgue measure, with density given, up to a Lebesgue-null set, by the corresponding standard inverse Fourier transform formula, i.e., by
$$
\int_{\R}\overline{\X}_\alpha(f)e^{j2\pi fx}\ell(df)
$$
for $\ell$-almost every $x\in\R$.

As we did for the optimal solution
\begin{equation*}
\sup\{h(Y):Y\in\Sigma(N;G,\alpha)\}=\max\{h(Y):Y\in\Sigma(N;G,\alpha)\}=h(Y^\dag_\alpha)
\end{equation*}
with
$$
Y_\alpha^\dag=\Sigma(X_\alpha^\dag,N;G,\alpha)\;\;\;\text{and}\;\;\;X_\alpha^\dag=g_\alpha(N),
$$
we also assume for the remaining of this section that for each $\alpha\in \I_G$, we in fact have the existence of such $X_{\alpha}^{ind}$, i.e., that for each $\alpha\in \I_G$ we may write
\begin{equation}\label{ind}
Y_\alpha^\dag\eqd X_{\alpha}^{ind}+N
\end{equation}
with the RVs $X_{\alpha}^{ind}$ and $N$ being independent. Again, conditions for that to be the case will be presented in Section \ref{existence} of the paper.

To proceed further, note since from (\ref{ind}) we have
$$
h(Y^\dag_{\alpha})=h(X^{ind}_{\alpha}+N)
$$
for each $\alpha\in \I_G$, and since as discussed in Section \ref{main} the rate of change of the optimum w.r.t. the isoperimetric parameter $\alpha$ is strictly positive and therefore $h(Y^\dag_{\alpha})$ is strictly increasing in $\alpha\in \I_G$, we conclude
$$
\E G(X^{ind}_{\alpha})\geq\alpha,
$$
and in fact, by the uniqueness of the optimum,
$$
\E G(X^{ind}_{\alpha})>\alpha,
$$
for each $\alpha\in \I_G$ as well. Set
$$
E(\alpha)\doteq\E G(X^{ind}_{\alpha}),\;\;\;\alpha\in \I_G,
$$
and note that $E(\alpha)$ is clearly strictly increasing and continuous for $\alpha\in \I_G$, that is for $\alpha>\alpha_0$ (recall $\alpha_0\doteq\inf_{x\in\R}G(x)$; what happens as $\alpha$ decreases to $\alpha_0$ when $\alpha_0>-\infty$ will be seen shortly). Therefore, if $\alpha_0=-\infty$, then for each $\alpha\in \I_G$ we have that there exists a unique $\alpha^\ddag(\alpha)\in \I_G$, with $\alpha^\ddag(\alpha)<\alpha$ (since $E(\alpha)>\alpha$), and such that
$$
Y^\dag_{\alpha^\ddag{(\alpha)}}\eqd X_{\alpha^\ddag{(\alpha)}}^{ind}+N
$$
and
$$
\E G(X_{\alpha^\ddag{(\alpha)}}^{ind})=\alpha.
$$
Hence, it is apparent that for all $\alpha\in \I_G$ we have
\begin{equation*}
\sup\{h(Y):Y\in\Sigma(N;G,\alpha^\ddag{(\alpha)})\}=\sup\{h(Y):Y\in\Sigma_{\bot}(N;G,\alpha)\}.
\end{equation*}

Indeed, assume on the contrary that for a given $\alpha\in\I_G$ there exists an RV $X^*$, independent of $N$, with
$$
\E G(X^*)\leq\alpha
$$
and such that, with $Y^*\doteq X^*+N$,
$$
h(Y_{\alpha^\ddag(\alpha)}^\dag)<h(Y^*)<\infty.
$$
Now note that, since
$$
h(Y_{\alpha^\ddag(\alpha)}^\dag)=h(X_{\alpha^\ddag{(\alpha)}}^{ind}+N),
$$
and since $X_{\alpha^\ddag{(\alpha)}}^{ind}$ and $N$ are independent, from a well know Information Theory inequality (see for example \cite{TC1991}) we may write
$$
h(Y_{\alpha^\ddag(\alpha)}^\dag)=h(X_{\alpha^\ddag{(\alpha)}}^{ind}+N)\geq\max\{h(X_{\alpha^\ddag{(\alpha)}}^{ind}), h(N)\},
$$
and therefore
$$
h(Y_{\alpha^\ddag(\alpha)}^\dag)\geq h(N).
$$
Consider the family of RVs
$$
X_\xi^*\doteq \xi X^*+s
$$
with $\xi\in [0,1]$ and $s\in\R$ a shifting parameter whose role will be stated shortly, and set accordingly
$$
Y_\xi^*\doteq X_\xi^*+N.
$$
Finally, for $\xi\in[0,1]$ set
$$
\hat{h}(\xi)\doteq h(Y_\xi^*),
$$
and therefore
$$
\hat{h}(\xi)=h(X_\xi^*+N)=h(\xi X^*+s+N)=h(\xi X^*+N),
$$
being $s$ just a shifting. It is easy to see that $\hat{h}$ is well defined and continuous, with
$$
\hat{h}(0)=h(N)\in\R
$$
and
$$
\hat{h}(1)=h(Y^*)\in\R.
$$
Then, since $h(Y_{\alpha^\ddag(\alpha)}^\dag)\geq h(N)$, we conclude that there exists $\xi_0\in[0,1)$ such that
$$
\hat{h}(\xi_0)=h(Y_{\alpha^\ddag(\alpha)}^\dag),
$$
that is
$$
h(Y_{\xi_0}^*)=h(X_{\xi_0}^*+N)=h(X_{\alpha^\ddag{(\alpha)}}^{ind}+N)=h(Y_{\alpha^\ddag(\alpha)}^\dag),
$$
with $X_{\xi_0}^*$ independent of $N$ ($X^*$ independent of $N$). Now, note since $\xi_0\in[0,1)$, $\xi_0 X^*$ is a compression of the values taken by $X^*$, and therefore, and without changing the entropy $h(Y_{\xi_0}^*)$, we can always choose $s$, by shifting the values of $\xi_0 X^*$ towards regions of lower values of $G$ if necessary, such that
$$
\E G(\xi_0 X^*+s)<\E G(X^*).
$$
Hence, we then have
\begin{equation}\label{cont}
\E G(\xi_0 X^*+s)<\E G(X^*)\leq\alpha =\E G(X_{\alpha^\ddag{(\alpha)}}^{ind}).
\end{equation}
Finally, note by the same arguments than in Section \ref{main}, we can write the representation
$$
Y_{\xi_0}^*\eqd g_{\xi_0}^*(N)+N
$$
for some appropriate function $g_{\xi_0}^*$, and therefore we then have
$$
h(g_{\xi_0}^*(N)+N)=h(Y_{\xi_0}^*)=h(Y_{\alpha^\ddag(\alpha)}^\dag)=h(g_{\alpha^\ddag(\alpha)}(N)+N).
$$
However, since $g_{\alpha^\ddag(\alpha)}$ is unique, being $h(Y_{\alpha^\ddag(\alpha)}^\dag)$ the corresponding optimum, we conclude
$$
g_{\xi_0}^*=g_{\alpha^\ddag(\alpha)}
$$
$\ell$-almost everywhere in $(\nll,\nuu)$, and therefore
$$
Y_{\xi_0}^*\eqd Y_{\alpha^\ddag(\alpha)}^\dag,
$$
from where, and since $X_{\xi_0}^*$ is independent of $N$ and being as well the decomposition
$$
Y^\dag_{\alpha^\ddag{(\alpha)}}\eqd X_{\alpha^\ddag{(\alpha)}}^{ind}+N
$$
unique,
$$
X_{\xi_0}^*\eqd X_{\alpha^\ddag{(\alpha)}}^{ind}.
$$
Hence,
$$
\E G(X_{\xi_0}^*)=\E G(X_{\alpha^\ddag{(\alpha)}}^{ind})=\alpha.
$$
This last equality, along with equation (\ref{cont}) above, then provides the desired contradiction,
$$
\alpha=\E G(X_{\xi_0}^*)<\alpha.
$$

Hence, we conclude that indeed for all $\alpha\in \I_G$ we have
\begin{equation*}
\sup\{h(Y):Y\in\Sigma(N;G,\alpha^\ddag{(\alpha)})\}=\sup\{h(Y):Y\in\Sigma_{\bot}(N;G,\alpha)\},
\end{equation*}
and therefore
\begin{equation*}
\sup\{h(Y):Y\in\Sigma_{\bot}(N;G,\alpha)\}=\max\{h(Y):Y\in\Sigma_{\bot}(N;G,\alpha)\}=h(Y^\ddag_{\alpha})
\end{equation*}
with
$$
Y^\ddag_{\alpha}=\Sigma_{\bot}(X_\alpha^\ddag,N;G,\alpha),
$$
and where
$$
X^\ddag_\alpha\doteq X_{\alpha^\ddag{(\alpha)}}^{ind},
$$
that is, with the probability law of $X^\ddag_\alpha$ so constructed being a corresponding capacity-achieving input distribution for each $\alpha\in \I_G$,
$$
C(N;G,\alpha)=h(Y^\ddag_{\alpha})-h(N).
$$
Note for each $\alpha\in \I_G$,
$$
Y^\ddag_{\alpha}\eqd Y^\dag_{\alpha^\ddag{(\alpha)}}
$$
and therefore
$$
C(N;G,\alpha)=h(Y^\dag_{\alpha^\ddag{(\alpha)}})-h(N).
$$

Now, if $\alpha_0>-\infty$ the exact same argument as before goes through by showing that
$$
\lim_{\substack{\alpha\rightarrow\alpha_0\\\alpha>\alpha_0}}E(\alpha)=\alpha_0.
$$
To that end, note if that is the case, that is if we have $\alpha_0=\inf_{x\in\R}G(x)>-\infty$, then, by the convexity of $G$ and the fact that $G'\neq0$ $\ell$-almost everywhere, we know there exists a unique $x_0\in\R^*$ (recall $\R^*\doteq\R\cup\{-\infty,+\infty\}$) such that
$$
\lim_{x\rightarrow x_0}G(x)=\alpha_0.
$$
Therefore, as $\alpha\rightarrow\alpha_0$, $\alpha>\alpha_0$, we have
$$
X_\alpha^\dag\arrowd x_0
$$
with $\arrowd$ denoting convergence in distribution (note if $x_0=\infty$, the previous convergence meaning that, and with $F_{X_\alpha^\dag}$ denoting the probability distribution function associated to the RV $X_\alpha^\dag$,
$$
F_{X_\alpha^\dag}(x)\rightarrow 0
$$
for all $x\in\R$ as $\alpha\rightarrow\alpha_0$, $\alpha>\alpha_0$, and similarly for the case when $x_0=-\infty$). Hence, since the constant RV $x_0$ is independent of $N$ (considering $N$ as an $\R^*$-valued RV if suitable), we also have the corresponding convergence
$$
X^{ind}_{\alpha}\arrowd x_0
$$
(with the same meaning as before when $x_0=\infty$ or $x_0=-\infty$). Set the function $G_0$ as
$$
G_0\doteq G-\alpha_0
$$
and note then $G_0\geq0$. Now, for $\alpha>\alpha_0$
$$
\E G_0(X^{ind}_{\alpha})=\E\sum_{n\geq 0}G_0(X^{ind}_{\alpha})\mathbbm{1}\{n\leq G_0<n+1\},
$$
where $\mathbbm{1}\{n\leq G_0<n+1\}$ denotes the usual indicator function of the event in parentheses, in this case
$$
\mathbbm{1}\{\omega\in\Omega: n\leq G_0(X^{ind}_{\alpha}(\omega))<n+1\}.
$$
By the Monotone Convergence theorem,
\begin{equation*}
\E\sum_{n\geq 0}G_0(X^{ind}_{\alpha})\mathbbm{1}\{n\leq G_0<n+1\}=\sum_{n\geq 0}\E G_0(X^{ind}_{\alpha})\mathbbm{1}\{n\leq G_0<n+1\}.
\end{equation*}
Since $X^{ind}_{\alpha}\arrowd x_0$, and by breaking each term for $n\geq1$ into two if necessary (by the convexity of $G_0$), from Portmanteau lemma we obtain
$$
\lim_{\substack{\alpha\rightarrow\alpha_0\\\alpha>\alpha_0}}\E G_0(X^{ind}_{\alpha})=\lim_{x\rightarrow x_0}G_0(x)=0
$$
and hence
$$
\lim_{\substack{\alpha\rightarrow\alpha_0\\\alpha>\alpha_0}}E(\alpha)=\lim_{\substack{\alpha\rightarrow\alpha_0\\\alpha>\alpha_0}}\E G(X^{ind}_{\alpha})=\lim_{x\rightarrow x_0}G(x)=\alpha_0,
$$
which concludes the argument.

Before stating the result, we note that the same $\P$-almost surely uniqueness as before holds for the independent case. Indeed, from the uniqueness of the pair $(X_\alpha^\dag,Y_\alpha^\dag)$ in a $\P$-almost surely sense for each $\alpha\in \I_G$, the fact that for every $\alpha\in \I_G$ there exists a unique $\alpha^\ddag(\alpha)\in \I_G$, $\alpha^\ddag(\alpha)<\alpha$, such that
$$
Y_\alpha^\ddag\eqd Y^\dag_{\alpha^\ddag(\alpha)},
$$
and the fact that $X_\alpha^\ddag$ is then uniquely characterized in distribution, and independent of $N$, allow us to conclude that, if in addition to the pair $(X_\alpha^\ddag,Y_\alpha^\ddag)$ satisfying
\begin{equation*}
\sup\{h(Y):Y\in\Sigma_{\bot}(N;G,\alpha)\}=\max\{h(Y):Y\in\Sigma_{\bot}(N;G,\alpha)\}=h(Y^\ddag_{\alpha})
\end{equation*}
with
$$
Y^\ddag_{\alpha}=\Sigma_{\bot}(X_\alpha^\ddag,N;G,\alpha)
$$
$\P$-almost surely, there exists other pair $(\hat{X}^\ddag_\alpha,\hat{Y}^\ddag_\alpha)$ such that
$$
\hat{Y}^\ddag_\alpha=\Sigma(\hat{X}^\ddag_\alpha,N;G,\alpha)
$$
$\P$-almost surely as well, and $h(\hat{Y}^\ddag_\alpha)=h(Y^\ddag_\alpha)$, then we have
$$
(\hat{X}^\ddag_\alpha,\hat{Y}^\ddag_\alpha)=(X^\ddag_\alpha,Y^\ddag_\alpha)\;\;\P\text{-almost surely too},
$$
for each $\alpha\in \I_G$.

In summary, we have thus proved the following corollary to Theorem \ref{thm1}.

\begin{corollary}\label{cor1}
Consider $N\in\N(\nll,\nuu)$ and cost function $G$. Assume the hyphoteses of Theorem \ref{thm1} are satisfied with maximum average cost $\alpha$ for each $\alpha\in \I_G$, so that we may write, as $\alpha$ ranges in $\I_G$,
\begin{equation*}
\sup\{h(Y):Y\in\Sigma(N;G,\alpha)\}=\max\{h(Y):Y\in\Sigma(N;G,\alpha)\}=h(Y_\alpha^\dag)
\end{equation*}
with
$$
Y_\alpha^\dag=\Sigma(X_\alpha^\dag,N;G,\alpha)\;\;\;\text{and}\;\;\;X_\alpha^\dag=g_\alpha(N).
$$
Also assume for each $\alpha\in \I_G$ the Fourier transform quotient
$$
\frac{\F[\Lambda_{Y^\dag_{\alpha}}]}{\F[\Lambda_N]}(f),
$$
defined initially in
$$
S_N\doteq\{f\in\R:\F[\Lambda_N](f)\neq 0\},
$$
has a positive definite continuous extension from $S_N$ to the whole of $\R$. Then, with $X_{\alpha}^{ind}$ denoting an independent of $N$ RV whose probability law is uniquely characterized by the Fourier transform of such positive definite continuous extension, we have that, for each $\alpha\in \I_G$ there exists a unique $\alpha^\ddag(\alpha)\in \I_G$, with $\alpha^\ddag(\alpha)<\alpha$, and such that
\begin{equation*}
\sup\{h(Y):Y\in\Sigma_{\bot}(N;G,\alpha)\}=\max\{h(Y):Y\in\Sigma_{\bot}(N;G,\alpha)\}=h(Y^\ddag_{\alpha})
\end{equation*}
with
$$
Y^\ddag_{\alpha}=\Sigma_{\bot}(X_\alpha^\ddag,N;G,\alpha),
$$
and where
$$
X^\ddag_\alpha\doteq X_{\alpha^\ddag{(\alpha)}}^{ind},
$$
that is, with the probability law of $X^\ddag_\alpha$, equivalently of $X_{\alpha^\ddag{(\alpha)}}^{ind}$, being a corresponding capacity-achieving input distribution for each $\alpha\in \I_G$,
$$
C(N;G,\alpha)=h(Y^\ddag_{\alpha})-h(N).
$$
Moreover, for each $\alpha\in \I_G$
$$
Y^\ddag_{\alpha}\eqd Y^\dag_{\alpha^\ddag{(\alpha)}},
$$
and therefore we may alternatively write
$$
C(N;G,\alpha)=h(Y^\dag_{\alpha^\ddag{(\alpha)}})-h(N).
$$
In addition, each pair $(X^\ddag_\alpha,Y^\ddag_\alpha)$ as above is unique in a $\P$-almost surely sense. Finally, if the positive definite continuous extension referred to above turns out to belong to the space of integrable functions $L^1(\R,\B(\R),\ell)$, then a corresponding probability density function (w.r.t. Lebesgue measure) may be associated to each capacity-achieving input distribution, by the corresponding standard inverse Fourier transform formula.
\end{corollary}

Before concluding the section, we make the following remarks.

\begin{remark}\label{entropygainwithfeedback}
Note since $\alpha>\alpha^\ddag(\alpha)$ and $h(Y^\dag_{\alpha})$ is strictly increasing in $\alpha\in \I_G$, then for each $\alpha\in \I_G$ the difference
$$
h(Y^\dag_{\alpha})-h(Y^\ddag_{\alpha})=h(Y^\dag_{\alpha})-h(Y^\dag_{\alpha^\ddag{(\alpha)}})>0
$$
represents a bound on how much entropy at the output may be gained when going from the independent of $N$ input case to the fully dependent one, so providing an estimate for the gain of entropy at the output by the use of feedback. This subject will be pursued further in Section \ref{generalcapacity}.
\end{remark}

\begin{remark}
Note if the Fourier transform quotient
$$
\frac{\F[\Lambda_{Y^\dag_{\alpha}}]}{\F[\Lambda_N]}(f),\;\alpha\in \I_G,
$$
fails to have a positive definite continuous extension from $S_N$ to the whole of $\R$ for some $\alpha^*\in \I_G$, then by Bochner's theorem we know $X_{\alpha^*}^{ind}$ fails to exist, and therefore there will be some $\hat{\alpha}(\alpha^*)\in \I_G$ such that the corresponding capacity
$$
C(N;G,\hat{\alpha}(\alpha^*))
$$
will not be achievable. This case will be properly handled in Section \ref{existence}, Theorem \ref{thm3}.
\end{remark}

\section{Amplitude Constraints}
\label{amplitude}

We briefly discuss in this section on the case where in addition to average cost constraints, amplitude constraints are imposed in the input.

Consider the problem of finding
$$
\sup\{h(Y):Y\in\Sigma(N;G,\beta,\overline{(\xll,\xuu)})\}.
$$
Motivated from the analysis in the previous section, we are lead to looking for an optimum differential entropy at the output occurring, as before, at $X=g(N)$ for some appropriate sufficiently regular function $g$. The characterization of such a $g$ via the calculus of variations is similar to that developed in the previous context, being now necessary, however, to take into account such additional amplitude constraint. Consider first the case where both $c,d\in\R$, i.e., with $-\infty<c<d<\infty$. Then, the constraint
$$
c\leq x\leq d\;\;\text{with $x=g(n)$, $n\in\overline{(\nll,\nuu)}$,}
$$
can be taken into account by introducing in the calculus of variations formulation the \textit{holonomic} constraint (see \cite{BVB2006} for the terminology)
$$
(d-g(n))(g(n)-c)-z^2(n)=0,\;\;\;n\in\overline{(\nll,\nuu)},
$$
where we have also introduced the $n$-dependent variable $z=z(n)$ \footnote{As standard, we abuse notation and write the same $z$ for the mapping $n\longmapsto z(n)$ and the corresponding dependent variable.}. Note then
$$
(d-g(n))(g(n)-c)\geq 0,\;\;\;n\in\overline{(\nll,\nuu)}.
$$
It is easy to see that the corresponding Euler-Lagrange equation in $(\nll,\nuu)$, or more precisely its first integral as in equation (\ref{firstintegral}), now becomes, upon including this additional constraint,
\begin{equation*}
p_N(n)=C(1+g'(n))e^{-\lambda\left(\int G'(g(n))dn+G(g(n))\right)}e^{\int\frac{\mu(n)(1+g'(n))}{p_N(n)}(2g(n)-(c+d))dn},\;\;\;n\in(\nll,\nuu),
\end{equation*}
where we use the same notation as before for $\int(\cdot)dn$ (that is, denoting a corresponding primitive function) and where we have introduced the associated Lagrange multiplier function
$$
\mu:\overline{(\nll,\nuu)}\rightarrow\R
$$
satisfying
$$
\mu(n)z(n)=0,\;\;\;n\in\overline{(\nll,\nuu)}.
$$
Analogously, for $c=-\infty$ and $d<\infty$ (i.e., $x\leq d$) we have
$$
(d-g(n))-z^2(n)=0,\;\;\;n\in\overline{(\nll,\nuu)},
$$
and therefore
\begin{equation*}
p_N(n)=C(1+g'(n))e^{-\lambda\left(\int G'(g(n))dn+G(g(n))\right)}e^{\int\frac{\mu(n)(1+g'(n))}{p_N(n)}dn},\;\;\;n\in(\nll,\nuu),
\end{equation*}
and, for $c>-\infty$ and $d=\infty$ (i.e., $x\geq c$)
$$
(g(n)-c)-z^2(n)=0,\;\;\;n\in\overline{(\nll,\nuu)},
$$
and therefore
\begin{equation}\label{ELAC}
p_N(n)=C(1+g'(n))e^{-\lambda\left(\int G'(g(n))dn+G(g(n))\right)}e^{-\int\frac{\mu(n)(1+g'(n))}{p_N(n)}dn},\;\;\;n\in(\nll,\nuu).
\end{equation}

Regarding now as to appropriate boundary conditions to be imposed on $g$ when further integrating the corresponding equation above, note the natural boundary conditions
$$
\lim_{\substack{n\rightarrow\nll\\n>\nll}}\frac{p_N(n)}{1+g'(n)}=\lim_{\substack{n\rightarrow\nuu\\n<\nuu}}\frac{p_N(n)}{1+g'(n)}=0
$$
may not be compatible with the imposed amplitude constraint. However, and since the natural boundary conditions indeed represent the appropriate boundary conditions to be imposed when considering a not-fixed-end-point variational problem as we do now, we impose those conditions, from the left (i.e., $n\rightarrow\nuu,n<\nuu$), from the right (i.e., $n\rightarrow\nll,n>\nll$), or both whenever compatible with the amplitude constraint under consideration. When at least one of them cannot be applied, we may solve for the remaining unknown constants of integration appearing in $g$ by performing an standard optimum search for those constants maximizing the corresponding differential entropy at the output. The problem can so be reduced to an standard optimization problem over real-valued parameters. Equivalently, we are then considering the corresponding one- or two-fixed-end-point variational problem, and performing a maximization over such end points.

If an optimum solution to the above problem indeed exists, then in addition to provide us with the optimum solution for
$$
\sup\{h(Y):Y\in\Sigma(N;G,\beta,\overline{(\xll,\xuu)})\},
$$
it can also be exploited, by the same sort of techniques already described in the previous section, to approach the related channel capacity and capacity-achieving input distribution problem enclosed in
$$
\sup\{h(Y):Y\in\Sigma_\bot(N;G,\beta,\overline{(\xll,\xuu)})\},
$$
of course taking into account the corresponding restrictions imposed in the average cost constraint parameter $\beta$ by the amplitude constraint
$$
X\in\overline{(\xll,\xuu)},\;\;\P\text{-almost surely}.
$$

Specific cases in the next section will exemplify the described methodology.

\section{Examples}
\label{examples}

In this section we present several examples on the use of the results so far provided in the paper. We consider the cases of normally, exponentially and uniformly distributed noise, as well as important and commonly used cost functions for the related average cost constraints. Amplitude constraints are also considered. We begin with the standard case of normally distributed noise. Of course, this case can be solved by almost pure inspection. It is useful in showing, however, and as a first example, how some known classical results are recovered from our approach. In that direction, though some results are guessable a priori, the main objective of the examples provided is to describe the use of our general setting and results as a tool for a systematic approach to the problems of maximum entropy, capacity and capacity-achieving distribution computations.

In some of the examples we will consider a linear cost function $G$. Though such function cannot deliver finite optimums (see Remark \ref{linear} in Section \ref{setting}), it is considered with the purpose of showing how the unboundedness appears from our equations. In all those examples we will then include amplitude constraints, turning the optimums finite.

\subsection{Normally Distributed Noise}

Let $N\sim\mathcal{N}(0,\sigma_N^2)$, i.e., with $N$ a Gaussian RV, mean zero, variance $\sigma_N^2$ with $\sigma_N\in(0,\infty)$, density
$$
p_N(n)=\frac{1}{\sqrt{2\pi}\sigma_N}e^{-\frac{n^2}{2\sigma_N^2}},\;\;\;n\in\overline{(\nll,\nuu)}=\R,
$$
and consider the cost function $G$ given by
$$
G(x)\doteq x^2,\;\;\;x\in\R,
$$
and the maximum average cost $\beta\in(0,\infty)$ (the case $\beta=0$ being trivial).

Since $G(x)=x^2$ and therefore $G'(x)=2x$, from equation (\ref{firstintegral}) we have, upon using that $G(g(n))=g^2(n)$ and $G'(g(n))=2g(n)$,
$$
p_N(n)=C(1+g'(n))e^{-\lambda\left(2\int g(n)dn+g^2(n)\right)},\;\;\;n\in\R.
$$
In light of the preceding equation, it is clear we must look for a solution of the form
$$
g(n)=\theta n+\kappa,\;\;\;n\in\R,
$$
with $\theta$ and $\kappa$ real constants to be determined, the same as $\lambda\geq 0$ and $C>0$. By direct replacement we find
$$
\frac{1}{\sqrt{2\pi}\sigma_N}e^{-\frac{n^2}{2\sigma_N^2}}=C(1+\theta)e^{-\lambda\left(\theta(1+\theta)n^2+2\kappa(1+\theta)n+\kappa^2\right)},
$$
and hence, since $\theta=g'(n)>-1$,
$$
\kappa=0.
$$
Thus, we have
$$
\frac{1}{\sqrt{2\pi}\sigma_N}e^{-\frac{n^2}{2\sigma_N^2}}=C(1+\theta)e^{-\lambda\theta(1+\theta)n^2}
$$
and therefore
\begin{equation}\label{GC1}
\frac{1}{\sqrt{2\pi}\sigma_N}=C(1+\theta)
\end{equation}
and
\begin{equation}\label{GC2}
\frac{1}{2\sigma_N^2}=\lambda\theta(1+\theta).
\end{equation}
Also, from the isoperimetric constraint,
\begin{equation}\label{GC3}
\beta=\E G(g(N))=\int_\R g^2(n)p_N(n)\ell(dn)=\theta^2\int_\R n^2p_N(n)\ell(dn)=\theta^2\sigma_N^2
\end{equation}
where we have used that $\kappa=0$ and $\E N=0$. Since $\lambda\geq 0$ and $\theta>-1$, from (\ref{GC2}) we conclude $\lambda>0$ and $\theta>0$. In particular, from (\ref{GC3}) we then obtain
$$
\theta=\frac{\sqrt{\beta}}{\sigma_N}.
$$
Direct replacement of $\theta$ above in (\ref{GC1}) and (\ref{GC2}) finally shows that
$$
C=\frac{1}{\sqrt{2\pi}(\sigma_N+\sqrt{\beta})}
$$
and
$$
\lambda=\frac{1}{2\sqrt{\beta}(\sigma_N+\sqrt{\beta})}.
$$
Note the (natural) boundary conditions
$$
\lim_{n\rightarrow\infty}\frac{p_N(n)}{1+\theta}=\lim_{n\rightarrow-\infty}\frac{p_N(n)}{1+\theta}=0
$$
are automatically satisfied.

We then have
$$
\sup\{h(Y):Y\in\Sigma(N;G,\beta)\}=h(Y^\dag)
$$
with
$$
Y^\dag=\Sigma(X^\dag,N;G,\beta)
$$
and
$$
X^\dag=\frac{\sqrt{\beta}}{\sigma_N}N,
$$
that is,
$$
X^\dag\sim\mathcal{N}(0,\beta),\;\;\;Y^\dag\sim\mathcal{N}(0,(\sigma_N+\sqrt{\beta})^2)
$$
and
$$
h(Y^\dag)=\frac{1}{2}\ln 2\pi e(\sigma_N+\sqrt{\beta})^2.
$$
Moreover, we have
$$
C(N;G,\beta)=h(Y^\ddag)-h(N)=\frac{1}{2}\ln\frac{\sigma_N^2+\beta}{\sigma_N^2}
$$
with
$$
h(Y^\ddag)=\sup\{h(Y):Y\in\Sigma_\bot(N;G,\beta)\}
$$
and
$$
Y^\ddag=\Sigma_\bot(X^\ddag,N;G,\beta),
$$
and with the probability law of $X^\ddag$ being associated to
$$
\overline{\X}_\beta(f)=\X_\beta(f)=e^{-2\pi^2\beta f^2},\;\;\;f\in\R,
$$
that is,
$$
X^\ddag\sim\mathcal{N}(0,\beta)
$$
and therefore
$$
Y^\ddag\sim\mathcal{N}(0,\sigma_N^2+\beta).
$$

\subsection{Exponentially Distributed Noise}

Let $N\sim\exp(\mu_N^{-1})$, i.e., with $N$ an exponentially distributed RV, parameter $\mu_N^{-1}\in(0,\infty)$, that is, mean $\mu_N$ and density
$$
p_N(n)=\frac{1}{\mu_N}e^{-\frac{1}{\mu_N}n},\;\;\;n\in\overline{(\nll,\nuu)}=[0,\infty).
$$
We discuss the cases with average cost constraints and with average cost and amplitude constraints separately. In both of them we consider the cost function $G$ given by
$$
G(x)\doteq x,\;\;\;x\in\R.
$$

\subsubsection{Average Cost Constraints}

Consider a maximum average cost $\beta\in\R$. Since $G(x)=x$ and therefore $G'(x)=1$, from equation (\ref{firstintegral}) we have, upon using that $G(g(n))=g(n)$ and $G'(g(n))=1$,
$$
p_N(n)=C(1+g'(n))e^{-\lambda(n+g(n))},\;\;\;n\in(0,\infty).
$$
It is clear that a solution to the previous equation must be of the form
$$
g(n)=\theta n+\kappa,\;\;\;n\in(0,\infty),
$$
with $\theta>-1$ and $\kappa$ real constants. However, and since then $g'(n)=\theta$, the (natural) boundary condition at $0+$,
$$
\lim_{\substack{n\rightarrow 0\\n>0}}\frac{p_N(n)}{1+\theta}=0,
$$
cannot be satisfied being
$$
\lim_{\substack{n\rightarrow 0\\n>0}}p_N(n)=p_N(0)=\frac{1}{\mu_N}>0.
$$

On the other hand, note with $g$ as above
$$
\E G(g(N))=\int_{(0,\infty)}\frac{(\theta n+\kappa)}{\mu_N}e^{-\frac{n}{\mu_N}}\ell(dn)=\theta\mu_N+\kappa,
$$
and therefore the isoperimetric constraint reads
\begin{equation}\label{EqT}
\theta\mu_N+\kappa=\beta,\;\;\;\text{i.e.,}\;\;\;\theta=\frac{\beta-\kappa}{\mu_N}.
\end{equation}
Since
$$
\lim_{\theta\rightarrow\infty}\frac{p_N(0)}{1+\theta}=0
$$
and from (\ref{EqT}) we have that $\theta$ does indeed increase to $\infty$ as $\kappa$ decreases to $-\infty$, we are lead to consider the sequences of RVs $\{X_i\}_{i\in\mathbb{N}}$ and $\{Y_i\}_{i\in\mathbb{N}}$ given by \footnote{$\mathbb{N}\doteq\{1,2,\ldots\}$.}
$$
X_i\doteq\frac{\beta+i}{\mu_N}N-i
$$
and
$$
Y_i\doteq X_i+N=\Sigma(X_i,N;G,\beta).
$$
Then, it is straightforward to check that
$$
h(Y_i)=\ln e|\beta+\mu_N+i|
$$
and therefore
$$
\lim_{i\rightarrow\infty}h(Y_i)=\infty.
$$
Thus, we have as expected (see Remark \ref{linear} in Section \ref{setting})
$$
\sup\{h(Y):Y\in\Sigma(N;G,\beta)\}=\infty,
$$
i.e., no finite extremum exists. Note also that
$$
\F[\Lambda_{Y_i}](f)=\frac{e^{j2\pi fi}}{1+j2\pi f(\beta+\mu_N+i)},\;\;\;f\in\R,
$$
and
$$
\F[\Lambda_{N}](f)=\frac{1}{1+j2\pi f\mu_N},\;\;\;f\in\R,
$$
and thus
$$
\frac{\F[\Lambda_{Y_i}]}{\F[\Lambda_{N}]}(f)=\frac{1+j2\pi f\mu_N}{1+j2\pi f(\beta+\mu_N+i)}e^{j2\pi fi},\;\;\;f\in\R,
$$
giving then the corresponding generalized density (in a tempered distributions sense) and for sufficiently large $i$ (such that $\beta+\mu_N+i>0$)
$$
\frac{(\beta+i)e^{-\frac{x+i}{\beta+\mu_N+i}}}{(\beta+\mu_N+i)^2}u(x+i)+\frac{\mu_N}{\beta+\mu_N+i}\delta(x+i),\;\;\;x\in\R,
$$
with $u(t)=1$ if $t\geq 0$, $u(t)=0$ otherwise, and with $\delta$ the Dirac delta (generalized) function. Since the mean value associated to the previous density is also $\E X_i=\beta$, we conclude that
$$
\sup\{h(Y):Y\in\Sigma_\bot(N;G,\beta)\}=\infty
$$
and therefore
$$
C(N;G,\beta)=\infty,
$$
as expected too (again, see Remark \ref{linear} in Section \ref{setting}).

\subsubsection{Average Cost and Amplitude Constraints}

We consider the case where the amplitude constraint
$$
X\in[0,\infty),\;\;\P\text{-almost surely},
$$
is also brought into the picture \footnote{The case when $X$ is restricted to lie in $[c,\infty)$ for any fixed $c\in\R$ does not add any generality since it is recovered from the present one by just shifting and using that $h(Z+c)=h(Z)$.}, along with the maximum average cost $\beta\in(0,\infty)$ (the case $\beta=0$ being now trivial).

In light of the previous case, and since
$$
g(n)=\theta n+\kappa\geq 0,\;\;\;n\in(0,\infty),
$$
when $\theta\geq 0$ and $\kappa\geq 0$, we are lead to consider equation (\ref{ELAC}) with $\mu\equiv 0$, i.e., so retaining the same form for $g$ as above. Direct replacement then shows that
\begin{eqnarray*}
\frac{1}{\mu_N}e^{-\frac{1}{\mu_N}n}&=&C(1+\theta)e^{-\lambda\left((1+\theta)n+\kappa\right)}\\&=&C(1+\theta)e^{-\lambda\kappa}e^{-\lambda(1+\theta)n}
\end{eqnarray*}
and therefore
\begin{equation}\label{EC1}
\frac{1}{\mu_N}=C(1+\theta)e^{-\lambda\kappa}
\end{equation}
and
\begin{equation}\label{EC2}
\frac{1}{\mu_N}=\lambda(1+\theta).
\end{equation}
As before, the isoperimetric constraint gives
\begin{equation}\label{EC3}
\theta=\frac{\beta-\kappa}{\mu_N}.
\end{equation}
Equations (\ref{EC1}), (\ref{EC2}) and (\ref{EC3}) give us three equations for the four unknowns $C,\lambda,\theta$ and $\kappa$. On the other hand, the (natural) boundary condition at $\infty$
$$
\lim_{n\rightarrow\infty}\frac{p_N(n)}{1+\theta}=0
$$
is trivially satisfied, being impossible to fulfill, as mentioned in the previous case, its counterpart at $0+$. However, since $g$ is a straight line, we may therefore exploit the equations above to reduce the problem to an optimal search over $\kappa\geq 0$, requiring of course for the corresponding slopes to be non negative as well, i.e., considering as admissible the values $\kappa\in[0,\beta]$. Note from (\ref{EC2}) and (\ref{EC1}), $\theta\geq 0$ in particular implies $\lambda>0$ and $C>0$. In fact,
$$
\lambda=\frac{1}{\mu_N+\beta-\kappa}\;\;\;\text{and}\;\;\;C=\frac{e^{\frac{\kappa}{\mu_N+\beta-\kappa}}}{\mu_N+\beta-\kappa}.
$$

It is straightforward to see that the density at the output $p_Y$ is given by
$$
p_Y(y)=\frac{1}{\mu_N(1+\theta)}e^{-\frac{y-\kappa}{\mu_N(1+\theta)}},\;\;\;y\geq\kappa,
$$
$0$ otherwise, and therefore its associated differential entropy
$$
h(Y)=\ln e\mu_N(1+\theta)
$$
which is strictly increasing in $\theta$. Thus, the optimum takes place at $\kappa=0$ where $\theta$ achieves its maximum allowable value
$$
\theta=\frac{\beta}{\mu_N}.
$$

We then have
$$
\sup\{h(Y):Y\in\Sigma(N;G,\beta)\}=h(Y^\dag)
$$
with
$$
Y^\dag=\Sigma(X^\dag,N;G,\beta)
$$
and
$$
X^\dag=\frac{\beta}{\mu_N}N,
$$
that is,
$$
X^\dag\sim\exp(\beta^{-1}),\;\;\;Y^\dag\sim\exp((\mu_N+\beta)^{-1})
$$
and
$$
h(Y^\dag)=\ln e(\mu_N+\beta).
$$
Moreover, and restricting attention to $\beta\in G([0,\infty))$, we have
$$
C(N;G,\beta)=h(Y^\ddag)-h(N)=\ln\frac{\mu_N+\beta}{\mu_N}
$$
with
$$
h(Y^\ddag)=\sup\{h(Y):Y\in\Sigma_\bot(N;G,\beta)\}
$$
and
$$
Y^\ddag=\Sigma_\bot(X^\ddag,N;G,\beta),
$$
and with the probability law of $X^\ddag$ being associated to
$$
\overline{\X}_\beta(f)=\X_\beta(f)=\frac{1+j2\pi f\mu_N}{1+j2\pi f (\mu_N+\beta)},\;\;\;f\in\R,
$$
that is, with generalized density
\begin{equation}\label{CADEC}
\frac{\beta e^{-\frac{x}{\beta+\mu_N}}}{(\beta+\mu_N)^2}u(x)+\frac{\mu_N}{\beta+\mu_N}\delta(x),\;\;\;x\in\R,
\end{equation}
supported accordingly in $[0,\infty)$, and therefore in this case we obtain, due to the amplitude constraint that has been brought into the picture, that
$$
Y^\ddag\eqd Y^\dag,
$$
i.e.,
\begin{equation}\label{CADEC2}
Y^\ddag\sim\exp((\mu_N+\beta)^{-1}).
\end{equation}
Note the amplitude constraint turns channel capacity finite with, in the jargon of Section \ref{main2} and since $G$ is linear and therefore the average cost constraint is equally satisfied by a jointly distributed with $N$ input RV as well as by an independent of $N$ input RV,
$$
\beta^\ddag(\beta)=\beta.
$$

Of course, the capacity-achieving input distribution associated to (\ref{CADEC}), as well as its corresponding output distribution in (\ref{CADEC2}), are known in the literature and they correspond to the also known fact that a non-negative RV with mean $\mu_N+\beta$ cannot have a higher differential entropy than in the exponentially distributed case, that is, $\ln e(\mu_N+\beta)$ (see \cite{SV1996}).

\subsection{Uniformly Distributed Noise}

Let $N\sim\mathcal{U}([-\vartheta,\vartheta])$, i.e., with $N$ a uniformly distributed RV over $[-\vartheta,\vartheta]$, $0<\vartheta<\infty$, density
$$
p_N(n)=\frac{1}{2\vartheta},\;\;\;n\in\overline{(\nll,\nuu)}=[-\vartheta,\vartheta].
$$
We discuss the cases with average cost constraints and with average cost and amplitude constraints separately. In both of them we consider, as in the exponentially distributed noise case, the cost function $G$ given by
$$
G(x)\doteq x,\;\;\;x\in\R.
$$

\subsubsection{Average Cost Constraints}

Consider a maximum average cost $\beta\in\R$. Proceeding as before, from equation (\ref{firstintegral}) we have
$$
p_N(n)=C(1+g'(n))e^{-\lambda(n+g(n))},\;\;\;n\in(-\vartheta,\vartheta),
$$
i.e.,
$$
\frac{1}{2\vartheta}=C(1+g'(n))e^{-\lambda(n+g(n))},\;\;\;n\in(-\vartheta,\vartheta).
$$
Alternatively, and in terms of
\begin{equation*}
y=f(n)=g(n)+n
\end{equation*}
(see equation (\ref{rel2}) and Remark \ref{ELAlt}), we have, upon using that $G'\equiv 1$,
$$
f^{-1'}(y)=2\vartheta Ce^{-\lambda y}
$$
with $y$ ranging in some appropriate interval. Assume by the time being $\lambda>0$. Then
$$
f^{-1}(y)=-\frac{2\vartheta C}{\lambda}e^{-\lambda y}+\kappa
$$
with $\kappa$ a real constant, and therefore
$$
g(n)=-n+f(n)=-n-\frac{1}{\lambda}\ln\frac{\lambda(\kappa-n)}{2\vartheta C}
$$
which in turn implies
$$
g'(n)=-1+\frac{1}{\lambda(\kappa-n)}.
$$
From the (natural) boundary condition at $\vartheta-$,
$$
\lim_{\substack{n\rightarrow \vartheta\\n<\vartheta}}\frac{p_N(n)}{1+g'(n)}=0,
$$
we conclude $\kappa=\vartheta$, and thus
$$
g'(n)=-1+\frac{1}{\lambda(\vartheta-n)}.
$$
Hence
\begin{equation}\label{NLg}
\lim_{\substack{n\rightarrow -\vartheta\\n>-\vartheta}}\frac{p_N(n)}{1+g'(n)}=\lambda>0,
\end{equation}
showing then the (natural) boundary condition at $(-\vartheta)+$ cannot be satisfied. Therefore, we look for a solution $g$ with $\lambda=0$. We then have
$$
f^{-1'}(y)=2\vartheta C
$$
from where
$$
g(n)=\frac{1-2\vartheta C}{2\vartheta C}n-\frac{\kappa}{2\vartheta C}
$$
and therefore the (natural) boundary conditions
$$
\lim_{\substack{n\rightarrow -\vartheta\\n>-\vartheta}}\frac{p_N(n)}{1+g'(n)}=\lim_{\substack{n\rightarrow \vartheta\\n<\vartheta}}\frac{p_N(n)}{1+g'(n)}=0
$$
cannot be satisfied either. However, since
$$
g'(n)=\frac{1-2\vartheta C}{2\vartheta C}
$$
and
$$
\lim_{\substack{C\rightarrow 0\\C>0}}\frac{1-2\vartheta C}{2\vartheta C}=\infty,
$$
and also from the isoperimetric constraint
$$
\E G(g(N))=\int_{(-\vartheta,\vartheta)}g(n)p_N(n)\ell(dn)=-\frac{\kappa}{2\vartheta C}=\beta,
$$
we are lead to consider the sequences of RVs $\{X_i\}_{i\in\mathbb{N}}$ and $\{Y_i\}_{i\in\mathbb{N}}$ given by \footnote{We may alternatively consider the previous $g$ and, from equation (\ref{NLg}), construct sequences with $\lambda_i$ decreasing to $0$ as $i$ increases to infinite.}
$$
X_i\doteq\frac{1-2\vartheta\frac{1}{i}}{2\vartheta\frac{1}{i}}N+\beta
$$
and
$$
Y_i\doteq X_i+N=\Sigma(X_i,N;G,\beta).
$$
Then, it is straightforward to check that each
$$
Y_i\sim\mathcal{U}([\beta-i2^{-1},\beta+i2^{-1}])
$$
and therefore
$$
h(Y_i)=\ln i,
$$
from where
$$
\lim_{i\rightarrow\infty}h(Y_i)=\infty.
$$
Thus, we have as expected (see Remark \ref{linear} in Section \ref{setting})
$$
\sup\{h(Y):Y\in\Sigma(N;G,\beta)\}=\infty,
$$
i.e., no finite extremum exists. Note also that, and being direct to verify,
$$
\frac{\F[\Lambda_{Y_i}]}{\F[\Lambda_{N}]}(f)=e^{-j2\pi f\beta}\frac{2\vartheta}{i}\frac{\sin \pi fi}{\sin 2\pi f\vartheta},\;\;\;f\neq\frac{k}{2\vartheta},\;\;\;k\in\mathbb{Z},
$$
where $\mathbb{Z}\doteq\{0,\pm 1,\pm 2,\ldots\}$. We may also set
$$
\frac{\F[\Lambda_{Y_i}]}{\F[\Lambda_{N}]}(0)\doteq 1=\lim_{f\rightarrow 0}\frac{\F[\Lambda_{Y_i}]}{\F[\Lambda_{N}]}(f).
$$
By appropriate re-scaling if necessary, and by the denseness of the rational numbers $\mathbb{Q}$ in $\R$ and the fact that $h(cZ)=h(Z)+\ln|c|$ for the differential entropy of $cZ$ and $c\in\R$, we may assume in what follows and w.l.o.g. $\vartheta\in\mathbb{Q}$, i.e., being in addition strictly positive, $\vartheta\in\mathbb{Q}\cap(0,\infty)$. Then, there exist strictly positive integers $q$ and $r$ such that
$$
\vartheta=\frac{q}{r}.
$$
Thus
$$
\vartheta=\frac{2q}{2r}\doteq\frac{i_1}{2r}=\frac{i_1k}{2rk}\doteq\frac{i_k}{2r_k},\;\;\;k\in\mathbb{N},
$$
and therefore there exists a subsequence $\{Y_{i_k}\}_{k\in\mathbb{N}}\subseteq\{Y_i\}_{i\in\mathbb{N}}$ such that
$$
\frac{\F[\Lambda_{Y_{i_k}}]}{\F[\Lambda_{N}]}(f)=e^{-j2\pi f\beta}\frac{2\vartheta}{i_k}\frac{\sin \pi fi_k}{\sin \pi f\frac{i_k}{r_k}},\;\;\;f\neq 0,
$$
retaining of course the same value at $f=0$ as before. That this last expression is not only well defined but also continuous can be easily verified. Indeed, by standard trigonometric expansions,
$$
\sin \pi fi_k=\sin r_k\pi f\frac{i_k}{r_k}=\sum_{l=1}^{r_k-1}\binom{r_k}{1}\psi_{l,0}(f)
$$
with
\begin{equation*}
\psi_{l,m}(f)\doteq\cos^l\left(\pi f\frac{i_k}{r_k}\right)\sin^{r_k-l-m}\left(\pi f\frac{i_k}{r_k}\right)\sin\left(\frac{\pi}{2}(r_k-l)\right),
\end{equation*}
and therefore
$$
\frac{\sin \pi fi_k}{\sin \pi f\frac{i_k}{r_k}}=\sum_{l=1}^{r_k-1}\binom{r_k}{1}\psi_{l,1}(f)
$$
which is a trigonometric polynomial in $f$. Moreover, by expressing the trigonometric functions appearing in $\psi_{l,1}$ in terms of complex exponentials, and upon taking corresponding inverse Fourier transformations, it can be easily seen that in fact the function
\begin{displaymath}
\frac{\F[\Lambda_{Y_{i_k}}]}{\F[\Lambda_{N}]}(f)= \left\{ \begin{array}{ll}
e^{-j2\pi f\beta}\frac{2\vartheta}{i_k}\frac{\sin \pi fi_k}{\sin \pi f\frac{i_k}{r_k}} & \textrm{if $f\neq 0$}\\
1 & \textrm{if $f=0$}
\end{array} \right.
\end{displaymath}
is a positive definite function corresponding to a pure atomic probability distribution over $\R$ (i.e., concentrating its total probability mass in a discrete set, which in fact is obviously finite). Since the mean value associated to the previous distribution is also $\E X_{i_k}=\beta$, we finally conclude that
$$
\sup\{h(Y):Y\in\Sigma_\bot(N;G,\beta)\}=\infty
$$
and therefore
$$
C(N;G,\beta)=\infty,
$$
as expected too (again, see Remark \ref{linear} in Section \ref{setting}).

\subsubsection{Average Cost and Amplitude Constraints}

As for the exponentially distributed noise case, we now consider the case where the amplitude constraint
$$
X\in[0,\infty),\;\;\P\text{-almost surely},
$$
is also brought into the picture, along with the maximum average cost $\beta\in(0,\infty)$ (the case when $\beta=0$ being now correspondingly trivial).

In light of the previous case, we are lead to consider firstly equation (\ref{ELAC}) with $\mu\equiv 0$ and assuming $\lambda>0$, i.e., with $g$ of the form
\begin{equation}\label{NNg}
g(n)=-n-\frac{1}{\lambda}\ln\frac{\lambda(\kappa-n)}{2\vartheta C},\;\;\;n\in(-\vartheta,\vartheta),
\end{equation}
which can be clearly made non-negative. As we will see, this will be the case when
$$
\beta\geq\vartheta\;\;\;\text{and}\;\;\;(1-2e^{-1})\vartheta\leq\beta<\vartheta.
$$
The case when
$$
0<\beta<(1-2e^{-1})\vartheta
$$
will require a non-identically null function $\mu$. Each of the three cases above will now be considered separately.

\paragraph{Case $\beta\geq\vartheta$}

As before, and from equation (\ref{NNg}) and the (natural) boundary condition at $\vartheta-$ we conclude $\kappa=\vartheta$, and thus
$$
g(n)=-n-\frac{1}{\lambda}\ln\frac{\lambda(\vartheta-n)}{2\vartheta C}.
$$
Note then $g(n)$ increases to infinity as $n$ approaches $\vartheta$ from below, the same as
$$
g'(n)=-1+\frac{1}{\lambda(\vartheta-n)}.
$$
Note also
$$
\lim_{\substack{n\rightarrow -\vartheta\\n>-\vartheta}}g'(n)=-1+\frac{1}{2\lambda\vartheta}\geq 0,
$$
and in fact
$$
g'(n)>0,\;\;\;n\in(-\vartheta,\vartheta),
$$
when
$$
\lambda\leq\frac{1}{2\vartheta}.
$$
Since
$$
\E G(g(N))=\int_{(-\vartheta,\vartheta)}g(n)p_N(n)\ell(dn)=\frac{1}{\lambda}\left(1-\ln\frac{\lambda}{C}\right),
$$
from the isoperimetric constraint we obtain
\begin{equation}\label{UNE1}
\ln\frac{\lambda}{C}=1-\lambda\beta\;\;\;\text{or, equivalently,}\;\;\;C=\lambda e^{\lambda\beta-1}.
\end{equation}
As in the preceding case, the (natural) boundary condition at $(-\vartheta)+$ cannot be satisfied. However, since from the previous equation we may express, say, $C$ in terms of $\lambda$, we may reduce the problem to a one-dimensional optimal search treating then $\lambda$ as a parameter.
Towards that direction, we set
$$
x_{(-\vartheta)}\doteq\lim_{\substack{n\rightarrow -\vartheta\\n>-\vartheta}}g(n)=\vartheta-\frac{1}{\lambda}\ln\frac{\lambda}{C},
$$
i.e., and from equation (\ref{UNE1}),
$$
x_{(-\vartheta)}=\vartheta+\beta-\frac{1}{\lambda},
$$
and, correspondingly,
$$
y_{(-\vartheta)}\doteq x_{(-\vartheta)}-\vartheta=\beta-\frac{1}{\lambda}.
$$
It is straightforward to see that the density at the output $p_Y$ is given by
\begin{equation}\label{ODen}
p_Y(y)=Ce^{-\lambda y},\;\;\;y\geq y_{(-\vartheta)},
\end{equation}
$0$ otherwise, and therefore, upon using again equation (\ref{UNE1}), we have for the corresponding differential entropy
$$
h(Y)=1-\ln\lambda,
$$
which is strictly decreasing in $\lambda>0$. Thus, the maximum entropy at the output is so associated to the minimum such attainable $\lambda$, subject of course to the constraints $x_{(-\vartheta)}\geq 0$ and $g(n)\geq 0$, $n\in(-\vartheta,\vartheta)$. Note since $\lambda>0$, $x_{(-\vartheta)}\geq 0$ and
$$
x_{(-\vartheta)}=\vartheta+\beta-\frac{1}{\lambda}\;\;\;\text{or, equivalently,}\;\;\;\lambda=\frac{1}{\vartheta+\beta-x_{(-\vartheta)}},
$$
we in fact have
$$
\lambda\in[(\beta+\vartheta)^{-1},\infty)\;\;\;\text{and}\;\;\;x_{(-\vartheta)}\in[0,\beta+\vartheta).
$$
Therefore, since by setting
$$
\lambda=\frac{1}{\beta+\vartheta}
$$
we have
$$
x_{(-\vartheta)}=0
$$
and, moreover,
$$
\lambda\leq\frac{1}{2\vartheta}
$$
from where
$$
g'(n)>0,\;\;\;n\in(-\vartheta,\vartheta),
$$
we conclude
\begin{equation}\label{SE}
\sup\{h(Y):Y\in\Sigma(N;G,\beta)\}=h(Y^\dag)=\ln e(\beta+\vartheta)
\end{equation}
where
$$
Y^\dag=\Sigma(X^\dag,N;G,\beta),
$$
$$
X^\dag=-N-\frac{1}{\lambda}\ln\frac{\lambda(\vartheta-N)}{2\vartheta C},
$$
$$
\lambda=\frac{1}{\beta+\vartheta}\;\;\;\text{and}\;\;\;C=\frac{e^{-\frac{\vartheta}{\beta+\vartheta}}}{\beta+\vartheta}.
$$
Note since (see equation (\ref{ODen}))
$$
p_{Y^\dag}(y)=\frac{1}{\beta+\vartheta}e^{-\frac{y+\vartheta}{\beta+\vartheta}},\;\;\;y\geq-\vartheta,
$$
$0$ otherwise, and therefore $Y^\dag$ is distributed as $Z-\vartheta$ with $Z\sim\exp((\beta+\vartheta)^{-1})$, we have $h(Y^\dag)=h(Z)$, in agreement with equation (\ref{SE}).

Regarding capacity, and restricting attention to $\beta\in G([\vartheta,\infty))$, we have the channel capacity upper bound
\begin{eqnarray}\label{cce}
    C(N;G,\beta)&\leq&\ln e(\beta+\vartheta)-\ln 2\vartheta\nonumber\\&=&\ln\frac{e(\beta+\vartheta)}{2\vartheta}
\end{eqnarray}
since
$$
\sup\{h(Y):Y\in\Sigma_\bot(N;G,\beta)\}\leq\ln e(\beta+\vartheta)
$$
and
$$
h(N)=\ln 2\vartheta.
$$
However, in this case, as straightforward to verify,
$$
\X_\beta(f)=\frac{e^{j\frac{2\pi f\vartheta(\beta-\vartheta)}{\beta+\vartheta}}}{1+j2\pi f(\beta+\vartheta)}\frac{2\pi f\vartheta}{\sin 2\pi f\vartheta},\;\;\;f\in S_N,
$$
and therefore, regardless of the value of $\beta$, a corresponding positive definite continuous extension $\overline{\X}_\beta$ does not exist since
$$
|\X_\beta(f)|\rightarrow\infty\;\;\;\text{as}\;\;\;f\rightarrow\frac{k}{2\vartheta}
$$
for each $k\in\{\pm 1,\pm 2,\ldots\}$, i.e., in particular $\X_\beta$ is not bounded in $S_N$ for any $\beta$. Thus, channel capacity is not achievable. Though not achievable, but since $G$ is linear and therefore the average cost constraint is equally satisfied by a jointly distributed with $N$ input RV as well as by an independent of $N$ input RV, and as the approximation scheme to be developed in Section \ref{existence} will show, where a the noise RV $N$ as the one at hand will be approximated by a sequence of noise RVs $\{N_k\}_{k\geq 0}$ with each $N_k$ such that $S_{N_k}=\R$, we in fact have that the bound in (\ref{cce}) indeed correspond to channel capacity, that is, we have
$$
C(N;G,\beta)=\ln\frac{e(\beta+\vartheta)}{2\vartheta}.
$$

\paragraph{Case $(1-2e^{-1})\vartheta\leq\beta<\vartheta$}

Proceeding as in the previous case, we have
$$
g(n)=-n-\frac{1}{\lambda}\ln\frac{\lambda(\vartheta-n)}{2\vartheta C}.
$$
However, since $\lambda\in[(\beta+\vartheta)^{-1},\infty)$ and now $\beta<\vartheta$, we have
$$
\lambda>\frac{1}{2\vartheta}
$$
and therefore
$$
\lim_{\substack{n\rightarrow -\vartheta\\n>-\vartheta}}g'(n)=-1+\frac{1}{2\lambda\vartheta}<0.
$$
It is straightforward to see that $g$ indeed decreases from $x_{(-\vartheta)}$ at $-\vartheta$, to reach its minimum
$$
g(\zeta)=-\vartheta+\frac{1}{\lambda}+\frac{1}{\lambda}\ln 2\vartheta C
$$
at
$$
\zeta\doteq\vartheta-\frac{1}{\lambda},
$$
to then increase to infinity as $n$ approaches $\vartheta$ (from below). Thus, maximum entropy at the output is attained for the minimum value of $\lambda\geq(\beta+\vartheta)^{-1}$ such that $g(\zeta)\geq 0$, i.e., such that
$$
-\vartheta+\frac{1}{\lambda}+\frac{1}{\lambda}\ln 2\vartheta C\geq 0
$$
or, equivalently, and after using equation (\ref{UNE1}), such that
$$
\Lambda(\lambda)\doteq\frac{\ln 2\vartheta\lambda}{\lambda}\geq\vartheta-\beta.
$$
Since also
$$
(1-2e^{-1})\vartheta\leq\beta<\vartheta,
$$
it is easy to see that the equation
$$
\Lambda(\lambda)=\vartheta-\beta
$$
has two solutions, which coincide when $\beta=(1-2e^{-1})\vartheta$ and with both lying in the interval $((\beta+\vartheta)^{-1},\infty)$, and that indeed the desired $\lambda$ corresponds to the minimum between them, that is,
$$
\lambda=\min\{\xi>(\beta+\vartheta)^{-1}:\Lambda(\xi)=\vartheta-\beta\}.
$$
All the results of the previous case apply the same here \textit{mutatis-mutandis}. In particular,
\begin{equation*}
\sup\{h(Y):Y\in\Sigma(N;G,\beta)\}=1-\ln\lambda=1-\ln\min\{\xi>(\beta+\vartheta)^{-1}:\Lambda(\xi)=\vartheta-\beta\}
\end{equation*}
and
$$
 C(N;G,\beta)=\ln\frac{e}{2\vartheta\min\{\xi>(\beta+\vartheta)^{-1}:\Lambda(\xi)=\vartheta-\beta\}},
$$
being channel capacity not achievable. Note when
$$
\Lambda(\beta^{-1})=\vartheta-\beta,
$$
i.e., when
$$
\frac{\ln 2\vartheta\frac{1}{\beta}}{\frac{1}{\beta}}=\vartheta-\beta
$$
(which is possible in the considered range), we have
$$
\lambda=C=\frac{1}{\beta}
$$
and therefore (see equation (\ref{ODen}))
$$
p_{Y^\dag}(y)=\frac{1}{\beta}e^{-\frac{y}{\beta}},\;\;\;y\geq 0,
$$
$0$ otherwise, that is,
$$
Y^\dag\sim\exp(\beta^{-1}),
$$
in agreement with $x_{(-\vartheta)}=\vartheta$ and $y_{(-\vartheta)}=0$.

\paragraph{Case $0<\beta<(1-2e^{-1})\vartheta$}

From the previous cases we conclude that no solution for $g$ of the form
$$
-n-\frac{1}{\lambda}\ln\frac{\lambda(\vartheta-n)}{2\vartheta C}
$$
can be made non-negative in the present case. However, it is easy to see that we must now look for a solution of the form
\begin{displaymath}
g(n)= \left\{ \begin{array}{ll}
-n-\frac{1}{\lambda_1}\ln\frac{\lambda_1(\kappa_1-n)}{2\vartheta C_1}  & \textrm{if $n\in[-\vartheta,n_1]$}\\
0 & \textrm{if $n\in(n_1,n_2]$}\\
-n-\frac{1}{\lambda_2}\ln\frac{\lambda_2(\vartheta-n)}{2\vartheta C_2}  & \textrm{if $n\in(n_2,\vartheta)$}
\end{array} \right.
\end{displaymath}
with
$\lambda_1$, $\kappa_1$, $C_1$, $n_1$, $n_2$, $\lambda_2$ and $C_2$ real constants such that
$$
\lambda_1, C_1, \lambda_2, C_2>0,\;\;\;\kappa_1-n_1>0
$$
and
$$
-\vartheta\leq n_1<n_2<\vartheta.
$$
(Note we are considering equation (\ref{ELAC}) with $\mu\equiv 0$ where $g$ above is $>0$ and $\mu\equiv$ constant $\neq 0$ otherwise, the actual value of this constant being irrelevant.) Note as before, the (natural) boundary condition at $\vartheta-$ has already been imposed. The available equations from continuity requirements are
$$
g(n_i-)\doteq\lim_{\substack{n\rightarrow n_i\\n<n_i}}g(n)=\lim_{\substack{n\rightarrow n_i\\n>n_i}}g(n)\doteq g(n_i+)
$$
and
$$
g'(n_i-)\doteq\lim_{\substack{n\rightarrow n_i\\n<n_i}}g'(n)=\lim_{\substack{n\rightarrow n_i\\n>n_i}}g'(n)\doteq g'(n_i+),
$$
both for $i=1,2$, and where for sake of well definiteness we may w.l.o.g. assume $n_1>-\vartheta$ in the computation of $g(n_1-)$ and $g'(n_1-)$ above (the case $n_1=-\vartheta$ being included in the resulting equations). Also, from the isoperimetric condition,
$$
\int_{(-\vartheta,\vartheta)}g(n)\ell(dn)=2\vartheta\beta.
$$
We so have five equations for the seven unknowns $\lambda_1$, $\kappa_1$, $C_1$, $n_1$, $n_2$, $\lambda_2$ and $C_2$, and therefore we may choose to express five of them in terms of the remaining two. For convenience we treat $\lambda_1$ and $n_1$ as the independent parameters and express $\kappa_1$, $C_1$, $n_2$, $\lambda_2$ and $C_2$ accordingly. After some algebraic manipulations we find
$$
\kappa_1=n_1+\frac{1}{\lambda_1}
$$
and
$$
C_1=\frac{e^{\lambda_1n_1}}{2\vartheta},
$$
and
$$
C_2=\frac{e^{\lambda_2(\vartheta-\frac{1}{\lambda_2})}}{2\vartheta},
$$
$$
n_2=\vartheta-\frac{1}{\lambda_2}
$$
with
$$
\lambda_2=\frac{1}{\sqrt{\Lambda_2(\lambda_1,n_1)}}
$$
and where
\begin{equation*}
\Lambda_2(\lambda_1,n_1)\doteq(\vartheta^2-n_1^2-4\vartheta\beta)-\frac{2}{\lambda_1}
-\frac{2}{\lambda_1^2}\left(\lambda_1(\vartheta+n_1)+1\right)\ln\frac{\lambda_1(\vartheta+n_1)+1}{e^{\lambda_1n_1}}.
\end{equation*}
Therefore, and upon expressing the corresponding entropy at the output $h(Y)$ in terms of $\lambda_1$ and $n_1$, abusing notation denoted as $h(\lambda_1,n_1)$, we must perform the associated maximization of $h(\lambda_1,n_1)$ over $(\lambda_1,n_1)$ in the valid domain $\mathcal{D}$ with
$$
\mathcal{D}\doteq\left\{(\lambda_1,n_1)\in\mathcal{D}_1:0<\Lambda_2(\lambda_1,n_1)<(\vartheta-n_1)^2\right\}
$$
and where
$$
\mathcal{D}_1\doteq\left\{(\lambda_1,n_1)\in\R^2:\lambda_1>0, -\vartheta\leq n_1<\vartheta\right\}.
$$
(Note $(\lambda_1,n_1)\in\mathcal{D}$ guarantees for $\lambda_2$ to be well defined and for $n_1<n_2$, as required.) Since the density at the output $p_Y$ is easily seen to be
\begin{displaymath}
p_Y(y)= \left\{ \begin{array}{ll}
C_1e^{-\lambda_1y}  & \textrm{if $y\in[-\frac{1}{\lambda_1}\ln\frac{\lambda_1(\kappa_1+\vartheta)}{2\vartheta C_1},n_1]$}\\
\frac{1}{2\vartheta} & \textrm{if $y\in(n_1,n_2]$}\\
C_2e^{-\lambda_2y}  & \textrm{if $y\in(n_2,\infty)$}
\end{array} \right.
\end{displaymath}
and $0$ otherwise, after some algebra we find
\begin{equation*}
h(Y)=C_1e^{-\lambda_1n_1}\left(-n_1+\frac{1}{\lambda_1}\ln\frac{C_1}{e}\right)+\frac{1}{2\vartheta}\left((\kappa_1+\vartheta)\ln\frac{2\vartheta e}{\lambda_1(\kappa_1+\vartheta)}+\ln 2\vartheta\right)
+C_2e^{-\lambda_2n_2}\left(n_2-\frac{1}{\lambda_2}\ln\frac{C_2}{e}\right)
\end{equation*}
which, and in light of the previous relationships, gives the corresponding expression for $h(\lambda_1,n_1)$ after replacements. Therefore, if indeed a maximum
$$
\max\{h(\xi_1,\xi_2):(\xi_1,\xi_2)\in\mathcal{D}\}
$$
exists, then we have with
$$
(\lambda_1,n_1)=\am\{h(\xi_1,\xi_2):(\xi_1,\xi_2)\in\mathcal{D}\}
$$
that
\begin{equation*}
\sup\{h(Y):Y\in\Sigma(N;G,\beta)\}=h(Y^\dag)=h(\lambda_1,n_1)
\end{equation*}
with
$$
Y^\dag=\Sigma(X^\dag,N;G,\beta)
$$
and
$$
X^\dag=g(N),
$$
and where $g$ is correspondingly given by replacing all the optimum values for the respective constants. Also,
$$
 C(N;G,\beta)= h(\lambda_1,n_1)-\ln2\vartheta,
$$
being as before the previous channel capacity not achievable.

Of course, the previous maximization
$$
\max\{h(\xi_1,\xi_2):(\xi_1,\xi_2)\in\mathcal{D}\}
$$
is algebraically complicated enough so as to not allow for a final closed-form expression. However, it can be easily implemented numerically for given $\beta$ and $\vartheta$ (satisfying of course $0<\beta<(1-2e^{-1})\vartheta$). As an example, we numerically solved the previous optimization for the case with
$$
\vartheta=3\;\;\;\text{and}\;\;\;\beta=0.3.
$$
After appropriate numerical testing, it was found that
$$
\lambda_1\in[48,48.5]\;\;\;\text{and}\;\;\;n_1\in[-0.985,-0.98].
$$
For steps sizes of $10^{-6}$ for $\lambda_1$ and $10^{-7}$ for $n_1$ in the corresponding grids partitioning the intervals above, the optimal values found for $\lambda_1$ and $n_1$ were
$$
\lambda_1=48.1888\;\;\;\text{and}\;\;\;n_1=-0.9825,
$$
giving a maximum entropy
$$
h(\lambda_1,n_1)=10.0745\;(\nus).
$$
Regarding the general expression for the optimum curve $g$, its main points were found to be given by
$$
g(-\vartheta)=1.9223,
$$
to then decrease towards zero in $(-\vartheta,n_1)=(-3,-0.9825)$, staying at zero in $[-0.9825,n_2]$ with
$$
n_2=2.9958,
$$
and to then increase towards infinity in $(n_2,\vartheta)=(2.9958,3)$.
Though $n_2=2.9958\simeq 3=\vartheta$, the corresponding contribution to the entropy at the output in the interval $(n_2,\vartheta)$ was found to be
$$
-\int_{(n_2,\vartheta)}p_Y(y)\ln p_Y(y)\ell(dy)=0.0020(>0),
$$
i.e., as discussed next, several orders of magnitude greater than the numerical precisions involved (and in accordance with the natural boundary condition at $\vartheta-$, which precludes the case $n_2=\vartheta$). As a numerical checking on the correctness of the values found, the continuity requirements were met with the following precisions. At $n_1=-0.9825$,
$$
g(n_1)=-1.1102\times 10^{-16}\;\;\text{and}\;\;g'(n_1)=-6.6613\times 10^{-16},
$$
and, at $n_2=2.9958$,
$$
g(n_2)=0\;\;\text{and}\;\;g'(n_2)=3.7081\times 10^{-14},
$$
that is, with all of them being $\simeq 0$ at least to the order of $10^{-13}$.

Finally, we have the following channel capacity
\begin{eqnarray*}
 C(N;G,\beta)&=& h(\lambda_1,n_1)-\ln2\vartheta\\
&=&10.0745-\ln 6\\
&=&8.2827\;(\nus).
\end{eqnarray*}

\section{Attainability}
\label{existence}

In this section of the paper we present general conditions for the optimum
\begin{equation*}
\sup\{h(Y):Y\in\Sigma(N;G,\beta)\}=\max\{h(Y):Y\in\Sigma(N;G,\beta)\}=h(Y^\dag)
\end{equation*}
to exist, being then attained, as established in Section \ref{main}, at $Y^\dag$ with
$$
Y^\dag=\Sigma(X^\dag,N;G,\beta)
$$
and
$$
X^\dag=g(N),
$$
and with function $g$ satisfying the Euler-Lagrange equation (\ref{ELCFI}), along with the boundary conditions in (\ref{BCs}).

Conditions for channel capacity to be achievable by means of a capacity-achieving input distribution will also be provided, along with an interpretation of theoretical value via an approximation scheme when such attainability is not a priori guaranteed.

The section is a bit technical, so that instead of as in Section \ref{main} where we provided intuition for the proofs to summarize afterwards the corresponding result, we will now adopt the result-proof scheme, stating first the result to be followed by its respective proof.

Proofs rely on tools from functional analysis and Sobolev Space theory, being all of the cited theorems along the section available in any classical reference on such subjects, such as \cite{Ru,Sch,Al} and \cite{Ad}, to name a few.

We begin by stating two lemmas that will be of use in proving the main result of the section, Theorem \ref{thm3}. The first one is related to the approximation of the entropy at the output in (\ref{ODE2})
$$
h(Y)=h(N)+\int_{\nll}^{\nuu}p_N(n)\ln(1+g'(n))\ell(dn),
$$
when considering variations in the noise, and it will be of use in two cases: when $p_N$ is supported in an infinite Lebesgue measure interval, that is when $|\nll|+|\nuu|=\infty$, or, in the case of finite support, when $p_N$ corresponds to a uniform RV, that is, with $p_N$ constant over $(\nll,\nuu)$, $|\nll|+|\nuu|<\infty$. For that purpose, for a given $N\in\N(\nll,\nuu)$ we set
$$
\mathcal{H}_N[g]\doteq h(N)+\int_{\nll}^{\nuu}\ln(1+g')\Lambda_N(dn)
$$
over the domain $\mathcal{D}_{\mathcal{H}_N}$ consisting of the class of functions $g$ over $(\nll,\nuu)$ that are $\ell$-almost everywhere differentiable in $(\nll,\nuu)$ with $g'\geq-1$ at those points and giving finite entropy at the output, that is with
$$
\int_{\nll}^{\nuu}|\ln(1+g')|\Lambda_N(dn)<\infty.
$$
The second one is related to the Fourier transform quotient
$$
\X_\beta(f)=\frac{\F[\Lambda_{Y^\dag_{\beta}}]}{\F[\Lambda_N]}(f)
$$
with $f\in S_N$ and
$$
S_N=\{f\in\R:\F[\Lambda_N](f)\neq 0\},
$$
for a given $\beta\in \I_G$, of importance when characterizing the existence of a capacity-achieving input distribution as seen in Section \ref{main2}, and establishing its positive definiteness when continuity over $\R$ is already guaranteed by means of $S_N$ being the whole of $\R$, as well as recognizing general settings for which $S_N=\R$ is precisely the case.

\begin{lemma}\label{lemma1}
Let $N\in\N(\nll,\nuu)$ and cost function $G$. If $|\nll|+|\nuu|=\infty$ or if $p_N$ is constant in $(\nll,\nuu)$, then there exist a sequence of probability densities $\{p_{N_k}\}_{k\geq 0}$, defined accordingly over $(\nll,\nuu)$ but correspondingly supported over finite Lebesgue measure intervals $(\nll_k,\nuu_k)\subseteq(\nll,\nuu)$ ($\nll_k<\nuu_k$), and such that each of them can be taken to be nonconstant in $(\nll_k,\nuu_k)$, with associated noise RVs $N_k\in\N(\nll_k,\nuu_k)$, and with $\{p_{N_k}\}_{k\geq 0}$ converging to $p_N$ in $L^1((\nll,\nuu),\B(\nll,\nuu),\ell)$ such that, for each $g\in\mathcal{D}_{\mathcal{H}_N}$, we have
$$
g\in\bigcap_{k\geq 0}\mathcal{D}_{\mathcal{H}_{N_k}}
$$
and
$$
\mathcal{H}_{N_k}[g]\rightarrow\mathcal{H}_N[g]
$$
as $k\rightarrow\infty$. Moreover, if the corresponding optimums exist
\begin{equation*}
\sup\{h(Y):Y\in\Sigma(N;G,\beta)\}=\max\{h(Y):Y\in\Sigma(N;G,\beta)\}=h(Y_\beta^\dag)
\end{equation*}
and
\begin{equation*}
\sup\{h(Y):Y\in\Sigma(N_k;G,\beta)\}=\max\{h(Y):Y\in\Sigma(N_k;G,\beta)\}=h(Y_{k,\beta}^\dag)
\end{equation*}
for a given $\beta\in \I_G$ and all $k\geq 0$, then we have
$$
h(Y_{k,\beta}^\dag)\rightarrow h(Y_\beta^\dag)
$$
with
$$
Y_{k,\beta}^\dag\arrowd Y_\beta^\dag
$$
as $k\rightarrow\infty$. In addition, if the previous optimums exist not only for a given $\beta\in \I_G$ but for $\beta$ ranging in some compact interval $I\subseteq \I_G$, then the family
$$
\{h(Y_{k,\beta}^\dag)\}_{k\geq 0},\;\beta\in I,
$$
converges uniformly to $h(Y_\beta^\dag)$ in $I$ as $k\rightarrow\infty$. Finally, assume for a given $\beta\in \I_G$ there is a sequence of RVs $\{X_{N_k}\}_{k\geq 0}$ with each $X_{N_k}$ independent of $N_k$ and, setting
$$
Y_{N_k}\doteq X_{N_k}+N_k,
$$
with $Y_{N_k}\in\Sigma_\bot(N_k;G,\beta)$. Consider the sequence of RVs $\{X_N^k\}_{k\geq 0}$ with each $X_N^k$ independent of $N$ and equally distributed as $X_{N_k}$, that is, with
$$
X_{N}^k\eqd X_{N_k}.
$$
Then, setting
$$
Y_{N}^k\doteq X_{N}^k+N,
$$
each $Y_{N}^k\in\Sigma_\bot(N;G,\beta)$ and, denoting as $p_{Y_{N_k}}$ and $p_{Y_N^k}$ the densities of $Y_{N_k}$ and $Y_{N}^k$, respectively, which for simplicity are assumed to be defined in all $\R$ (that is, being set to $0$ where necessary), we have the following convergence as $k\rightarrow\infty$
$$
\int_\R|p_{Y_{N_k}}(y)-p_{Y_N^k}(y)|\ell(dy)\rightarrow 0
$$
and
$$
|h(Y_{N_k})-h(Y_N^k)|\rightarrow 0.
$$
\end{lemma}

\begin{proof}
Consider first $|\nll|+|\nuu|=\infty$. We will provide the proof for the case $\nll=-\infty$ and $\nuu=\infty$, the others following mutatis-mutandis. Take sequences $\{\nll_k\}_{k\geq 0}$ decreasing monotonically to $-\infty$ and $\{\nuu_k\}_{k\geq 0}$ increasing monotonically to $\infty$, with $\nll_0<\nuu_0$, and set for $k\geq 0$
$$
p_{N_k}(n)\doteq\frac{p_N(n)}{\int_{\nll_k}^{\nuu_k}p_N(\xi)d\xi}
$$
for $n\in[\nll_k,\nuu_k]$, $p_{N_k}(n)=0$ for $n\in(-\infty,\nll_k)\cup (\nuu_k,\infty)$. That each associated $N_k\in\N(\nll_k,\nuu_k)$ is trivial to verify, as well as the convergence of $\{p_{N_k}\}_{k\geq 0}$ to $p_N$ in $L^1((\nll,\nuu),\B(\nll,\nuu),\ell)$,
$$
\int_{\nll}^{\nuu}|p_N(n)-p_{N_k}(n)|\ell(dn)\rightarrow 0
$$
as $k\rightarrow\infty$. Moreover, the convergence $p_N(n)\rightarrow 0$ as $n\rightarrow\pm\infty$ guarantees that $a_0$ and $b_0$ can be taken to make each $p_{N_k}$ nonconstant in $(\nll_k,\nuu_k)$. Now, let $g\in\mathcal{D}_{\mathcal{H}_N}$ and note that
\begin{eqnarray*}
\int_{\nll}^{\nuu}|\ln(1+g')|\Lambda_{N_k}(dn)
&=&\left[\int_{\nll_k}^{\nuu_k}p_N(n)\ell(dn)\right]^{-1}\int_{\nll_k}^{\nuu_k}|\ln(1+g')|\Lambda_{N}(dn)\\&\leq&
\left[\int_{\nll_k}^{\nuu_k}p_N(n)\ell(dn)\right]^{-1}\int_{\nll}^{\nuu}|\ln(1+g')|\Lambda_{N}(dn)<\infty
\end{eqnarray*}
since $g\in\mathcal{D}_{\mathcal{H}_N}$, and therefore
$$
g\in\bigcap_{k\geq 0}\mathcal{D}_{\mathcal{H}_{N_k}}.
$$
Then
\begin{equation*}
|\mathcal{H}_N[g]-\mathcal{H}_{N_k}[g]|\leq|h(N)+h(N_k)|+\int_{\nll}^{\nuu}|\ln(1+g'(n))|\cdotp|p_N(n)-p_{N_k}(n)|\ell(dn).
\end{equation*}
It is straightforward to check that
$$
|h(N)+h(N_k)|\rightarrow 0
$$
as $k\rightarrow\infty$. As for the remaining term,
\begin{multline*}
\int_{\nll}^{\nuu}|\ln(1+g'(n))|\cdotp|p_N(n)-p_{N_k}(n)|\ell(dn)\\=\int_{-\infty}^{\nll_k}|\ln(1+g')|\Lambda_N(dn)+\int_{\nuu_k}^{\infty}|\ln(1+g')|\Lambda_N(dn)
+\left|1-\frac{1}{\int_{\nll_k}^{\nuu_k}p_N(\xi)\ell(d\xi)}\right|\int_{\nll_k}^{\nuu_k}|\ln(1+g')|\Lambda_N(dn)
\end{multline*}
whose convergence to $0$ as $k\rightarrow\infty$ follows from the fact that $g\in\mathcal{D}_{\mathcal{H}_N}$ and therefore
$$
\int_{\nll}^{\nuu}|\ln(1+g')|\Lambda_N(dn)<\infty.
$$

Having established the convergence $\mathcal{H}_{N_k}[g]\rightarrow\mathcal{H}_N[g]$ as $k\rightarrow\infty$ for all $g\in\mathcal{D}_{\mathcal{H}_N}$, the assertion of the lemma in that, when the optimums exist,
$$
h(Y_{k,\beta}^\dag)\rightarrow h(Y_\beta^\dag)
$$
with
$$
Y_{k,\beta}^\dag\arrowd Y_\beta^\dag
$$
as $k\rightarrow\infty$, follows then from the concavity of the functionals $\mathcal{H}_{N}$ and $\mathcal{H}_{N_k}$, $k\geq 0$. Moreover, when the optimums exist for $\beta\in I$ with $I\subseteq \I_G$ a compact interval, and since then $h(Y_\beta^\dag)$ is continuous in $I$ and each $h(Y_{k,\beta}^\dag)$ is monotone in $I$ (in fact they all, the limit as well as the sequence, are strictly increasing and differentiable in $I$), the previous pointwise convergence as $k\rightarrow\infty$
$$
h(Y_{k,\beta}^\dag)\rightarrow h(Y_\beta^\dag),
$$
$\beta\in I$, is in fact uniform in $I$ (see for example (\cite{HGE1949})).

Consider now $\beta\in \I_G$ and the sequences $\{X_{N_k}\}_{k\geq 0}$, $\{Y_{N_k}\}_{k\geq 0}$, $\{X_N^k\}_{k\geq 0}$ and $\{Y_N^k\}_{k\geq 0}$, as stated in the lemma. Then, since for each $k\geq 0$, $X_{N_k}$ independent of $N_k$ and $X_N^k$ independent of $N$,
\begin{eqnarray*}
|\F[\Lambda_{Y_{N_k}}](f)-\F[\Lambda_{Y_N^k}](f)|
&=&|\F[\Lambda_{X_{N_k}}](f)\F[\Lambda_{N_k}](f)-\F[\Lambda_{X_N^k}](f)\F[\Lambda_{N}](f)|
\\&=&|\F[\Lambda_{X_{N_k}}](f)|\cdot|\F[\Lambda_{N_k}](f)-\F[\Lambda_{N}](f)|,
\end{eqnarray*}
where for the last equality we have used that $X_{N}^k\eqd X_{N_k}$. Then,
\begin{equation*}
|\F[\Lambda_{Y_{N_k}}](f)-\F[\Lambda_{Y_N^k}](f)|
\leq|\F[\Lambda_{N_k}](f)-\F[\Lambda_{N}](f)|
\leq\int_{\nll}^\nuu|p_{N_k}(n)-p_N(n)|\ell(dn)\rightarrow 0
\end{equation*}
as $k\rightarrow\infty$ for each $f\in\R$, from where the convergence
$$
\int_\R|p_{Y_{N_k}}(y)-p_{Y_N^k}(y)|\ell(dy)\rightarrow 0
$$
as $k\rightarrow\infty$ now follows. For simplicity we will proceed with the proof assuming the RVs $\{X_{N_k}\}_{k\geq 0}$ and $\{X_N^k\}_{k\geq 0}$ are absolutely continuous w.r.t. Lebesgue measure in $\R$, that is, assuming the corresponding densities $\{p_{X_{N_k}}\}_{k\geq 0}$ and $\{p_{X_N^k}\}_{k\geq 0}$ exist. (Otherwise the same proof that follows can be given in terms of the respective laws, but with a more tedious notation.) Then, from independence, for each $k\geq 0$ we may write the corresponding convolution integrals
$$
p_{Y_N^k}(y)=\int_\nll^\nuu p_{X_N^k}(y-\xi)p_N(\xi)\ell(d\xi)
$$
and
$$
p_{Y_{N_k}}(y)=\int_{\nll_k}^{\nuu_k} p_{X_{N_k}}(y-\xi)p_{N_k}(\xi)\ell(d\xi)
$$
with $y\in\R$. Setting for $k\geq 0$
$$
D_k\doteq\left[\int_{\nll_k}^{\nuu_k}p_N(\xi)\ell(d\xi)\right]^{-1},
$$
and since $X_{N}^k\eqd X_{N_k}$, we may write for $y\in\R$
$$
p_{Y_{N_k}}(y)=D_k\int_{\nll_k}^{\nuu_k} p_{X_{N}^k}(y-\xi)p_{N}(\xi)\ell(d\xi),
$$
and therefore
\begin{equation*}
\int_\R p_{Y_{N_k}}(y)\ln p_{Y_{N_k}}(y)\ell(dy)=
D_k\ln D_k\int_\R\int_{\nll_k}^{\nuu_k}p_{X_{N}^k}(y-\xi)p_{N}(\xi)\ell(d\xi)\ell(dy)
+D_k\int_\R\phi_k(\nll_k,\nuu_k,y)\ell(dy),
\end{equation*}
where we have set for $k\geq 0$, $\nll\leq u\leq v\leq\nuu$ and $y\in\R$,
\begin{equation*}
\phi_k(u,v,y)\doteq \int_{u}^v p_{X_{N}^k}(y-\xi)p_{N}(\xi)
\ln\left[p_{X_{N}^k}(y-\xi)p_{N}(\xi)\right]\ell(d\xi).
\end{equation*}
Then
\begin{multline*}
\left|\int_\R p_{Y_{N_k}}(y)\ln p_{Y_{N_k}}(y)\ell(dy)-\int_R p_{Y_{N}^k}(y)\ln p_{Y_{N}^k}(y)\ell(dy)\right|\\
\leq\left|D_k\ln D_k\int_\R\int_{\nll_k}^{\nuu_k}p_{X_{N}^k}(y-\xi)p_{N}(\xi)\ell(d\xi)\ell(dy)\right|
+\left|D_k\int_\R\phi_k(\nll_k,\nuu_k,y)\ell(dy)-\int_\R\phi_k(\nll,\nuu,y)\ell(dy)\right|.
\end{multline*}
Since for all $k\geq 0$
$$
0\leq\int_\R\int_{\nll_k}^{\nuu_k}p_{X_{N}^k}(y-\xi)p_{N}(\xi)\ell(d\xi)\ell(dy)\leq 1
$$
and also
$$
\lim_{k\rightarrow\infty}D_k=1,
$$
we have
$$
\left|D_k\ln D_k\int_\R\int_{\nll_k}^{\nuu_k}p_{X_{N}^k}(y-\xi)p_{N}(\xi)\ell(d\xi)\ell(dy)\right|\rightarrow 0
$$
as $k\rightarrow\infty$. Now, for $j,k\in\mathbb{N}_0\doteq\{0,1,2,\ldots\}$ set
$$
L(j,k)\doteq\left\{y\in[-j,j]:\int_{\nll_k}^{\nuu_k}p_{X_{N}^k}(y-\xi)p_{N}(\xi)<1\right\}
$$
and
$$
U(k)\doteq\left\{y\in\R:\int_{\nll_k}^{\nuu_k}p_{X_{N}^k}(y-\xi)p_{N}(\xi)\geq 1\right\}.
$$
Then, for each $k\geq 0$,
\begin{equation*}
\int_\R\phi_k(\nll_k,\nuu_k,y)\ell(dy)=\int_{U(k)}\phi_k(\nll_k,\nuu_k,y)\ell(dy)
+\sum_{j=0}^{\infty}\int_{L(j,k)}\phi_k(\nll_k,\nuu_k,y)\ell(dy)
\end{equation*}
and
\begin{equation*}
\int_\R\phi_k(\nll,\nuu,y)\ell(dy)=\int_{U(k)}\phi_k(\nll,\nuu,y)\ell(dy)
+\sum_{j=0}^{\infty}\int_{L(j,k)}\phi_k(\nll,\nuu,y)\ell(dy).
\end{equation*}
Since $\phi_k(\nll_k,\nuu_k,y)\geq 0$ over $U(k)$ and $\phi_k(\nll_k,\nuu_k,y)$ increases to $\phi_k(\nll,\nuu,y)$ with $k$, Lebesgue's Monotone Convergence theorem shows that
$$
\left|\int_{U(k)}\phi_k(\nll_k,\nuu_k,y)\ell(dy)-\int_{U(k)}\phi_k(\nll,\nuu,y)\ell(dy)\right|\rightarrow 0
$$
as $k\rightarrow\infty$. Now, Lebesgue's Dominated Convergence theorem shows that, for each $j\geq 0$,
$$
\left|\int_{L(j,k)}\phi_k(\nll_k,\nuu_k,y)\ell(dy)-\int_{L(j,k)}\phi_k(\nll,\nuu,y)\ell(dy)\right|\rightarrow 0
$$
as $k\rightarrow\infty$ as well, and, since
$$
-\int_{L(j,k)}\phi_k(\nll_k,\nuu_k,y)\ell(dy)\geq 0
$$
for all $j,k\in\mathbb{N}_0$, we conclude once again from Lebesgue's Monotone Convergence theorem that
\begin{equation*}
\left|\sum_{j=0}^{\infty}\int_{L(j,k)}\phi_k(\nll_k,\nuu_k,y)\ell(dy)
-\sum_{j=0}^{\infty}\int_{L(j,k)}\phi_k(\nll,\nuu,y)\ell(dy)\right|\rightarrow 0
\end{equation*}
as $k\rightarrow\infty$. Thus, since $D_k\rightarrow 1$ as $k\rightarrow\infty$, we then conclude
$$
\left|D_k\int_\R\phi_k(\nll_k,\nuu_k,y)\ell(dy)-\int_\R\phi_k(\nll,\nuu,y)\ell(dy)\right|\rightarrow 0
$$
as $k\rightarrow\infty$, and therefore
\begin{equation*}
\left|\int_\R p_{Y_{N_k}}(y)\ln p_{Y_{N_k}}(y)\ell(dy)-\int_\R p_{Y_{N}^k}(y)\ln p_{Y_{N}^k}(y)\ell(dy)\right|\rightarrow 0
\end{equation*}
as $k\rightarrow\infty$ as well. Since $Y_{N_k}\in\Sigma_\bot(N_k;G,\beta)$ for all $k\geq 0$, and therefore
$$
\int_\R p_{Y_{N_k}}(y)|\ln p_{Y_{N_k}}(y)|\ell(dy)<\infty,
$$
we conclude $Y_{N}^k\in\Sigma_\bot(N;G,\beta)$ for all $k\geq 0$ as well (redefining if necessary the sequence to start, say, at $k_0$), and hence
$$
|h(Y_{N_k})-h(Y_N^k)|\rightarrow 0
$$
as $k\rightarrow\infty$. The proof is thus concluded for the case $\nll=-\infty$ and $\nuu=\infty$.

Considering now $N$ a uniform RV in $(\nll,\nuu)$, and setting
$$
p_{N_k}(n)\doteq\theta_k\left[n-\frac{\nll+\nuu}{2}\right]+\frac{1}{\nuu-\nll}
$$
for $n\in(\nll,\nuu)$ with $\{\theta_k\}_{k\geq0}$ any sequence converging monotonically to $0$ (taking of course $\nll_k=\nll$ and $\nuu_k=\nuu$, $k\geq 0$), we get the advertised convergence to $p_N$ in $L^1((\nll,\nuu),\B(\nll,\nuu),\ell)$ and the remaining of the proof follows by exactly the same arguments as before. We have thus proved the lemma.
\end{proof}

As mentioned before, the next lemma considers the Fourier transform quotient
$$
\X_\beta(f)=\frac{\F[\Lambda_{Y^\dag_{\beta}}]}{\F[\Lambda_N]}(f)
$$
with $f\in S_N$ and
$$
S_N=\{f\in\R:\F[\Lambda_N](f)\neq 0\},
$$
for a given $\beta\in \I_G$, and states that when $S_N=\R$, giving then the continuity of $\X_\beta$ in the whole of $\R$ from the individual continuities of numerator and denominator, we obtain a positive definite function. In addition, it recognizes general settings for which $S_N=\R$ is precisely the case.

\begin{lemma}\label{lemma2}
Let $N\in\N(\nll,\nuu)$ and cost function $G$. Assume that $S_N=\R$ and that the optimum exists
\begin{equation*}
\sup\{h(Y):Y\in\Sigma(N;G,\beta)\}=\max\{h(Y):Y\in\Sigma(N;G,\beta)\}=h(Y_\beta^\dag)
\end{equation*}
for a given $\beta\in \I_G$. Then, the continuous function
$$
\X_\beta(f)=\frac{\F[\Lambda_{Y^\dag_{\beta}}]}{\F[\Lambda_N]}(f),\;\;\;f\in\R,
$$
is positive definite. Moreover, we in fact have $S_N=\R$ either when $p_N$ is nonconstant and of compact support ($|\nll|+|\nuu|<\infty$), when $p_N$ is non-increasing and of semi-infinite support ($|\nll|<\infty$, $\nuu=\infty$), or when $p_N$ is non-decreasing and of semi-infinite support as well ($\nll=-\infty$, $|\nuu|<\infty$).
\end{lemma}

\begin{proof}
Let $n\in\mathbb{N}$ and $f_1, f_2, ..., f_n\in\R$, $\xi_1, \xi_2, ..., \xi_n\in\mathbb{C}$. Then,
\begin{equation*}
\sum_{l=1}^n\sum_{m=1}^n\frac{\F[\Lambda_{Y^\dag_{\beta}}](\xi_l-\xi_m)}{\F[\Lambda_N](\xi_l-\xi_m)}\xi_l\xi^*_m
=\frac{\F[\Lambda_{Y^\dag_{\beta}}](\xi_l-\xi_m)}{\F[\Lambda_N](\xi_l-\xi_m)}\xi_l\xi^*_m
+\frac{\F[\Lambda_{Y^\dag_{\beta}}](\xi_m-\xi_l)}{\F[\Lambda_N](\xi_m-\xi_l)}\xi_m\xi^*_l+\dots,
\end{equation*}
where we have written together the terms for $(l,m)$ and $(m,l)$, and where the dots denote all the remaining terms in the sum not explicitly written. Since $Y^\dag_{\beta}$ and $N$ are real valued RVs, we have for $f\in\R$
$$
\F[\Lambda_{Y^\dag_{\beta}}](f)=\F^*[\Lambda_{Y^\dag_{\beta}}](-f),
$$
and similarly for $\F[\Lambda_N]$, with $\F^*$ denoting complex conjugate of the corresponding Fourier transform at $f\in\R$. Hence, we may write
\begin{eqnarray*}
\sum_{l=1}^n\sum_{m=1}^n\frac{\F[\Lambda_{Y^\dag_{\beta}}](\xi_l-\xi_m)}{\F[\Lambda_N](\xi_l-\xi_m)}\xi_l\xi^*_m
&=&\left[\frac{\F[\Lambda_{Y^\dag_{\beta}}](\xi_m-\xi_l)}{\F[\Lambda_N](\xi_m-\xi_l)}\xi_l^*\xi_m\right]^*
+\frac{\F[\Lambda_{Y^\dag_{\beta}}](\xi_m-\xi_l)}{\F[\Lambda_N](\xi_m-\xi_l)}\xi_m\xi^*_l+\dots\\
&\doteq&\frac{A^*_{m,l}}{\F^*[\Lambda_N](\xi_m-\xi_l)}+\frac{A_{m,l}}{\F[\Lambda_N](\xi_m-\xi_l)}+\dots\\
&=&\frac{A^*_{m,l}\F[\Lambda_N](\xi_m-\xi_l)+A_{m,l}\F^*[\Lambda_N](\xi_m-\xi_l)}{|\F[\Lambda_N](\xi_m-\xi_l)|^2}+\dots
\end{eqnarray*}
Now, both numerator and denominator of the previous expression are real numbers, and since by Bochner's theorem $\F[\Lambda_N](f)$ is positive definite and thus
$$
|\F[\Lambda_N](f)|\leq|\F[\Lambda_N](0)|=1
$$
for all $f\in\R$, we conclude
\begin{equation*}
\frac{A^*_{m,l}\F[\Lambda_N](\xi_m-\xi_l)+A_{m,l}\F^*[\Lambda_N](\xi_m-\xi_l)}{|\F[\Lambda_N](\xi_m-\xi_l)|^2}+\dots
\geq A^*_{m,l}\F[\Lambda_N](\xi_m-\xi_l)+A_{m,l}\F^*[\Lambda_N](\xi_m-\xi_l)+\dots
\end{equation*}
Thus, we have
\begin{equation*}
\sum_{l=1}^n\sum_{m=1}^n\frac{\F[\Lambda_{Y^\dag_{\beta}}](\xi_l-\xi_m)}{\F[\Lambda_N](\xi_l-\xi_m)}\xi_l\xi^*_m
\geq\sum_{l=1}^n\sum_{m=1}^n\F[\Lambda_{Y^\dag_{\beta}}](\xi_m-\xi_l)\F^*[\Lambda_N](\xi_m-\xi_l)\xi^*_l\xi_m
\doteq\sum_{l=1}^n\sum_{m=1}^n\F(\xi_m-\xi_l)\xi^*_l\xi_m,
\end{equation*}
that is, with $\F(f)\doteq\F[\Lambda_{Y^\dag_{\beta}}](f)\F^*[\Lambda_N](f)$ for all $f\in\R$. But, it is easy to see that $\F(f)$ so defined corresponds to the Fourier transform of the convolution function of the densities of the RVs $Y^\dag_{\beta}$ and $-N$, i.e., $\F$ can be looked at as the Fourier transform of a density w.r.t. Lebesgue measure for the RV $\hat{X}$, with
$$
\hat{X}\doteq\hat{Y}-N
$$
and with $\hat{Y}\eqd Y^\dag_{\beta}$ and $\hat{Y}$ and $N$ independent. Hence, by Bochner's theorem we conclude that $\F$ is positive definite and therefore
\begin{equation*}
\sum_{l=1}^n\sum_{m=1}^n\frac{\F[\Lambda_{Y^\dag_{\beta}}](\xi_l-\xi_m)}{\F[\Lambda_N](\xi_l-\xi_m)}\xi_l\xi^*_m
\geq\sum_{l=1}^n\sum_{m=1}^n\F(\xi_m-\xi_l)\xi^*_l\xi_m\geq0,
\end{equation*}
which gives the first part of the claimed result.

As for the second one, since
\begin{equation*}
\F[\Lambda_{N}](f)=\int_{\nll}^{\nuu}e^{-j2\pi f\xi}p_{N}(\xi)\ell(d\xi)
=\int_{\nll}^{\nuu}\cos(2\pi f\xi)p_{N}(\xi)\ell(d\xi)-j\int_{\nll}^{\nuu}\sin(2\pi f\xi)p_{N}(\xi)\ell(d\xi),
\end{equation*}
$f\in\R$, and since $\E|N|<\infty$ and $p_N\in\C^1(\nll,\nuu)$, in particular continuous, by \cite[Thms. 2.1.1 and 2.1.5, pp. 4264 and 4268, respt.]{AMS2005} we conclude, and by linearly transforming $N$ if necessary, that
$$
\int_{\nll}^{\nuu}\sin(2\pi f\xi)p_{N}(\xi)\ell(d\xi)\neq 0
$$
for all $f\in\R$, from where
$$
\F[\Lambda_N](f)\neq 0
$$
for all $f\in\R$ as well, in all the advertised cases, that is, either when $p_N$ is nonconstant and of compact support ($|\nll|+|\nuu|<\infty$), when $p_N$ is non-increasing and of semi-infinite support ($|\nll|<\infty$, $\nuu=\infty$), or when $p_N$ is non-decreasing and of semi-infinite support as well ($\nll=-\infty$, $|\nuu|<\infty$). We have thus proved the lemma.
\end{proof}

We are now in position to state and prove the main result of this section, Theorem \ref{thm3} below. As mentioned at the beginning of the section, it recognizes general conditions for the optimum to exist for all $\alpha\in \I_G$,
\begin{equation*}
\sup\{h(Y):Y\in\Sigma(N;G,\alpha)\}=\max\{h(Y):Y\in\Sigma(N;G,\alpha)\}=h(Y^{\dag}_{\alpha}),
\end{equation*}
where each $Y^{\dag}_{\alpha}$, as established in Section \ref{main}, is given then by
$$
Y^\dag_\alpha=\Sigma(X^\dag_\alpha,N;G,\alpha)
$$
with
$$
X^\dag_\alpha=g_\alpha(N),
$$
and with each function $g_\alpha$ satisfying the Euler-Lagrange equation (\ref{ELCFI}), along with the boundary conditions in (\ref{BCs}). The theorem also identify general conditions for the corresponding capacity to be achievable.

\begin{theorem}\label{thm3}
Let $N\in\N(\nll,\nuu)$ and cost function $G$. Assume there exist $p,q,r\in\R$, $p>0$, such that
$$
px^2+qx+r\leq G(x)
$$
for all $x\in\R$. Then, the optimum exists for all $\alpha\in \I_G$, that is, for each $\alpha\in \I_G$ there exists $Y_\alpha^\dag\in\Sigma(N;G,\alpha)$ such that
\begin{equation*}
\sup\{h(Y):Y\in\Sigma(N;G,\alpha)\}=\max\{h(Y):Y\in\Sigma(N;G,\alpha)\}=h(Y_\alpha^\dag),
\end{equation*}
and therefore we also have
$$
C(N;G,\alpha)<\infty.
$$
Moreover, either when $p_N$ is nonconstant and of compact support ($|\nll|+|\nuu|<\infty$), when $p_N$ is non-increasing and of semi-infinite support ($|\nll|<\infty$, $\nuu=\infty$), or when $p_N$ is non-decreasing and of semi-infinite support as well ($\nll=-\infty$, $|\nuu|<\infty$), we have that for each $\alpha\in \I_G$ there exists $Y^\ddag_\alpha\in\Sigma_{\bot}(N;G,\alpha)$ such that, with
$$
Y^\ddag_{\alpha}=\Sigma_{\bot}(X_\alpha^\ddag,N;G,\alpha),
$$
we have
\begin{equation*}
\sup\{h(Y):Y\in\Sigma_{\bot}(N;G,\alpha)\}=\max\{h(Y):Y\in\Sigma_{\bot}(N;G,\alpha)\}=h(Y^\ddag_{\alpha}),
\end{equation*}
that is, with the probability law of $X^\ddag_\alpha$ being a corresponding capacity-achieving input distribution, and therefore we may write
$$
C(N;G,\alpha)=h(Y^\ddag_{\alpha})-h(N).
$$
In addition, for each $\alpha\in \I_G$ there exists a unique $\alpha^\ddag(\alpha)\in \I_G$, $\alpha^\ddag(\alpha)<\alpha$, such that
$$
Y^\ddag_{\alpha}\eqd Y_{\alpha^\ddag(\alpha)}^\dag
$$
and therefore we may alternatively write
$$
C(N;G,\alpha)=h(Y_{\alpha^\ddag(\alpha)}^\dag)-h(N).
$$
Finally, for all cases not included above, when capacity achievability is not a priori guaranteed but depends on the existence of a positive definite continuous extension of the Fourier transform quotient
$$
\X_\alpha(f)=\frac{\F[\Lambda_{Y^\dag_{\alpha}}]}{\F[\Lambda_N]}(f),\;\alpha\in \I_G,
$$
from $S_N=\{f\in\R:\F[\Lambda_N](f)\neq 0\}$ to the whole of $\R$, cases that are subsumed when either $|\nll|+|\nuu|=\infty$ or $p_N$ constant in $(\nll,\nuu)$, we still have that there exists a sequence of noise RVs $\{N_k\}_{k\geq 0}$ with each $N_k\in\N(\nll_k,\nuu_k)$, $(\nll_k,\nuu_k)\subseteq(\nll,\nuu)$ ($\nll_k<\nuu_k$) and $|\nll_k|+|\nuu_k|<\infty$, and with densities $\{p_{N_k}\}_{k\geq 0}$ converging to $p_N$ in $L^1((\nll,\nuu),\B(\nll,\nuu),\ell)$ as $k\rightarrow\infty$ and such that, for each $\alpha\in \I_G$,
$$
C(N;G,\alpha)=\lim_{k\rightarrow\infty}C(N_k;G,\alpha)
$$
with each $C(N_k;G,\alpha)$ being attainable, that is, for each $\alpha\in \I_G$ there exists a sequence of RVs $\{Y_{k,\alpha}^\ddag\}_{k\geq 0}$, each one correspondingly belonging to $\Sigma_{\bot}(N_k;G,\alpha)$ and such that, with
$$
Y^\ddag_{k,\alpha}=\Sigma_{\bot}(X_{k,\alpha}^\ddag,N_k;G,\alpha),
$$
we have
\begin{equation*}
\sup\{h(Y):Y\in\Sigma_{\bot}(N_k;G,\alpha)\}=\max\{h(Y):Y\in\Sigma_{\bot}(N_k;G,\alpha)\}=h(Y^\ddag_{k,\alpha})
\end{equation*}
and therefore
$$
C(N_k;G,\alpha)=h(Y^\ddag_{k,\alpha})-h(N_k),
$$
that is, with the probability law of each $X_{k,\alpha}^\ddag$ being a corresponding capacity-achieving input distribution. As before, for each $k\geq 0$ and each $\alpha\in \I_G$ there exists a unique $\alpha^\ddag(k,\alpha)\in \I_G$, $\alpha^\ddag(k,\alpha)<\alpha$, such that
$$
Y^\ddag_{k,\alpha}\eqd Y_{k,\alpha^\ddag(k,\alpha)}^\dag,
$$
and therefore we may alternatively write
$$
C(N_k;G,\alpha)=h(Y_{k,\alpha^\ddag(k,\alpha)}^\dag)-h(N_k),
$$
where for each $k\geq 0$ and $\alpha\in \I_G$, $Y^\dag_{k,\alpha}\in\Sigma(N_k;G,\alpha)$ is the corresponding optimum
\begin{equation*}
\sup\{h(Y):Y\in\Sigma(N_k;G,\alpha)\}=\max\{h(Y):Y\in\Sigma(N_k;G,\alpha)\}=h(Y^\dag_{k,\alpha}).
\end{equation*}
In particular, we still have that for each $\alpha\in \I_G$ there exists a unique $\alpha^{\ddag}(\alpha)\in \I_G$, now with $\alpha^{\ddag}(\alpha)\leq\alpha$, and such that
$$
C(N;G,\alpha)=\lim_{k\rightarrow\infty}C(N_k;G,\alpha)=h(Y^\dag_{\alpha^\ddag(\alpha)})-h(N),
$$
even if the corresponding Fourier transform quotient
$$
\X_{\alpha^\ddag(\alpha)}(f)=\frac{\F[\Lambda_{Y^\dag_{\alpha^\ddag(\alpha)}}]}{\F[\Lambda_N]}(f)
$$
does not have a positive definite continuous extension from $S_N$ to the whole of $\R$, making then the corresponding capacity $C(N;G,\alpha)$ not achievable.
\end{theorem}

Before giving the proof, we make the following remark.

\begin{remark}
Note when the capacity is not achievable, that is when
$$
C(N;G,\alpha)=\sup\{h(Y):Y\in\Sigma_\bot(N;G,\beta)\}-h(N)
$$
but with
$$
\sup\{h(Y):Y\in\Sigma_\bot(N;G,\beta)\}
$$
not being attainable (not a true maximum), then, though a supremum can always be properly approximated, the theoretical relevance of the approximation scheme stated in Theorem \ref{thm3}, and based in turn in Lemmas \ref{lemma1} and \ref{lemma2} and their corresponding proofs, relies in that it provides a systematic way of construction of such approximation for any noise RV $N$, as well as providing an approximation via capacity attainable channels.
\end{remark}

\begin{proof}
With $\Omega_{\nll,\nuu}\doteq(\nll,\nuu)$, we will write
$$
L^2(\Lambda_N)=L^2(\Omega_{\nll,\nuu},\B(\Omega_{\nll,\nuu}),\Lambda_N)
$$
for the corresponding space of square integrable functions, and $H^1(\Lambda_N)$ for the corresponding Sobolev space of order $(1,2)$ in $\Omega_{\nll,\nuu}$ and w.r.t. the measure $\Lambda_N$, that is
$$
H^1(\Lambda_N)=W^{(1,2)}(\Omega_{\nll,\nuu},\Lambda_N)
$$
where $W^{(1,2)}(\Omega_{\nll,\nuu},\Lambda_N)$ is the cited Sobolev space,
$$
W^{(1,2)}(\Omega_{\nll,\nuu},\Lambda_N)\doteq\left\{g\in L^2(\Lambda_N):Dg\in L^2(\Lambda_N)\right\}
$$
with $Dg$ denoting the weak derivative of $g$. Writing $L^1(\Omega_{\nll,\nuu})$ for the space of integrable functions
$$
L^1(\Omega_{\nll,\nuu},\B(\Omega_{\nll,\nuu}),\ell)
$$
and $L^1_{loc}(\Omega_{\nll,\nuu})$ for its corresponding localized version (that is, the collection of all functions belonging to $L^1(K)$ for all compact $K\subseteq\Omega_{\nll,\nuu}$), we recall that $v\in L^1_{loc}(\Omega_{\nll,\nuu})$ is the weak derivative (unique up to a Lebesgue null set) of $u\in L^1_{loc}(\Omega_{\nll,\nuu})$, writen $v=Du$, if
$$
\int_{\Omega_{\nll,\nuu}}u\phi'\ell(dn)=-\int_{\Omega_{\nll,\nuu}}v\phi\ell(dn)
$$
for all $\phi\in\mathcal{C}_0^{\infty}(\Omega_{\nll,\nuu})$, the collection of infinitely differentiable functions with compact support in $\Omega_{\nll,\nuu}$. Note since $1/p_N\in L^1_{loc}(\Omega_{\nll,\nuu})$, then $g\in L^2(\Lambda_N)$ implies $g\in L^1_{loc}(\Omega_{\nll,\nuu})$. Indeed, by H\"{o}lder's inequality and for compact $K\subseteq\Omega_{\nll,\nuu}$,
$$
\int_K|g|\ell(dn)\leq||g||_{L^2(\Lambda_N)}\left[\int_K\frac{\ell(dn)}{p_N}\right]^\frac{1}{2}<\infty
$$
for $g\in L^2(\Lambda_N)$, with $||\cdot||_{L^2(\Lambda_N)}$ denoting the usual norm in the Banach space $L^2(\Lambda_N)$.

It is a known fact that the space $H^1(\Lambda_N)$ is a Banach space when endowed with the norm
$$
||g||_{H^1(\Lambda_N)}\doteq\left[||g||_{L^2(\Lambda_N)}^2+||Dg||_{L^2(\Lambda_N)}^2\right]^\frac{1}{2},
$$
and in fact a Hilbert space when endowed with the inner product
\begin{equation*}
<g_1,g_2>_{H^1(\Lambda_N)}\doteq<g_1,g_2>_{L^2(\Lambda_N)}+<Dg_1,Dg_2>_{L^2(\Lambda_N)}
\end{equation*}
with $<\cdot,\cdot>_{L^2(\Lambda_N)}$ denoting the usual inner product in the Hilbert space $L^2(\Lambda_N)$.

Now we may proceed with the proof. We fix throughout an $\alpha\in \I_G$ and set
$$
\mathcal{D}_\alpha\doteq\left\{g\in H^1(\Lambda_N):\int_{\Omega_{\nll,\nuu}}G(g)\Lambda_N(dn)\leq\alpha\right\}.
$$
We will first establish that $\mathcal{D}_\alpha$ is indeed weakly compact, that is, that for every sequence $\{g_i\}_{i\geq 0}\subseteq\mathcal{D}_\alpha$, there exists a corresponding subsequence $\{g_{i_k}\}_{k\geq 0}$ and a $g\in\mathcal{D}_\alpha$, such that
$$
g_{i_k}\arrowwh g
$$
as $k\rightarrow\infty$, with $\arrowwh$ denoting weak convergence in $H^1(\Lambda_N)$, that is, with
$$
g_{i_k}\arrowwl2 g
$$
and
$$
Dg_{i_k}\arrowwl2 Dg
$$
as $k\rightarrow\infty$, with $\arrowwl2$ denoting in turn weak convergence in $L^2(\Lambda_N)$. Recall for a sequence $\{f_i\}_{i\geq 0}\subseteq L^2(\Lambda_N)$ and $f\in L^2(\Lambda_N)$ we have
$$
f_{i}\arrowwl2 f
$$
if
$$
\int_{\Omega_{\nll,\nuu}}f_i\phi\Lambda_N(dn)\rightarrow\int_{\Omega_{\nll,\nuu}}f\phi\Lambda_N(dn)
$$
for all $\phi\in L^2(\Lambda_N)$ as $i\rightarrow\infty$ (being $L^2(\Lambda_N)$ its self-dual space). Let then a sequence $\{g_i\}_{i\geq 0}\subseteq\mathcal{D}_\alpha$ be given. Note since there exist $p,q,r\in\R$, $p>0$, such that
$$
px^2+qx+r\leq G(x)
$$
for all $x\in\R$, then there also exist $p^*,r^*\in\R$, $p^*>0$, such that
$$
p^*x^2+r^*\leq G(x)
$$
for all $x\in\R$ as well. Therefore we have
$$
p^*\int_{\Omega_{\nll,\nuu}}g_{i}^2\Lambda_N(dn)+r^*\leq
\int_{\Omega_{\nll,\nuu}}G(g_{i})\Lambda_N(dn)\leq\alpha,
$$
for all $i\geq 0$, and thus the sequence $\{g_{i}\}_{i\geq 0}$ is bounded in $L^2(\Lambda_N)$, that is, there exists $M\in(0,\infty)$ such that
$$
||g_{i}||_{L^2(\Lambda_N)}\leq M
$$
for all $i\geq 0$. Therefore we conclude there exists a subsequence, $\{g_{i_k}\}_{k\geq 0}$, and a $g\in L^2(\Lambda_N)$ such that
$$
g_{i_k}\arrowwl2 g
$$
as $k\rightarrow\infty$. Hence, by Banach-Saks theorem, there in turn exists a subsequence, which for simplicity we still denote as $\{g_{i_k}\}_{k\geq 0}$, such that its Ces\`{a}ro means
$$
\overline{g}_j\doteq\frac{1}{j}\sum_{k=0}^{j-1}g_{i_k},
$$
$j\geq 1$, converges to $g$ strongly in $L^2(\Lambda_N)$, that is,
$$
||\overline{g}_j-g||_{L^2(\Lambda_N)}\rightarrow 0
$$
as $j\rightarrow\infty$. Note from the convexity of $G$ we have
\begin{equation*}
\int_{\Omega_{\nll,\nuu}}G(\overline{g}_j)\Lambda_N(dn)
=\int_{\Omega_{\nll,\nuu}}G\left(\frac{1}{j}\sum_{k=0}^{j-1}g_{i_k}\right)\Lambda_N(dn)
\leq\frac{1}{j}\sum_{k=0}^{j-1}\int_{\Omega_{\nll,\nuu}}G(g_{i_k})\Lambda_N(dn)\leq\frac{1}{j}\sum_{k=0}^{j-1}\alpha=\alpha,
\end{equation*}
that is,
$$
\int_{\Omega_{\nll,\nuu}}G(\overline{g}_j)\Lambda_N(dn)\leq\alpha
$$
for all $j\geq 1$. Now, given strong convergence
$$
||\overline{g}_j-g||_{L^2(\Lambda_N)}\rightarrow 0
$$
as $j\rightarrow\infty$, we conclude that there exists a subsequence $\{\overline{g}_{j_k}\}_{k\geq 0}\subseteq\{\overline{g}_{j}\}_{j\geq 0}$ such that
$$
\overline{g}_{j_k}(n)\rightarrow g(n)
$$
for $\ell$-almost every $n\in\Omega_{\nll,\nuu}$, and therefore, from the continuity of $G$, that
$$
G(\overline{g}_{j_k}(n))\rightarrow G(g(n))
$$
for $\ell$-almost every $n\in\Omega_{\nll,\nuu}$ as well. Since $G$ is bounded below, say by $-R$ with $R\in(0,\infty)$, and
$$
\int_{\Omega_{\nll,\nuu}}R\Lambda_N(dn)=R,
$$
then from Fatou's lemma we conclude
$$
\int_{\Omega_{\nll,\nuu}}\liminf_{k\rightarrow\infty}G(\overline{g}_{j_k})\Lambda_N(dn)\leq\liminf_{k\rightarrow\infty}
\int_{\Omega_{\nll,\nuu}}G(\overline{g}_{j_k})\Lambda_N(dn)
$$
and therefore, since
$$
\int_{\Omega_{\nll,\nuu}}G(\overline{g}_j)\Lambda_N(dn)\leq\alpha
$$
for all $j\geq 1$, and
$$
G(\overline{g}_{j_k}(n))\rightarrow G(g(n))
$$
for $\ell$-almost every $n\in\Omega_{\nll,\nuu}$ (and hence for $\Lambda_N$-almost every $n\in\Omega_{\nll,\nuu}$, being $\Lambda_N$ and $\ell$ equivalent in $\B(\Omega_{\nll,\nuu})$), we conclude
$$
\int_{\Omega_{\nll,\nuu}}G(g)\Lambda_N(dn)\leq\alpha.
$$
Thus, for the given sequence $\{g_i\}_{i\geq 0}\subseteq\mathcal{D}_\alpha$, there exists $\{g_{i_k}\}_{k\geq 0}\subseteq\{g_i\}_{i\geq 0}$ and $g\in L^2(\Lambda_N)$ such that
$$
g_{i_k}\arrowwl2 g
$$
as $k\rightarrow\infty$ and
$$
\int_{\Omega_{\nll,\nuu}}G(g)\Lambda_N(dn)\leq\alpha.
$$
To proceed further, note by the weak convergence of $g_{i_k}$ to $g$ in $L^2(\Lambda_N)$ we have
$$
\int_{\Omega_{\nll,\nuu}}g_{i_k}\phi\Lambda_N(dn)\rightarrow\int_{\Omega_{\nll,\nuu}}g\phi\Lambda_N(dn)
$$
as $k\rightarrow\infty$ for all $\phi\in L^2(\Lambda_N)$. Moreover, since each $g_{i_k}$ is weakly differentiable, we have
$$
\int_{\Omega_{\nll,\nuu}}g_{i_k}\phi'\ell(dn)=-\int_{\Omega_{\nll,\nuu}}\phi Dg_{i_k}\ell(dn)
$$
for all $k\geq 0$ and all $\phi\in\mathcal{C}_0^{\infty}(\Omega_{\nll,\nuu})$. Now, let $\phi\in\mathcal{C}_0^{\infty}(\Omega_{\nll,\nuu})$ and note that, since $\phi'$ is of compact support contained in $\Omega_{\nll,\nuu}$ and $p_N>0$ and continuous in $\Omega_{\nll,\nuu}$, we have
$$
\frac{\phi'}{p_N}\in L^2(\Lambda_N),
$$
and therefore
\begin{equation*}
\int_{\Omega_{\nll,\nuu}}g_{i_k}\phi'\ell(dn)=\int_{\Omega_{\nll,\nuu}}g_{i_k}\frac{\phi'}{p_N}\Lambda_N(dn)
\rightarrow\int_{\Omega_{\nll,\nuu}}g\frac{\phi'}{p_N}\Lambda_N(dn)=\int_{\Omega_{\nll,\nuu}}g\phi'\ell(dn)
\end{equation*}
as $k\rightarrow\infty$. Hence
\begin{equation*}
\int_{\Omega_{\nll,\nuu}}g\phi'\ell(dn)=\lim_{k\rightarrow\infty}\int_{\Omega_{\nll,\nuu}}g_{i_k}\phi'\ell(dn)
=-\lim_{k\rightarrow\infty}\int_{\Omega_{\nll,\nuu}}\phi Dg_{i_k}\ell(dn),
\end{equation*}
from where we conclude, and since H\"{o}lder's inequality shows
$$
\int_{\Omega_{\nll,\nuu}}|g\phi'|\ell(dn)
\leq||g||_{L^2(\Lambda_N)}\left|\left|\frac{\phi'}{p_N}\right|\right|_{L^2(\Lambda_N)}<\infty,
$$
that
$$
\lim_{k\rightarrow\infty}\int_{\Omega_{\nll,\nuu}}\phi(Dg_{i_{k+j}}-Dg_{i_k})\ell(dn)=0
$$
for all $j\geq 0$ and all $\phi\in\mathcal{C}_0^{\infty}(\Omega_{\nll,\nuu})$. But, since $\mathcal{C}_0^{\infty}(\Omega_{\nll,\nuu})$ is dense in $L^2(\Lambda_N)$, the above equation also holds for all $\phi\in L^2(\Lambda_N)$, and therefore in particular for all $\phi\in\mathcal{C}_0^{1}(\Omega_{\nll,\nuu})$, the collection of continuously differentiable functions with compact support in $\Omega_{\nll,\nuu}$. Then, we have
\begin{equation*}
\lim_{k\rightarrow\infty}\int_{\Omega_{\nll,\nuu}}\phi(Dg_{i_{k+j}}-Dg_{i_k})\ell(dn)
=\lim_{k\rightarrow\infty}\int_{\Omega_{\nll,\nuu}}\frac{\phi}{p_N}(Dg_{i_{k+j}}-Dg_{i_k})\Lambda_N(dn)=0
\end{equation*}
for all $\phi\in\mathcal{C}_0^{1}(\Omega_{\nll,\nuu})$, and then, since $\phi p_N$ is again an arbitrary function in $\mathcal{C}_0^{1}(\Omega_{\nll,\nuu})$, that
$$
\lim_{k\rightarrow\infty}\int_{\Omega_{\nll,\nuu}}\phi(Dg_{i_{k+j}}-Dg_{i_k})\Lambda_N(dn)=0
$$
for all $\phi\in\mathcal{C}_0^{1}(\Omega_{\nll,\nuu})$. But, $\mathcal{C}_0^1(\Omega_{\nll,\nuu})$ is also dense in $L^2(\Lambda_N)$, and therefore we have that the last equation holds for all $\phi\in L^2(\Lambda_N)$, that is,
$$
\lim_{k\rightarrow\infty}|<\phi,Dg_{i_{k+j}}>_{L^2(\Lambda_N)}-<\phi,Dg_{i_k}>_{L^2(\Lambda_N)}|=0
$$
for all $j\geq 0$ and all $\phi\in L^2(\Lambda_N)$. Hence, the sequence $\{Dg_{i_k}\}_{k\geq 0}$ is a weak Cauchy sequence in $L^2(\Lambda_N)$, and therefore, since $L^2(\Lambda_N)$ is a Hilbert space, it is also bounded in $L^2(\Lambda_N)$, that is, there exists $B\in (0,\infty)$ such that
$$
||Dg_{i_k}||_{L^2(\Lambda_N)}\leq B
$$
for all $k\geq 0$, which in turn implies the existence of a subsequence, still denoted as $\{Dg_{i_k}\}_{k\geq 0}$, and a $g^*\in L^2(\Lambda_N)$ such that
$$
Dg_{i_k}\arrowwl2 g^*
$$
as $k\rightarrow\infty$.
Now,
$$
\int_{\Omega_{\nll,\nuu}}g_{i_k}\phi'\ell(dn)=-\int_{\Omega_{\nll,\nuu}}\phi Dg_{i_k}\ell(dn),
$$
that is,
$$
\int_{\Omega_{\nll,\nuu}}g_{i_k}\frac{\phi'}{p_N}\Lambda_N(dn)=-\int_{\Omega_{\nll,\nuu}}\frac{\phi}{p_N} Dg_{i_k}\Lambda_N(dn)
$$
for all $\phi\in\mathcal{C}_0^{\infty}(\Omega_{\nll,\nuu})$. But, by the same arguments as for $\phi'/p_N$, for $\phi\in\mathcal{C}_0^{\infty}(\Omega_{\nll,\nuu})$ we also have
$$
\frac{\phi}{p_N}\in L^2(\Lambda_N),
$$
and therefore we may write ($g_{i_k}\arrowwl2 g$)
\begin{equation*}
\int_{\Omega_{\nll,\nuu}}g_{i_k}\frac{\phi'}{p_N}\Lambda_N(dn)\rightarrow
\int_{\Omega_{\nll,\nuu}}g\frac{\phi'}{p_N}\Lambda_N(dn)=\int_{\Omega_{\nll,\nuu}}g\phi'\ell(dn)
\end{equation*}
and ($Dg_{i_k}\arrowwl2 g^*$)
\begin{equation*}
\int_{\Omega_{\nll,\nuu}}\frac{\phi}{p_N}Dg_{i_k}\Lambda_N(dn)\rightarrow
\int_{\Omega_{\nll,\nuu}}\frac{\phi}{p_N}g^*\Lambda_N(dn)=\int_{\Omega_{\nll,\nuu}}\phi g^*\ell(dn),
\end{equation*}
both as $k\rightarrow\infty$, to conclude
$$
\int_{\Omega_{\nll,\nuu}}g\phi'\ell(dn)=-\int_{\Omega_{\nll,\nuu}}\phi g^*\ell(dn),
$$
for all $\phi\in\mathcal{C}_0^{\infty}(\Omega_{\nll,\nuu})$, that is, $g$ is weakly differentiable with, up to a Lebesgue null set of course,
$$
Dg=g^*.
$$
Therefore, we have proved that for every sequence $\{g_i\}_{i\geq 0}\subseteq\mathcal{D}_\alpha$, there exists a corresponding subsequence $\{g_{i_k}\}_{k\geq 0}$ and a $g\in\mathcal{D}_\alpha$, such that
$$
g_{i_k}\arrowwh g
$$
as $k\rightarrow\infty$, that is, $\mathcal{D}_\alpha$ is weakly compact, as claimed. Note then, since $\mathcal{D}_\alpha\subseteq H^1(\Lambda_N)$ is weakly compact and $H^1(\Lambda_N)$ is a Hilbert space, we conclude that $\mathcal{D}_\alpha$ is indeed bounded \footnote{It is then not only bounded but also weakly closed, meaning that every weakly convergent sequence contained in $\mathcal{D}_\alpha$ weakly converges in fact to a point in $\mathcal{D}_\alpha$.}, that is, there exists $L\in(0,\infty)$ such that
$$
||g||_{H^1(\Lambda_N)}\leq L
$$
for all $g\in\mathcal{D}_\alpha$.

To proceed further, we now consider the entropy functional
$$
\mathcal{H}_N[g]=h(N)+\int_{\Omega_{\nll,\nuu}}\ln(1+Dg)\Lambda_N(dn)
$$
over the class of functions $\mathcal{D}_{\alpha}^*$ consisting of all functions $g\in\mathcal{D}_\alpha$ whose weak derivatives $Dg$ satisfy $Dg\geq-1$ $\ell$-almost everywhere in $\Omega_{\nll,\nuu}$ (equivalently, $\Lambda_N$-almost everywhere in $\Omega_{\nll,\nuu}$). Then, since $\ln(1+x)\leq x$ for all $x\geq -1$ (recall we set $\ln x=-\infty$ for $x\leq0$ and $0[\pm\infty]=0$), it is easy to see that
\begin{equation*}
\int_{\Omega_{\nll,\nuu}}\ln(1+Dg)\Lambda_N(dn)\leq\int_{\Omega_{\nll,\nuu}}Dg\Lambda_N(dn)
\leq\int_{\Omega_{\nll,\nuu}}|Dg|\Lambda_N(dn)\leq ||Dg||_{L^2(\Lambda_N)}
\leq||g||_{H^1(\Lambda_N)},
\end{equation*}
where we have used H\"{o}lder's inequality, and therefore that
$$
\int_{\Omega_{\nll,\nuu}}\ln(1+Dg)\Lambda_N(dn)\leq L
$$
for all $g\in\mathcal{D}_\alpha^*$. Hence, $\mathcal{H}_N$ is bounded above in $\mathcal{D}_\alpha^*$. Set
$$
\mathcal{H}_N^*\doteq\sup_{g\in\mathcal{D}_\alpha^*}\mathcal{H}_N[g],
$$
and consider $\{g_i\}_{i\geq 0}\subseteq\mathcal{D}_\alpha^*$ such that
$$
\mathcal{H}_N[g_i]\rightarrow\mathcal{H}_N^*
$$
as $i\rightarrow\infty$. Since $\{g_i\}_{i\geq 0}\subseteq\mathcal{D}_\alpha^*\subseteq\mathcal{D}_\alpha$ and $\mathcal{D}_\alpha$ is weakly compact, we conclude that there exists a corresponding subsequence $\{g_{i_k}\}_{k\geq 0}$ and a $g^*\in\mathcal{D}_\alpha$ such that
$$
g_{i_k}\arrowwh g^*
$$
as $k\rightarrow\infty$. That indeed $g^*\in\mathcal{D}_\alpha^*$ can be seen as follows. Consider the indicator function
$\mathbbm{1}\{n\in\Omega_{\nll,\nuu}: Dg^*(n)\leq-1\}$, denoted for simplicity as $\mathbbm{1}\{Dg^*\leq-1\}$, and note since obviously
$$
\mathbbm{1}\{Dg^*\leq-1\}\in L^2(\Lambda_N),
$$
the weak convergence $Dg_{i_k}\arrowwl2 Dg^*$ shows then that
\begin{equation*}
\int_{\Omega_{\nll,\nuu}}Dg_{i_k}\mathbbm{1}\{Dg^*\leq-1\}\Lambda_N(dn)\rightarrow
\int_{\Omega_{\nll,\nuu}}Dg^*\mathbbm{1}\{Dg^*\leq-1\}\Lambda_N(dn)
\end{equation*}
as $k\rightarrow\infty$. Moreover, since $Dg_{i_k}\geq-1$ $\Lambda_N$-almost everywhere in $\Omega_{\nll,\nuu}$ for each $k\geq 0$, we have
$$
-\mathbbm{1}\{Dg^*\leq-1\}\leq Dg_{i_k}\mathbbm{1}\{Dg^*\leq-1\},
$$
$\Lambda_N$-almost everywhere in $\Omega_{\nll,\nuu}$ and for each $k\geq 0$ as well, and therefore, with
$$
\mu\doteq\Lambda_N(\{n\in\Omega_{\nll,\nuu}: Dg^*(n)\leq-1\}),
$$
we have
\begin{equation*}
-\mu\leq\int_{\Omega_{\nll,\nuu}}Dg_{i_k}\mathbbm{1}\{Dg^*\leq-1\}\Lambda_N(dn)
\rightarrow
\int_{\Omega_{\nll,\nuu}}Dg^*\mathbbm{1}\{Dg^*\leq-1\}\Lambda_N(dn)\leq-\mu,
\end{equation*}
that is
$$
\int_{\Omega_{\nll,\nuu}}Dg^*\mathbbm{1}\{Dg^*\leq-1\}\Lambda_N(dn)=-\mu,
$$
and hence in fact $Dg^*\geq -1$ $\Lambda_N$-almost everywhere in $\Omega_{\nll,\nuu}$. Thus, we conclude $g^*\in\mathcal{D}_\alpha^*$. Moreover, from the concavity of $\ln$, and therefore of $\mathcal{H}_N$, we have that $\mathcal{H}_N$ is indeed weakly upper semi-continuous, that is,
$$
\mathcal{H}_N[f]\geq\limsup_{i\rightarrow\infty}\mathcal{H}_N[f_i]
$$
for all $\{f_i\}\subseteq\mathcal{D}_\alpha^*$ converging weakly in $H^1(\Lambda_N)$ to $f\in\mathcal{D}_\alpha^*$, and therefore we may write
$$
\mathcal{H}_N^*\geq\mathcal{H}_N[g^*]\geq\limsup_{k\rightarrow\infty}\mathcal{H}_N[g_{i_k}]=\mathcal{H}_N^*,
$$
from where we conclude that the supremum is indeed attainable in $\mathcal{D}_\alpha^*$, that is, it is a true maximum,
$$
\sup_{g\in\mathcal{D}_\alpha^*}\mathcal{H}_N[g]=\max_{g\in\mathcal{D}_\alpha^*}\mathcal{H}_N[g]=\mathcal{H}_N[g^*].
$$
Now, since
$$
g^*\in W^{(1,p)}(\Omega_{\nll,\nuu},\Lambda_N)\doteq\left\{g\in L^p(\Lambda_N):Dg\in L^p(\Lambda_N)\right\}
$$
with $p=2$ and $\Omega_{\nll,\nuu}\subseteq\R^n$ with $n=1$, and therefore $n<p$, we conclude that $g^*$ is indeed differentiable $\ell$-almost everywhere in $\Omega_{\nll,\nuu}$, and that its derivative $(g^*)'$ satisfies
$$
Dg^*=(g^*)'
$$
$\ell$-almost everywhere in $\Omega_{\nll,\nuu}$ as well. Therefore we may rewrite the optimum value of the entropy functional as
$$
\mathcal{H}_N[g^*]=h(N)+\int_{\Omega_{\nll,\nuu}}\ln(1+(g^*)')\Lambda_N(dn),
$$
the arguments leading to Theorem \ref{thm1} in Section \ref{main} then showing that in fact the optimum function $g^*\in\C^2(\Omega_{\nll,\nuu})$.

Now, as for the claimed capacity achievability either when $p_N$ is nonconstant and of compact support ($|\nll|+|\nuu|<\infty$), when $p_N$ is non-increasing and of semi-infinite support ($|\nll|<\infty$, $\nuu=\infty$), or when $p_N$ is non-decreasing and of semi-infinite support as well ($\nll=-\infty$, $|\nuu|<\infty$), note it follows directly from the arguments leading to Corollary \ref{cor1} in Section \ref{main} and Lemma \ref{lemma2}.

Finally we prove the last assertion of the theorem, regarding the approximation
$$
C(N;G,\alpha)=\lim_{k\rightarrow\infty}C(N_k;G,\alpha)
$$
when achievability is not a priori guaranteed. For that purpose, consider $N\in\N(\nll,\nuu)$ when either $|\nll|+|\nuu|=\infty$ or $p_N$ constant in $(\nll,\nuu)$, and the corresponding sequence of approximating noise RVs $\{N_k\}_{k\geq 0}$, as introduced in Lemma \ref{lemma1}. We have already proved that for each $\alpha\in \I_G$ the optimums exist
\begin{equation*}
\sup\{h(Y):Y\in\Sigma(N;G,\alpha)\}=\max\{h(Y):Y\in\Sigma(N;G,\alpha)\}=h(Y_\alpha^\dag)
\end{equation*}
and
\begin{equation*}
\sup\{h(Y):Y\in\Sigma(N_k;G,\alpha)\}=\max\{h(Y):Y\in\Sigma(N_k;G,\alpha)\}=h(Y_{k,\alpha}^\dag),\;k\geq 0,
\end{equation*}
where we may write
$$
Y_\alpha^\dag=X_\alpha^\dag+N\;\text{with }X_\alpha^\dag=g_\alpha(N)
$$
and, for $k\geq 0$,
$$
Y_{k,\alpha}^\dag=X_{k,\alpha}^\dag+N_k\;\text{with }X_{k,\alpha}^\dag=g_{k,\alpha}(N_k).
$$
We know for each $\alpha\in \I_G$ there exists a sequence of RVs $\{X_{N_k}^\alpha\}_{k\geq 0}$ with each $X_{N_k}^\alpha$ independent of $N_k$ and such that
$$
Y_{k,\alpha}^\dag\eqd X_{N_k}^\alpha+N_k,
$$
and where the law of each $X_{N_k}^\alpha$ is uniquely determined by considering the Fourier transform quotient
$$
\frac{\F[\Lambda_{Y_{k,\alpha}^\dag}]}{\F[\Lambda_{N_k}]}(f),
$$
$f\in\R$, and taking the corresponding inverse Fourier transform. Now, for each $\alpha\in \I_G$ as well, we consider the sequence of RVs $\{X_N^{k,\alpha}\}_{k\geq 0}$ with each $X_N^{k,\alpha}$ independent of $N$ and such that
$$
X_N^{k,\alpha}\eqd X_{N_k}^\alpha.
$$
We set
$$
Y_N^{k,\alpha}\doteq X_N^{k,\alpha}+N.
$$
From the arguments leading to Corollary \ref{cor1}, we know for each $\alpha\in \I_G$ and $k\geq 0$ there exists a unique $\alpha^{\ddag}(k,\alpha)$, $\alpha^{\ddag}(k,\alpha)<\alpha$, such that
$$
\E G(X_{N_k}^{\alpha^\ddag(k,\alpha)})=\alpha
$$
and
$$
C(N_k;G,\alpha)=h(Y_{k,\alpha^\ddag(k,\alpha)}^\dag)-h(N_k)
$$
that is, with the law of each $X_{N_k}^{\alpha^\ddag(k,\alpha)}$ being a corresponding capacity-achieving input distribution. (To be coherent with the statement of the theorem, we could now set
$$
X_{k,\alpha}^\ddag\doteq X_{N_k}^{\alpha^\ddag(k,\alpha)}
$$
and
$$
Y_{k,\alpha}^\ddag\doteq Y_{k,\alpha^\ddag(k,\alpha)}^\dag
$$
for $k\geq 0$ and $\alpha\in \I_G$.) Now, from Lemma \ref{lemma1} we know we have the convergence
$$
|h(Y_N^{k,\alpha^\ddag(k,\alpha)})-h(Y_{k,\alpha^\ddag(k,\alpha)}^\dag)|\rightarrow 0
$$
as $k\rightarrow\infty$ for each $\alpha\in \I_G$. From Lemma \ref{lemma1} we also know
$$
|h(Y_{k,\alpha}^\dag)-h(Y_\alpha^\dag)|\rightarrow 0
$$
with
$$
Y_{k,\alpha}^\dag\arrowd Y_\alpha^\dag
$$
as $k\rightarrow\infty$ for each $\alpha\in \I_G$, and therefore we have that, for each $\alpha\in \I_G$ as well, there exists a unique $\alpha^\ddag(\alpha)\in \I_G$, now with $\alpha^\ddag(\alpha)\leq\alpha$, and such that
$$
\alpha^\ddag(k,\alpha)\rightarrow \alpha^\ddag(\alpha)
$$
as $k\rightarrow\infty$. By considering for each $\alpha\in \I_G$ a compact interval $I_\alpha$ with $\alpha^\ddag(\alpha)\in I_\alpha\subseteq \I_G$, the uniform convergence stated in Lemma \ref{lemma1} then shows that
$$
h(Y_{k,\alpha^\ddag(k,\alpha)}^\dag)\rightarrow h(Y^\dag_{\alpha^\ddag(\alpha)})
$$
and therefore
$$
h(Y_N^{k,\alpha^\ddag(k,\alpha)})\rightarrow h(Y^\dag_{\alpha^\ddag(\alpha)})
$$
too, both as $k\rightarrow\infty$. This last convergence shows that
$$
\sup\{h(Y):Y\in\Sigma_\bot(N;G,\alpha)\}\geq h(Y^\dag_{\alpha^\ddag(\alpha)}).
$$
Let us assume that
$$
\sup\{h(Y):Y\in\Sigma_\bot(N;G,\alpha)\}> h(Y^\dag_{\alpha^\ddag(\alpha)}).
$$
Then, there exists $Y_\alpha^*\in\Sigma_\bot(N;G,\alpha)$, say $Y_\alpha^*=X_\alpha^*+N$, such that
$$
h(Y_\alpha^*)>h(Y^\dag_{\alpha^\ddag(\alpha)}).
$$
Consider then the sequence $\{X_{k,\alpha}^*\}_{k\geq 0}$ with each $X_{k,\alpha}^*$ independent of $N_k$ and
$$
X_{k,\alpha}^*\eqd X_\alpha^*.
$$
For each $k\geq 0$ set
$$
Y_{k,\alpha}^*=X_{k,\alpha}^*+N_k.
$$
The same arguments as in the proof of Lemma \ref{lemma1} show that, since $Y_\alpha^*\in\Sigma_\bot(N;G,\alpha)$, each $Y_{k,\alpha}^*\in\Sigma_\bot(N_k;G,\alpha)$, and therefore Lemma \ref{lemma1} itself that
$$
h(Y_{k,\alpha}^*)\rightarrow h(Y_\alpha^*)
$$
as $k\rightarrow\infty$.
Hence
$$
\lim_{k\rightarrow\infty}h(Y_{k,\alpha}^*)=h(Y_\alpha^*)>h(Y^\dag_{\alpha^\ddag(\alpha)})
=\lim_{k\rightarrow\infty}h(Y_{k,\alpha^\ddag(k,\alpha)}^\dag),
$$
a contradiction since then there exists $k_0\geq 0$ such that, for all $k\geq k_0$ (for just one such $k$ is indeed enough),
$$
h(Y_{k,\alpha}^*)>h(Y_{k,\alpha^\ddag(k,\alpha)}^\dag),
$$
but
$$
h(Y_{k,\alpha^\ddag(k,\alpha)}^\dag)=\sup\{h(Y):Y\in\Sigma_\bot(N_k;G,\alpha)\}
$$
and
$$
Y_{k,\alpha}^*\in\Sigma_\bot(N_k;G,\alpha).
$$
Thus
$$
\sup\{h(Y):Y\in\Sigma_\bot(N;G,\alpha)\}=h(Y^\dag_{\alpha^\ddag(\alpha)}),
$$
and therefore we conclude
\begin{eqnarray*}
C(N;G,\alpha)&=&\sup\{h(Y):Y\in\Sigma_\bot(N;G,\alpha)\}-h(N)\\
&=&h(Y^\dag_{\alpha^\ddag(\alpha)})-h(N)\\
&=&\lim_{k\rightarrow\infty}h(Y_{k,\alpha^\ddag(k,\alpha)}^\dag)-\lim_{k\rightarrow\infty}h(N_k)\\
&=&\lim_{k\rightarrow\infty}\left[h(Y_{k,\alpha^\ddag(k,\alpha)}^\dag)-h(N_k)\right]\\
&=&\lim_{k\rightarrow\infty}C(N_k;G,\alpha),
\end{eqnarray*}
that is,
$$
C(N;G,\alpha)=\lim_{k\rightarrow\infty}C(N_k;G,\alpha).
$$
Finally we note that the above arguments show that the existence of $\alpha^\ddag(\alpha)$ for each $\alpha\in \I_G$ and such that
$$
\sup\{h(Y):Y\in\Sigma_\bot(N;G,\alpha)\}=h(Y^\dag_{\alpha^\ddag(\alpha)}),
$$
that is, such that
$$
C(N;G,\alpha)=\lim_{k\rightarrow\infty}C(N_k;G,\alpha)=h(Y^\dag_{\alpha^\ddag(\alpha)})-h(N),
$$
does not depend on the existence of a positive definite continuous extension of the Fourier transform quotient
$$
\X_{\alpha^\ddag(\alpha)}(f)=\frac{\F[\Lambda_{Y^\dag_{\alpha^\ddag(\alpha)}}]}{\F[\Lambda_N]}(f)
$$
from $S_N$ to the whole of $\R$, even if such an extension does not exist, which, by Bochner's theorem, implies there does not exist a probability measure over $\R$ representing an independent of $N$ RV, satisfying the constraint and generating the distribution of $Y^\dag_{\alpha^\ddag(\alpha)}$.
The theorem is thus proved.
\end{proof}

\section{Some Remarks on Feedback Capacity}
\label{generalcapacity}

In this last section of the paper we briefly describe some applications of our results in the context of feedback capacity. To that end, we consider an additive channel with noise process being an identical distributed sequence of RVs. Specifically, let the noise process $N$ be the sequence $\{N_i\}_{i\geq1}$ with each $N_i\in\N(\nll,\nuu)$ and
$$
N_i\eqd N_1
$$
for all $i\geq 2$, the channel model being
$$
Y_i=X_i+N_i,
$$
$i\geq 1$. Note the noise process $N$ is just assumed to be a sequence of identically distributed RVs, that is, no independence nor strong stationarity is explicitly demanded. We do assume however not only for the corresponding joint differential entropy
$$
h(N_1,\cdots,N_k)
$$
to be well defined in $\R$ for every $k\geq1$, but also for the differential entropy rate
$$
h_{\infty}(N)\doteq\lim_{k\rightarrow\infty}\frac{1}{k}h(N_1,\cdots,N_k)
$$
to exists (as a number in $\R$).

As usual, we will denote by the superscript $k$ the corresponding collection of RVs
$$
N^k=(N_1,\cdots,N_k),
$$
$$
X^k=(X_1,\cdots,X_k)
$$
and
$$
Y^k=(Y_1,\cdots,Y_k),
$$
and the notation used so far is generalized in a straightforward way to write
$$
Y^k=\Sigma(X^k,N^k;G,\beta),
$$
$$
Y^k=\Sigma_{\bot}(X^k,N^k;G,\beta),
$$
$$
\Sigma(N^k;G,\beta)
$$
and
$$
\Sigma_{\bot}(N^k;G,\beta),
$$
where the restriction is understood to be componentwise
$$
\E G(X_i)\leq\beta,
$$
$1\leq i\leq k$, and where, in the independent case ($\Sigma_{\bot}$), independence is demanded between the collections $(X_1,\cdots,X_k)$ and $(N_1,\cdots,N_k)$. Moreover, since we are interested in a feedback setting, and because of causality, we demand in the dependent case ($\Sigma$) for each corresponding $X_i$ to be of the form
$$
X_i=F_i(Y^{i-1},V_i)
$$
($X_1=V_1$) with each $F_i:\R^{i}\rightarrow\R$ Borel-measurable and the process $V$ independent of the noise process $N$. It is important to emphasize that relationships as above take care of causality in time, however they of course do not rule out the stochastic dependence that generally will exist between each pair $X_i$ and $N_i$ (for $i\geq 2$) when the noise process $N$ is not white but colored.

Thinking as usual of the channel input block $X^k$ as conveying a message random variable $W_k$ independent of $N$, we set capacity as the corresponding limit (when it does exist) of the extremal mutual information $I(W_k;Y^k)$,
$$
\lim_{k\rightarrow\infty}\frac{1}{k}\sup \{I(W_k;Y^k)\}=\lim_{k\rightarrow\infty}\frac{1}{k}\sup\{h(Y^k)-h(N^k)\},
$$
with the supremum being taken over the class of all admissible inputs. (Note as pointed out in \cite{CP1989}, in the context of a time-varying additive (colored) Gaussian noise channel with feedback, the corresponding limit above may be avoided by thinking of capacity in bits per transmission if the channel is to be used for the time block $\{1,\cdots,k\}$. However, since the goal of this section is just to show how our results can be generally applied in the context of feedback capacity, we just keep the limit as part of the definitions that follow.) Specifically, we set capacity without feedback, or forward capacity, as
\begin{equation}\nonumber
C(N;G,\beta)\doteq\lim_{k\rightarrow\infty}\frac{1}{k}\left[\sup\{h(Y^k):Y^k\in\Sigma_{\bot}(N^k;G,\beta)\}-h(N^k)\right]
\end{equation}
whenever the above limit exists (in $[0,\infty]$), and capacity with feedback, or backward capacity, as
\begin{equation}\nonumber
C_{FB}(N;G,\beta)\doteq\lim_{k\rightarrow\infty}\frac{1}{k}\left[\sup\{h(Y^k):Y^k\in\Sigma(N^k;G,\beta)\}-h(N^k)\right]
\end{equation}
whenever the above limit exists (in $[0,\infty]$) as well.

Being $N$ and $G$ usually clear from the context, from now on we just correspondingly write $C(\beta)$ and $C_{FB}(\beta)$.

To proceed further, note from the chain rule for differential entropies we get
$$
h(Y^k)=h(Y_1,\cdots,Y_k)\leq h(Y_1)+\cdots+h(Y_k)
$$
for any $Y^k$ belonging to $\Sigma(N^k;G,\beta)$ or $\Sigma_{\bot}(N^k;G,\beta)$, and therefore, assuming the individual optimum entropies exist (e.g., under the hypotheses of Theorem \ref{thm3} in previous section) and by noting that the $\{N_i\}_{i\geq 1}$ are all identically distributed, we may write the following upper bounds for $C(\beta)$ and $C_{FB}(\beta)$ under the noise process $N$
$$
C(\beta)\leq h(Y^\ddag_{\beta})-h_{\infty}(N)
$$
and
$$
C_{FB}(\beta)\leq h(Y^\dag_{\beta})-h_{\infty}(N).
$$
In what follows we set
$$
C^*(\beta)\doteq h(Y^\ddag_{\beta})-h_{\infty}(N)
$$
and
$$
C^*_{FB}(\beta)\doteq h(Y^\dag_{\beta})-h_{\infty}(N)
$$
and briefly explore on the relationship between both, which, though strictly speaking not precise of course (being both just upper bounds), provides some theoretical insight on the potential gain in capacity by the use of feedback.

We begin by noting that
\begin{eqnarray*}
C^*_{FB}(\beta)-C^*(\beta)&=&h(Y^\dag_{\beta})-h(Y^\ddag_{\beta})\\&=&h(Y^\dag_{\beta})-h(Y^\dag_{\beta^\ddag{(\beta)}})
\end{eqnarray*}
where the above difference is of course non-negative ($\beta\geq\beta^\ddag{(\beta)}$), and in fact strictly positive when the optimum entropy
$$
h(Y^\ddag_{\beta})=h(Y^\dag_{\beta^\ddag{(\beta)}})
$$
is achievable ($\beta>\beta^\ddag{(\beta)}$ in that case). It is from this point of view that, as mentioned in Remark \ref{entropygainwithfeedback} in Section \ref{main2}, the difference
$$
h(Y^\dag_{\beta})-h(Y^\ddag_{\beta})=h(Y^\dag_{\beta})-h(Y^\dag_{\beta^\ddag{(\beta)}})
$$
provides an estimate for the gain of entropy at the output, and in turn in capacity, by the use of feedback. Note also for any $\beta\in\I_G$ we have
$$
C^*_{FB}(\beta^\ddag{(\beta)})=C^*(\beta),
$$
and therefore the difference
$$
\beta-\beta^\ddag{(\beta)}
$$
in turn provides an estimate too but of the extra average cost required to compensate for the absence of feedback.

We now briefly review the Gaussian noise case
$$
N_1\sim\mathcal{N}(0,\sigma^2)
$$
(w.l.o.g. we take zero mean, and leave the correlation structure of process $N$ open), with cost function
$$
G(x)=x^2
$$
(power constraint). Then, for any $\beta>0$ we have
$$
h(Y^\dag_{\beta})=\frac{1}{2}\ln 2\pi e(\sigma+\sqrt{\beta})^2,
$$
$$
h(Y^\ddag_{\beta})=h(Y^\dag_{\beta^\ddag{(\beta)}})=\frac{1}{2}\ln 2\pi e(\sigma^2+\beta),
$$
$$
C^*(\beta)=h(Y^\ddag_{\beta})-h_{\infty}(N)=\frac{1}{2}\ln 2\pi e(\sigma^2+\beta)-h_{\infty}(N)
$$
and
$$
C^*_{FB}(\beta)=h(Y^\dag_{\beta})-h_{\infty}(N)=\frac{1}{2}\ln 2\pi e(\sigma+\sqrt{\beta})^2-h_{\infty}(N).
$$
Moreover, in this case it is easy to see that the associated $\beta^\ddag{(\beta)}$ is determined by
$$
\beta=\beta^\ddag{(\beta)}+2\sigma\sqrt{\beta^\ddag{(\beta)}},
$$
and therefore
$$
\beta^\ddag{(\beta)}=\left(\sqrt{\sigma^2+\beta}-\sigma\right)^2.
$$
Hence, $\beta-\beta^\ddag{(\beta)}$, the estimate of the extra average cost required to compensate for the absence of feedback, in this case becomes
$$
\beta-\left(\sqrt{\sigma^2+\beta}-\sigma\right)^2=2\sigma\left(\sqrt{\sigma^2+\beta}-\sigma\right).
$$

By using that $(\sigma+\sqrt{\beta})^2\leq 2(\sigma^2+\beta)$ and considering logarithms in the definition of differential entropies from now on to the base of $2$ instead of $e$, we have
\begin{equation*}
h(Y^\dag_{\beta})-h(Y^\ddag_{\beta})=\frac{1}{2}\log_2 2\pi e(\sigma+\sqrt{\beta})^2-\frac{1}{2}\log_2 2\pi e(\sigma^2+\beta)
\leq\frac{1}{2}\log_2 4\pi e(\sigma^2+\beta)-\frac{1}{2}\log_2 2\pi e(\sigma^2+\beta),
\end{equation*}
that is
$$
h(Y^\dag_{\beta})-h(Y^\ddag_{\beta})\leq\frac{1}{2}\log_2 2=\frac{1}{2},
$$
and therefore we conclude
$$
C^*(\beta)<C^*_{FB}(\beta)\leq C^*(\beta)+\frac{1}{2}\;\;\text{(bits per transmission)},
$$
which resembles the well known result for Gaussian feedback capacity,
\begin{equation}\label{fc1}
C(\beta)<C_{FB}(\beta)\leq C(\beta)+\frac{1}{2}\;\;\text{(bits per transmission)},
\end{equation}
stated in \cite{CP1989} for a time-varying additive Gaussian noise channel with feedback.

Equation (\ref{fc1}) is known to have a refinement, namely
$$
C_{FB}(\beta)\leq C(\gamma\beta)+\frac{1}{2}\log_2\left(1+\frac{1}{\gamma}\right)\;\;\text{(bits per transmission)},
$$
established in \cite{CY1999} for Gaussian feedback capacity and any $\gamma>0$.

Considering $\gamma>0$, in our context we get
\begin{eqnarray*}
C^*_{FB}(\beta)-C^*(\gamma\beta)&=&h(Y_\beta^\dag)-h(Y_{\gamma\beta}^\ddag)\\
&=&\frac{1}{2}\log_2\frac{(\sigma+\sqrt{\beta})^2}{\sigma^2+\gamma\beta}\\
&=&\frac{1}{2}\log_2\left(1+\frac{2\sigma\sqrt{\beta}+\beta-\gamma\beta}{\sigma^2+\gamma\beta}\right).
\end{eqnarray*}
However, since obviously
$$
(\sigma-\gamma\sqrt{\beta})^2\geq 0,
$$
we conclude
$$
\frac{2\sigma\sqrt{\beta}+\beta-\gamma\beta}{\sigma^2+\gamma\beta}\leq\frac{1}{\gamma}
$$
for all $\gamma>0$, and therefore
$$
C^*_{FB}(\beta)\leq C^*(\gamma\beta)+\frac{1}{2}\log_2\left(1+\frac{1}{\gamma}\right)\;\;\text{(bits per transmission)}
$$
for all $\gamma>0$ as well.

Along the previous lines, and in the context of the above additive upper bounds, there is also another well known result for Gaussian feedback capacity, stated in \cite{P1969,E1970}, and establishing the multiplicative upper bound
\begin{equation}\label{fc2}
C_{FB}(\beta)\leq2C(\beta).
\end{equation}

Being a multiplicative relationship, we expect for the bounds $C^*_{FB}(\beta)$ and $C^*(\beta)$ not to fully comply with equation (\ref{fc2}), basically because in $C^*_{FB}(\beta)$,
$$
C^*_{FB}(\beta)=\frac{1}{2}\ln 2\pi e(\sigma+\sqrt{\beta})^2-h_{\infty}(N),
$$
the first term on the right, appearing from the bounding
$$
\frac{\sup\{h(Y^k):Y^k\in\Sigma(N^k;G,\beta)\}}{k}\leq\frac{1}{2}\ln 2\pi e(\sigma+\sqrt{\beta})^2,
$$
does not take into account the correlation structure present in the noise process $N$, and therefore, as $N$ is considered as a whiter and whiter noise process, this component in $C^*_{FB}(\beta)$ remains however unchanged (unlike $C^*(\beta)$, which converges then to the forward capacity $C(\beta)$, this last one being also the limit in this case of the feedback capacity $C_{FB}(\beta)$).

Indeed, since $h_\infty(N)$ does not depend on $\beta$, it is easy to see that
$$
C^*_{FB}(\beta)\leq2C^*(\beta)
$$
if, and only if,
$$
h_\infty(N)\leq\frac{1}{2}\log_2 2\pi e\left[2(2-\sqrt{2})^2\right]\sigma^2,
$$
and therefore enough correlation structure is required on the noise process $N$ so as to lower $h_\infty(N)$ from the white case
$$
h_\infty(N)=h(N_1)=\frac{1}{2}\log_2 2\pi e\sigma^2
$$
to
$$
h_\infty(N)\leq\frac{1}{2}\log_2 2\pi e\left[2(2-\sqrt{2})^2\right]\sigma^2.
$$
Equivalently, we require for $\sigma_\infty^2$, the variance of the error in the best estimate of $N_k$ given the infinite past, to satisfy
$$
\sigma_\infty^2\leq\left[2(2-\sqrt{2})^2\right]\sigma^2,
$$
as
$$
h_\infty(N)=\frac{1}{2}\log_2 2\pi e\sigma_\infty^2.
$$
(See for example \cite{TC1991}.)

It is important to point out however that a multiplicative relationship between $C^*_{FB}(\beta)$ and $C^*(\beta)$, such as
$$
C^*_{FB}(\beta)\leq2C^*(\beta),
$$
is of course of relevance in the context of establishing by how much capacity can be improved by the use of feedback, which in turn requires enough structure on the noise to be taken advantage of by the feedback.

Finally, equation (\ref{fc2}) is also known to have a refinement, namely
$$
C_{FB}(\beta)\leq\left(1+\frac{1}{\gamma}\right)C(\gamma\beta),
$$
established for Gaussian feedback capacity in \cite{CY1999} too for any $\gamma>0$ as well.

The same comments as before being valid, in this case we find
$$
C^*_{FB}(\beta)\leq\left(1+\frac{1}{\gamma}\right)C^*(\gamma\beta)
$$
for all $\gamma>0$ if, and only if, $h_\infty(N)$ is lowered from the white case
$$
h_\infty(N)=h(N_1)=\frac{1}{2}\log_2 2\pi e\sigma^2
$$
to
$$
h_\infty(N)\leq\inf_{\beta>0}\inf_{\gamma>0}\frac{1}{2}\log_22\pi e\frac{(\sigma^2+\gamma\beta)^{\gamma+1}}{(\sigma+\sqrt{\beta})^{2\gamma}}.
$$

\ifCLASSOPTIONcaptionsoff
  \newpage
\fi

\bibliographystyle{IEEEtran}
\bibliography{IEEEfull,mybib}

\end{document}